\documentclass[11pt, psamsfonts, reqno]{amsart}
\usepackage[T1]{fontenc}
\usepackage[latin1]{inputenc}
\usepackage[frenchb, english]{babel}
\usepackage{amsthm}
\usepackage{amssymb,amsmath}
\usepackage{enumitem}
\usepackage[all]{xy}
\usepackage{mathrsfs}
\usepackage[left=0.0cm,right=0.0cm,top=1.7cm,bottom=1.7cm]{geometry}
\usepackage{hyperref}
\usepackage{verbatim}

\theoremstyle{plain}
\newtheorem{theorem}{Théorème}[section]

\newtheorem{lemme}[theorem]{Lemme}
\newtheorem{proposition}[theorem]{Proposition}
\newtheorem{prop}[theorem]{Proposition}
\newtheorem{corollary}[theorem]{Corollaire}
\newtheorem{cor}[theorem]{Corollaire}

\theoremstyle{definition}
\newtheorem{definition}{Definition}

\newtheorem{conjecture}[theorem]{Conjecture}

\theoremstyle{remark}
\newtheorem{remarque}[theorem]{Remarque}

\newcommand*{\defeq}{\mathrel{\vcenter{\baselineskip0.5ex \lineskiplimit0pt
                     \hbox{\scriptsize.}\hbox{\scriptsize.}}}%
                     =}
                     
\def\Q{{\bf Q}}
\def\Z{{\bf Z}}
\def\C{{\bf C}}
\def\N{{\bf N}}
\def\R{{\bf R}}
\def\F{{\bf F}}
\def\PP{{\bf P}}

\def\H{{H}}

\def\zp{{\Z_p}}
\def\zpe{{\mathbf{Z}_p^\times}}
\def\qpe{{\mathbf{Q}_p^\times}}
\def\cpe{{\mathbf{C}_p^\times}}
\def\cp{{\C_p}}
\def\qp{{\Q_p}}

\def\D{{\bf D}}

\def\Nrig{{\mathbf{N}_{\mathrm{rig}}}}
\def\DdR{{\D_{\mathrm{dR}}}}
\def\Dcris{{\D_{\mathrm{cris}}}}
\def\Ddif{{\D_{\mathrm{dif}}}}
\def\Ddifp{{\D_{\mathrm{dif}}^+}}
\def\Drig{{\D_{\mathrm{rig}}}}
\def\Robba{\mathscr{R}}
\def\E{\mathscr{E}}

\def\BdR{{\mathbf{B}_{\mathrm{dR}}}}

\def\Ebf{\mathbf{E}}

\def\j{j}
\def\epsilon{\varepsilon}
\def\det{\mathrm{det}}

\def\matrice#1#2#3#4{{\big(\begin{smallmatrix}#1&#2\\ #3&#4\end{smallmatrix}\big)}}

\title{$(\varphi, \Gamma)$-modules de de Rham et fonctions $L$ $p$-adiques}
\author{Joaqu\'in Rodrigues Jacinto}

\makeatletter
\def\@tocline#1#2#3#4#5#6#7{\relax
  \ifnum #1>\c@tocdepth 
  \else
    \par \addpenalty\@secpenalty\addvspace{#2}%
    \begingroup \hyphenpenalty\@M
    \@ifempty{#4}{%
      \@tempdima\csname r@tocindent\number#1\endcsname\relax
    }{%
      \@tempdima#4\relax
    }%
    \parindent\z@ \leftskip#3\relax \advance\leftskip\@tempdima\relax
    \rightskip\@pnumwidth plus4em \parfillskip-\@pnumwidth
    #5\leavevmode\hskip-\@tempdima
      \ifcase #1
       \or\or \hskip 1em \or \hskip 2em \else \hskip 3em \fi%
      #6\nobreak\relax
    \dotfill\hbox to\@pnumwidth{\@tocpagenum{#7}}\par
    \nobreak
    \endgroup
  \fi}
\makeatother
\usepackage{hyperref}

\begin{document}

\maketitle

\selectlanguage{french}

\begin{abstract} On développe une variante des méthodes de Coleman et Perrin-Riou permettant, pour une représentation galoisienne de de Rham, construire des fonctions $L$ $p$-adiques à partir d'un système compatible d'éléments globaux. On obtient de la sorte des fonctions analytiques sur un ouvert de l'espace des poids contenant les caractères localement algébriques de conducteur assez grand. Appliqué au système d'Euler de Kato, cela fournit des fonctions $L$ $p$-adiques pour les courbes elliptiques à mauvaise réduction additive et, plus généralement, pour les formes modulaires supercuspidales en $p$. En dimension $2$, nous prouvons une équation fonctionnelle pour nos fonctions $L$ $p$-adiques.
\end{abstract}

\selectlanguage{english}
\begin{abstract} We develop a variant of Coleman and Perrin Riou's methods giving, for a de Rham $p$-adic Galois representation, a construction of $p$-adic  $L$ functions from a compatible system of global elements. As a result, we construct analytic functions on an open set of the $p$-adic weight space containing all locally algebraic characters of large enough conductor. Applied to Kato's Euler system, this gives $p$-adic $L$-functions for elliptic curves with additive bad reduction and, more generally, for modular forms which are supercuspidal at $p$.  In the case of dimension $2$, we prove a functional equation for our $p$-adic $L$-functions. \end{abstract}

\selectlanguage{french}
\setcounter{tocdepth}{3}
\tableofcontents

\section*{Introduction}

Cette article est consacré à l'étude des fonctions $L$ $p$-adiques associées aux formes modulaires. En utilisant la théorie des $(\varphi, \Gamma)$-modules et en généralisant certains résultats de Perrin-Riou, on montre comment construire des fonctions $L$ $p$-adiques associées à une représentation $p$-adique du groupe de Galois absolu $\mathscr{G}_\qp$ de $\qp$ (possiblement à mauvaise réduction en $p$) munie d'un système compatible de classes de cohomologie. En particulier, ceci fournit des fonctions $L$ $p$-adiques pour une forme modulaire supercuspidale en $p$ tordue par des caractères suffisamment ramifiés.

Soit \[ f = \sum_{n = 1}^{+\infty} a_n q^n \in S_k(\Gamma_1(N), \omega_f) \otimes \C \] une forme primitive (i.e cuspidale, nouvelle, propre pour les opérateurs de Hecke et normalisée) de poids $k \geq 2$, niveau $N$ et caractère $\omega_f: (\Z / N\Z)^\times \to \C^\times$. Pour $\eta$ un caractère de Dirichlet, notons \[ L(f, \eta, s) = \sum_{n = 1}^{+\infty} a_n \eta(n) n^{-s} \] la fonction $L$ complexe associée à $f$ et $\eta$. La série définissant $L(f, \eta, s)$ converge pour $\mathrm{Re} \; s > 1$, admet un prolongement analytique à tout le plan complexe et elle satisfait une équation fonctionnelle reliant les valeurs $L(f, \eta, s)$ et $L(\check{f}, \eta^{-1}, k - s)$, où $\check{f} = \sum_{n \geq 1} \overline{a}_n q^n \in S_k(\Gamma_1(N), \omega_f^{-1})$ est la forme conjuguée à $f$. La théorie des symboles modulaires permet de montrer l'existence des périodes complexes $\Omega^+_f$ et $\Omega^-_f$ telles que, si $\eta$ est un caractère de Dirichlet, $j$ est un entier tel que $1 \leq j \leq k - 1$ et $\pm$ est tel que $\eta(-1) (-1)^j = \pm 1$, alors
\[ \frac{\Gamma(j)}{(2 i \pi)^j} \frac{L(f, \eta, j)}{\Omega_f^\pm} \in \overline{\Q} . \]
Ceci permet, en fixant une immersion $\overline{\Q} \subseteq \overline{\Q}_p$ et en notant $\Lambda_\infty(f, \eta, s) = \frac{\Gamma(s)}{(2 i \pi)^s} L(f, \eta, s)$, de voir ces valeurs dans le monde $p$-adique en posant \[ \iota_p(\Lambda_\infty(f, \eta, j)) = \frac{\Lambda_\infty(f, \eta, j)}{\Omega^\pm_f} \in \overline{\Q}_p. \]

Les fonctions $L$ $p$-adiques peuvent être vues naturellement comme des fonctions rigides analytiques sur l'espace des poids $p$-adiques $\mathfrak{X}$ \footnote{L'espace $\mathfrak{X}$ est un espace analytique rigide dont les $L$-points, pour $L$ une extension finie de $\qp$, sont donnés par $\mathfrak{X}(L) = \mathrm{Hom}_{\mathrm{cont}}(\zpe, L)$. Il est une union de boules ouvertes et donc quasi-Stein. Par un théorème d'Amice, les distributions sur $\zpe$ correspondent aux fonctions (rigides) analytiques sur $\mathfrak{X}$.}. Si $\eta: \zpe \to \cpe$ est un caractère d'ordre fini, une telle fonction $L_p \in \mathcal{O}(\mathfrak{X})$ est déterminée par ses valeurs sur les caractères de la forme $x \mapsto \eta(x)x^k$, $k \in \Z$. En particulier, si l'on pose, pour $s \in \zp$ et $x \in \zpe$, $\langle x \rangle^s = \exp(s \log x)$, une fonction analytique sur l'espace des poids donne naissance à une famille de fonctions $L_p(\eta, s)$ d'une variable $p$-adique $s \in \zp$, pour $\eta$ parcourant les caractères d'ordre fini.

Fixons un isomorphisme $\overline{\mathbf{Q}}_p \cong \C$ et notons \[ X^2 - a_p X + \omega_f(p) p^{k - 1} \] le polynôme de Hecke en $p$ de la forme $f$ et $\alpha, \beta \in \overline{\Q}_p$ ses racines. On dit que $f$ est de pente finie si au moins une de ces racines est non nulle. Remarquons simplement que, si $p \nmid N$, alors $f$ est de pente finie, et que, si $p$ divise $N$, alors $f$ est de pente finie si et seulement si $a_p \neq 0$. Une des premières constructions de la fonction $L$ $p$-adique de $f$ dépend du choix d'une racine non nulle, disons $\alpha$, du polynôme de Hecke en $p$ de la forme $f$.

\begin{theorem} [\cite{Manin}, \cite{AmiceVelu}, \cite{Visik}, \cite{MTT}, \cite{P-S}, \cite{Bellaiche}, \cite{Delbourgo}] \label{MTT} Soient $f$ de pente finie et $\alpha$ comme ci-dessus. Il existe une (unique si $v_p(\alpha) < k - 1$) fonction $L_{p, \alpha}(f) \in \mathcal{O}(\mathfrak{X})$, d'ordre $v_p(\alpha)$, telle que, si $\eta \colon \zpe \to \C_p^\times$ est un caractère de Dirichlet de conducteur $p^n$ et $j$ un entier tel que $ 0 \leq j \leq	 k - 2$, on a
\[ L_{p, \alpha}(f)(\eta \chi^j) = e_{p, \alpha}(f, \eta, j) \frac{p^{n (j + 1)}}{G(\eta^{-1})} \cdot \iota_p(\Lambda_\infty(f, \eta^{-1}, j + 1)) \]
où $G(\eta^{-1}) = \sum_{a = 0}^{p^n - 1} \eta^{-1}(a) \exp({\frac{2 i \pi a}{p^n}}) \in \overline{\Q}_p$ dénote la somme de Gauss du caractère $\eta^{-1} $ et le facteur $e_{p, \alpha}(f, \eta, j)$ est défini par la formule \[ e_{p, \alpha}(f, \eta, j) = \left\{ 
\begin{array}{l l}
  \alpha^{-n}  & \quad \text{si $n > 0$}\\
  (1 - \alpha^{-1} \omega_f(p) p^{k - 2 - j})(1 - \alpha^{-1} p^j) & \quad \text{si $n = 0$.} \\ \end{array} \right. \]
\end{theorem}

La construction peut être adaptée (\cite{Delbourgo}) à la situation où $f$ est de pente infinie mais $f \otimes \xi$ est de pente finie pour un certain caractère $\xi$. Si $f \otimes \xi$ n'est pas de pente finie pour aucun caractère $\xi$, on dit que la forme $f$ est \textit{supercuspidale}. Cette condition correspond (\cite[Prop. 2.8]{L-W}) à ce que la représentation lisse de $\mathrm{GL}_2(\qp)$ associée à $f$ soit supercuspidale.

%
%

\subsection{Le système d'Euler de Kato} 

Il existe une construction alternative (\cite{Kato}, \cite{ColmezBourbaki}) de la fonction $L$ $p$-adique, qui est moins élémentaire mais a pourtant l'avantage de permettre de relier les valeurs en certains points entiers de la fonction $L$ $p$-adique à des quantités de nature cohomologique, ce qui permet, par exemple, de démontrer des instances de la conjecture de Birch et Swinnerton-Dyer $p$-adique et la conjecture principale d'Iwasawa pour les représentations galoisiennes attachées aux formes modulaires. Nous rappelons dans ce qui suit la construction de Kato de la fonction $L$ $p$-adique dans le cas où $p$ est un nombre premier ne divisant pas le niveau $N$ de $f$ \footnote{cf. \cite{ColmezBourbaki} pour les modifications nécessaires dans le cas semi-stable et \cite{Delbourgo} pour le cas où la représentation dévient cristalline sur une extension abélienne de $\qp$.}.

Soit $L$ une extension finie de $\qp$ contenant les coefficients de Fourier de la forme $f$ et notons $\mathscr{O}_L$ l'anneau des entiers de $L$. Soit $V(f)$ la $L$-représentation du groupe de Galois absolu $\mathscr{G}_\mathbf{Q}$ de $\mathbf{Q}$ associée à $f$ par Shimura-Deligne: elle est de dimension $2$, non ramifiée en dehors $Np$, de Rham en $p$ et caractérisée par le fait que, pour tout $\ell \nmid Np$, alors $$ \mathrm{det}(1 - \mathrm{Frob}_\ell^{-1}X | V(f)^{I_\ell}) = 1 - a_\ell X + \ell^{k-1} \omega_f(\ell) X^2, $$ où $\mathrm{Frob}_\ell$ désigne le Frobenius arithmétique en $\ell$ et $I_\ell \subseteq \mathscr{G}_{\Q_\ell} \subseteq \mathscr{G}_{\Q}$ dénote le groupe d'inertie du groupe de Galois absolu $\mathscr{G}_{\Q_\ell} = \mathrm{Gal}(\overline{\Q}_\ell / \Q_\ell)$ de $\Q_\ell$ vu dedans $\mathscr{G}_\Q$ comme un groupe de décomposition. Notons aussi par $V(f)$ la restriction de $V(f)$ au groupe $\mathscr{G}_\qp$, qui est une représentation de dimension $2$, de Rham (cristalline si $p \nmid N$) à poids de Hodge-Tate $0$ et $1 - k$.

Soient $F_\infty = \cup_n \qp(\mu_{p^n})$ l'extension cyclotomique de $\qp$, $\Gamma_n = \mathrm{Gal}(F_\infty / \qp(\mu_{p^n})) \subseteq \Gamma = \mathrm{Gal}(F_\infty / \qp)$ et $\chi: \Gamma \xrightarrow{\sim} \zpe$ le caractère cyclotomique. La construction de Kato repose sur la construction d'un système d'Euler \footnote{Les éléments zêta de Kato dépendent d'un certain nombre de choix que l'on ignore dans cette introduction afin de simplifier l'exposition.} $\mathbf{z}_{\rm Kato}$ attaché à $f$ dans la représentation $V = V(f)^*(1) = V(\check{f})(k)$ (qui est de Rham à poids de Hodge-Tate $1$ et $k$ en $p$) et dont les niveaux en les différentes puissances de $p$ fournissent un élément, aussi noté $\mathbf{z}_{\mathrm{Kato}}$, de la cohomologie d'Iwasawa de la représentation $V$ définie par
\[ H_{\rm Iw}^1(\qp, V) = \varprojlim_n H^1(\qp(\mu_{p^n}), T) \otimes \qp \cong H^1(\mathscr{G}_{\qp}, \Lambda \otimes V), \]
où $T$ dénote n'importe quel $\mathscr{O}_L$-réseau de $V$ stable par $\mathscr{G}_\qp$, $\Lambda = \zp[[\Gamma]]$ dénote l'algèbre d'Iwasawa de $\Gamma$ et où la limite projective est prise par rapport aux applications de corestriction (le dernier isomorphisme étant une conséquence du lemme de Shapiro).

Un théorème fondamental de Kato (\cite[Thm. 12.5]{Kato}) montre que l'élément $\mathbf{z}_{\rm Kato}$ est intimement lié aux valeurs spéciales de la fonction $L$ complexe de la forme $f$. Si $p \nmid N$ (i.e si $V(f)$ est cristalline), ce théorème permet à Kato d'appliquer la machine à fonctions $L$ $p$-adiques de Perrin-Riou (\cite{PerrinRiou1}, \cite{PerrinRiou2}, \cite{Kato})
\[ \mathrm{Log}_{V} \colon H^1_{\rm Iw}(\qp, V) \otimes_{\Lambda(\Gamma)} \mathscr{D}(\Gamma) \to \mathcal{O}(\mathfrak{X}) \otimes \Dcris(V), \]
où $\mathscr{D}(\Gamma)$ dénote l'algèbre de distributions sur $\Gamma$ \footnote{D'après un théorème d'Amice, $\mathscr{D}(\Gamma)$ et $\mathscr{O}(\mathfrak{X})$ sont isomorphes, or on garde la notation $\mathscr{D}(\Gamma)$, qui est plus souvent utilisée.}, interpolant $p$-adiquement les applications exponentielles et (d'après un théorème de Colmez, Benois et Kato-Kurihara-Tsuji) exponentielles duales de Bloch-Kato pour des différentes tordues de la représentation en question, pour obtenir une nouvelle construction de la fonction $L$ $p$-adique de $f$.

\begin{theorem} [ {\cite[Thm. 16.6]{Kato}} ] Soit $e_\alpha \in \Dcris(V(f)) = \Dcris(V^*(1))$ un vecteur propre du Frobenius cristallin de valeur propre $\alpha$. On a alors $$ L_{p,\alpha}(f) = \langle \mathrm{Log}_V(\mathbf{z}_{\mathrm{Kato}}), e_{\alpha} \rangle, $$ où $\langle \, , \, \rangle : \Dcris(V) \times \Dcris(V^*(1)) \to L$ dénote l'accouplement de Poincaré.
\end{theorem}

Remarquons que l'on dispose dans tous les cas d'un système d'Euler $\mathbf{z}_{\rm Kato}$ et que l'application de Perrin-Riou a été généralisée pour les représentations de de Rham par des travaux de Colmez (\cite{ColmezIw1}) et Cherbonnier-Colmez (\cite{ColmezIw2}) et, pour un $(\varphi, \Gamma)$-module sur l'anneau de Robba $\Robba$, par Nakamura (\cite{Nakamura}). Or ces applications ne s'expriment pas naturellement en termes de fonctions analytiques sur l'espace des poids et ne fournissent malheureusement (ou heureusement) pas si simplement des fonctions $L$ $p$-adiques.

\subsection{Le cas supercuspidal}

Décrivons brièvement le résultat principal de cet article. L'isomorphisme $\overline{\mathbf{Q}}_p \cong \C$ que l'on a fixé permet de voir $\eta \colon \zpe \to L^\times$ comme un caractère de Dirichlet de conducteur $p^n$, $n \geq 0$, et on note $f \otimes \eta^{-1}$ la forme tordue de $f$ par $\eta^{-1}$, qui est aussi une forme primitive (de niveau $Np^{2n}$ et caractère $\omega_f \eta^{-2}$, au moins si $n$ est assez grand). Rappelons que l'on a posé \[ \Lambda_\infty(f, \eta^{-1}, s) = \frac{\Gamma(s)}{(2 i \pi)^s} \cdot L(f, \eta^{-1}, s) \] et que la forme $f$ donne lieu à une représentation automorphe $\pi(f) = {\prod'_{v}} \pi_v(f)$ de $\mathrm{GL}_2(\mathbf{A}_\Q)$, où $v$ parcourt l'ensemble des places de $\Q$ et $\pi_{v}(f)$ est une représentation lisse de $\mathrm{GL}_2(\Q_v)$. Pour $\eta$ un caractère de Dirichlet, vu comme un caractère des adèles, et $v$ une place de $\Q$, on note $\epsilon(\pi_v(f) \otimes \eta)$ les facteurs epsilon des composantes locales de la représentation $\pi(f) \otimes \eta$ \footnote{On a $\epsilon(\pi_\infty(f)) = i^k$.}, de sorte que le facteur epsilon global associé à $f$ et $\eta$ est donné par la formule \[ \epsilon(\pi(f) \otimes \eta) = \epsilon(\pi_\infty(f)) \cdot \prod_{\ell \mid N} \epsilon(\pi_\ell(f) \otimes \eta). \] La fonction $\Lambda_\infty(f, \eta^{-1}, s)$ satisfait alors l'équation fonctionnelle \[ \Lambda_\infty(f, \eta^{-1}, j) = i^k (-1)^{j} \epsilon(\pi(f) \otimes \eta^{-1} \otimes | \cdot |^{j - \frac{k - 1}{2}}) \cdot \Lambda_\infty (f, \eta, k - j), \;\;\; j \in \Z \] On obtient le théorème suivant:

\begin{theorem} [def. \ref{defplong} + thm. \ref{propfina} + lemme \ref{eaea2} + thm. \ref{eqfonct4}] \label{eaealala}
Soit $f \in S_k(\Gamma_1(N), \omega_f)$ une forme primitive et soit $V = V(\check{f})(k - 1)$. Il existe des plongements naturels \[ \Lambda_\infty(f, \eta^{-1}, j) \mapsto \iota_p(\Lambda_\infty(f, \eta^{-1}, j)) \in \DdR(V), \;\;\; j \leq k-1, \] un ouvert $\mathfrak{U}_f \subseteq \mathfrak{X}$ ne dépendant que de la puissance de $p$ divisant le niveau $N$ de la forme $f$ et contenant tous les caractères d'ordre $p^n$ pour $n$ assez grand, et une unique fonction rigide analytique $\Lambda_p(f) \in \mathcal{O}(\mathfrak{U}_f) \otimes \DdR(V)$ telle que, si $\eta \colon \zpe \to L^\times$ est un caractère de conducteur $p^n$ et $j < k-1$ sont tels que $\eta x^j \in \mathfrak{U}_f$, alors
\[ \Lambda_p(f)(\eta x^j) = \frac{p^{n( j + 1)}}{G(\eta^{-1})} \cdot \iota_p(\Lambda_\infty(f, \eta^{-1}, j + 1)). \]

De plus, la fonction $\Lambda_p(f)$ satisfait une équation fonctionnelle de la forme \footnote{La formulation de l'équation fonctionnelle est légèrement imprécise, cf. théorème \ref{eqfonct4}.} \[ \Lambda_p(f)(\eta x^j) = C(f, \eta, j) \cdot \Lambda_p(\check{f})(\eta^{-1} x^{k - 2 - j}), \] où
\[ C(f, \eta, j) =  p^n \, \epsilon(\eta \otimes | \cdot |^{-j + \frac{k - 1}{2}})^2 \epsilon(\pi_p(\check{f}) \otimes \eta \otimes | \cdot |^{-j + k - 1})^{-1} \cdot \prod_{\ell \mid N'} \epsilon(\pi_\ell(\check{f}) \otimes \eta^{-1} \otimes |\cdot|^{-j + \frac{k - 1}{2}})^{-1}, \] où $N'$ dénote la partie de $N$ première à $p$ et $\epsilon(\eta \otimes | \cdot |^s) = p^{-ns} \eta(p)^n G(\eta^{-1})$.

Finalement, si $p \nmid N$ alors $\mathfrak{U}_f = \mathfrak{X}$ et, si $\alpha$ est une valeur propre du polynôme de Hecke en $p$ de $f$ et $e_\alpha \in \Dcris(V(f)) = \Dcris(V)^*$ est un vecteur propre du Frobenius cristallin de valeur propre $\alpha$, on a \[ L_{p, \alpha}(f) = \langle \Lambda_{p}(f), e_\alpha \rangle. \]
\end{theorem}

\begin{remarque}
 Les méthodes du théorème fournissent, plus généralement, pour toute représentation $V \in \mathrm{Rep}_L \mathscr{G}_\qp$ de Rham de dimension $d$ et tout $z \in H^1_{\rm Iw}(\qp, V)$, un ouvert $\mathfrak{U}_V \subseteq \mathfrak{X}$, ne dépendant que de la valuation $p$-adique du discriminant de la plus petite extension finie galoisienne $K / \qp$ sur laquelle $V$ dévient semi-stable, et une unique fonction $\Lambda_{V, z} \in \mathcal{O}(\mathfrak{U}_V)$ interpolant des exponentielles et exponentielles duales de différentes spécialisations de l'élément $z$. Dans le cas général, on n'a pas, hélas! une interprétation si satisfaisante des valeurs interpolées.
\end{remarque}

\subsection{Démonstration}

Disons quelques mots sur la démonstration du théorème.

\subsubsection{Plongements $p$-adiques des valeurs spéciales}

Soit $M = M(f \otimes \eta^{-1})$ le motif associé à la forme $f \otimes \eta^{-1}$ et considérons, pour $j \geq 0$, \[ M^*(1+j) = M(\check{f} \otimes \eta)(k+j), \] dont $V(\eta \chi^{j + 1})$ est la réalisation $p$-adique. En utilisant les symboles d'Eisenstein définis par Beilinson (\cite{Beilinson2}), on peut construire des éléments (cf. \cite[\S 3.2]{Gealy2}) $ \mathscr{Z}(\check{f} \otimes \eta, j) \in H^1(M^*(1+j))$ et un résultat de Gealy (\cite[Thm. 4.1.1]{Gealy2}) montre qu'ils satisfont une variante de la conjecture de Bloch-Kato.

Par ailleurs, grâce aux travaux de Gros (\cite{Gros}), Niziol (\cite{Niziol}), Besser (\cite{Besser}), et Nekovar-Niziol (\cite{NN}), on dispose aussi des régulateurs $p$-adiques \[ r_p \colon H^1(M^*(1 + j)) \to \DdR(V(\eta \chi^{j + 1})), \] qui satisfont la relation de commutativité \[ r_{\text{ét}} = \exp \circ \; r_p, \] où $\exp$ est l'exponentielle de Bloch-Kato.

Le résultat de Gealy mentionné ci-dessus suggère de considérer les régulateurs $p$-adiques des éléments $\mathscr{Z}(\check{f} \otimes \eta, j)$ comme les valeurs $p$-adiques naturels à interpoler. Notons que les éléments $r_p(\mathscr{Z}(\check{f} \otimes \eta, j)) \in \DdR(V(\eta \chi^{j + 1}))$ vivent dans des modules différents. Nous allons expliquer comment les voir tous naturellement comme des éléments dans $\DdR(V)$.


Si $\eta \colon \zpe \to L^\times$ est un caractère d'ordre fini, $G(\eta)$ dénote sa somme de Gauss et $a \in \zpe$, alors $\sigma_a(G(\eta)) = \eta(a)^{-1} G(\eta)$ et, si $e_\eta$ dénote une base du module $L(\eta)$ \footnote{$L(\eta)$ dénote le $L$-espace vectoriel de dimension $1$ sur lequel $\mathscr{G}_\qp$ agit à travers $\eta$.}, le groupe $\Gamma$ (et donc aussi $\mathscr{G}_\qp$) agit trivialement sur l'élément $\mathbf{e}^{\rm dR}_\eta = G(\eta) \cdot e_\eta \in \BdR \otimes L(\eta)$ et est donc une base du $L$-espace vectoriel $\DdR(L(\eta))$. Si $j \in \Z$, il en est de même de $\mathbf{e}^{\rm dR}_j = t^{-j} e_j \in \BdR \otimes L(\chi^j)$. Notons \[ \mathbf{e}^{\rm dR}_{\eta, j} =  \mathbf{e}^{\rm dR}_\eta \otimes \mathbf{e}^{\rm dR}_j = G(\eta) e_\eta \otimes t^{-j} e_j \] qui est une base du module $\DdR(L(\eta \chi^j))$ et, pour $e_\eta^\vee = e_{\eta^{-1}}$ et $e_{-j}$ les éléments duaux de $e_\eta$ et $e_j$, on note \[ \mathbf{e}^{\rm dR, \vee}_{\eta, j} = G(\eta)^{-1} e_\eta^\vee \otimes t^j e_{-j}, \] qui est une base du module $\DdR(L(\eta \chi^j)^*)$.

Si $V \in \mathrm{Rep}_L \mathscr{G}_\qp$ est de Rham, alors $V(\eta \chi^j)$ l'est aussi et on a, par ce qui précède, \[ \DdR(V(\eta \chi^j)) = (\BdR \otimes V \otimes L(\eta \chi^j))^{\mathscr{G}_\qp} = (\BdR \otimes V)^{\mathscr{G}_\qp} \otimes \mathbf{e}^{\rm dR}_{\eta, j} = \DdR(V) \otimes \mathbf{e}^{\rm dR}_{\eta, j}, \] de sorte que l'application $x \mapsto x \otimes \mathbf{e}^{\rm dR, \vee}_{\eta, j}$ induit un isomorphisme \[ \DdR(V(\eta \chi^j)) \xrightarrow{\sim} \DdR(V). \]

Notons \[ \Gamma^*(j+1) = \left\{
  \begin{array}{c c}
    j! & \quad \text{ si $j \geq 0$} \\
    \frac{(-1)^{j-1}}{(-j - 1)!} & \quad \text{si $j < 0$}  \\
   \end{array} \right. \] le coefficient principal de la série de Laurent de la fonction $\Gamma(s)$ en $s = j + 1$. On pose, pour $j \geq 0$,
\[ \iota_p(\Lambda_\infty(f, \eta^{-1}, -j)) =\Gamma^*(j) \, G(\eta) \cdot r_p( \mathscr{Z}(\check{f} \otimes \eta, j)) \otimes \mathbf{e}^{\rm dR, \vee}_{\eta, j + 1} \in \DdR(V). \] Remarquons que, dans la formule définissant $\iota_p(\Lambda_\infty(f, \eta^{-1}, -j))$, on "multiplie" et "divise" par $G(\eta)$, de sorte que son introduction n'a moralement aucun effet. La valeur $\Gamma^*(j)$ correspond donc à $\Gamma(j)$ et la puissance $t^{j + 1}$ correspond à la puissance de $2 i \pi$ dans la définition de $\Lambda_\infty(f, \eta^{-1}, -j)$.

\subsubsection{Interpolation: un premier exemple}

La construction de la fonction $L$ locale repose sur la théorie des $(\varphi, \Gamma)$-modules, et est inspirée de la ressemblance entre celle-ci et l'analyse fonctionnelle $p$-adique, via le dictionnaire d'analyse fonctionnelle $p$-adique de Colmez, qui est aussi à la base même de la construction de la correspondance de Langlands $p$-adique pour $\mathrm{GL}_2(\qp)$.

Soit $\mathrm{LA}(\zp, L)$ l'espace des fonctions localement analytiques sur $\zp$ à valeurs dans $L$ et $\mathscr{D}(\zp, L)$, son $L$-dual topologique, l'espace des distributions sur $\zp$. Si $\phi \in\mathrm{LA}(\zp, L)$ et $\mu \in \mathscr{D}(\zp, L)$, on note $\int_\zp \phi \cdot \mu$ l'évaluation de $\mu$ en $\phi$. On a des actions du groupe $\Gamma$ et des opérateurs $\varphi$ et $\psi$ sur l'espace de distributions définis par les formules
\small{\[ \int_\zp \phi(x) \cdot \sigma_a(\mu) = \int_\zp \phi(ax) \cdot \mu, \;\;\; \int_\zp \phi(x) \cdot \varphi(\mu) = \int_\zp \phi(px) \cdot \mu, \;\;\; \int_\zp \phi(x) \cdot \psi(\mu) = \int_{p \zp} \phi(x /p) \cdot \mu, \]}\normalsize
et une opération de "multiplication par $x$" définie par $\int_\zp \phi(x) \cdot m_x(\mu) = \int_\zp \phi(x) x \cdot \mu$. Notons que $\psi(\mu) = 0$ si et seulement si la distribution $\mu$ est supportée sur $\zpe$ et que, si $\psi(\mu) = \mu$, alors $(1 - \varphi) \mu$ est la restriction à $\zpe$ de $\mu$.

La transformée d'Amice \[ \mu \mapsto \mathscr{A}_\mu = \int_\zp (1 + T)^x \cdot \mu \] induit un isomorphisme $\mathscr{D}(\zp, L) \xrightarrow{\sim} \Robba^+$, où $\Robba^+ = \Robba \cap L[[T]]$. L'opérateur différentiel $\partial = (1 + T) \frac{d}{dT}$ sur $\Robba$ correspond à la multiplication par $x$ sur les distributions au sens que $\partial \mathscr{A}_\mu = \mathscr{A}_{m_x(\mu)}$ et les actions de $\varphi$ et de $\sigma_a \in \Gamma$ correspondent à $\varphi(T) = (1 + T)^p - 1$ et $\sigma_a(T) = (1 + T)^a- 1$. En fixant un système compatible $(\zeta_{p^n})_{n \in \N}$ de racines primitives de l'unité (c'est-à-dire, $\zeta_{p^n} \in \overline{\Q}_p$ est une racine primitive $p^n$-ième de $1$ et $\zeta_{p^{n + 1}}^p = \zeta_{p^n}$ pour tout $n \in \N$), on peut définir des applications de "localisation" \[ \varphi^{-n} \colon \E^{]0, r_n]} \to L_\infty[[t]]: \;\;\;  T \mapsto \zeta_{p^n} e^{t / p^n} - 1, \] où $\E^{]0, r_n]}$ dénote les éléments de $\Robba$ qui convergent sur la couronne $0 < v_p(z) \leq r_n = \frac{1}{p^{n - 1}(p - 1)}$ et $L_\infty[[t]] = \varinjlim L_n[[t]]$. L'opérateur $\partial$ stabilise $\E^{]0, r_n]}$ et on a l'identité $\varphi^{-n} \circ \partial = p^n \frac{d}{dt} \circ \varphi^{-n}$. Si $x = \sum_{l \in \N} a_l t^l \in L_\infty[[t]]$, on note $[x]_0 = a_0$. Le lemme suivant nous permet de réécrire l'intégration $p$-adique en termes des applications de localisation.

\begin{lemme} \label{lemmeintro} 
Pour tout $f \in \Robba^{\psi = 0}$, il existe une unique fonction analytique $\kappa \mapsto \kappa(\partial) f \colon \mathfrak{X} \to \Robba^{\psi = 0}$ interpolant les valeurs $\partial^j f$, $j \in \Z$, aux caractères $x^j$. De plus, si $\mu \in \mathscr{D}(\zp, L)^{\psi = 1}$, $\eta \colon \zpe \to L^\times$ est un caractère de conducteur $p^n$, $n > 0$, et $\kappa \in \mathfrak{X}$, alors \footnote{Si $a \in (\Z / p^n \Z)^\times$, on note $\sigma_a \in \Gamma_n$ l'élément lui correspondant par le caractère cyclotomique.} \[  \int_\zpe \eta^{-1} \kappa \cdot \mu = G(\eta)^{-1} \sum_{a \in (\Z / p^n \Z)^\times} \eta(a) \sigma_a [\varphi^{-n} \kappa(\partial) (1 - \varphi) \mathscr{A}_\mu]_0. \]
\end{lemme}

L'avantage du lemme ci-dessus est que le terme de droite a un sens, pour tout $(\varphi, \Gamma)$-module $D$ de de Rham, si l'on remplace $\mathscr{A}_\mu$ par $z \in \Nrig(D)^{\psi = 1}$, où $\Nrig(D)$ est l'équation différentielle $p$-adique associée à $D$ par Berger, dès que $n$ est assez grand, comme on le verra à continuation.

\subsubsection{Interpolation: le cas général}

Comme suggéré par les résultats de Kato, la construction de $\Lambda_p(f)$ demande à étendre la construction du logarithme de Perrin-Riou aux $(\varphi, \Gamma)$-modules de de Rham.

Soit $V \in \mathrm{Rep}_L \mathscr{G}_\qp$ et notons $D = \D_{\rm rig}(V)$ le $(\varphi, \Gamma)$-module sur l'anneau de Robba $\Robba$ qui lui est associé par l'équivalence de catégories de Fontaine \cite{Fontaine90}, Cherbonnier-Colmez \cite{CC} et Kedlaya \cite{Kedlaya}. On a un isomorphisme, dû à Fontaine et Pottharst (\cite{ColmezIw2}, \cite{Pottharst}), \[ \mathrm{Exp}^* \colon H^1(\qp, \mathscr{D}(\Gamma) \otimes V)) \xrightarrow{\sim} \D_{\rm rig}(V)^{\psi = 1}. \]

Suppposons $V$ de Rham et notons $\Delta = \Nrig(D)$ l'équation différentielle $p$-adique associée à $D$ par Berger (\cite{Berger02}, \cite{Berger03}). Si $D$ est à poids de Hodge-Tate positifs, on a $D \subseteq \Delta$. On dispose d'un opérateur de connexion $\partial$ sur $\Delta$ au-dessus de l'opérateur $\partial = (1 + T) \frac{d}{dT}$ de "multiplication par $x$" sur $\Robba$.

Pour $n \gg 0$, il existe des sous-$\E^{]0, r_n]}$-modules $\Delta^{]0, r_n]} \subseteq \Delta$ tels que $\Delta^{]0, r_n]} \otimes_{\E^{]0, r_n]}} \Robba = \Delta$. Ces modules sont stables par l'action de $\Gamma$ et satisfont $\varphi(\Delta^{]0, r_n]}) \subseteq \Delta^{]0, r_{n + 1}]}$. On a, pour tout $n \gg 0$, des applications de localisation \[ \varphi^{-n} \colon \Delta^{]0, r_n]} \to L_n[[t]] \otimes \DdR(V), \]
et on note $[\cdot]_0 \colon L_n[[t]] \otimes \DdR(V) \to L_n \otimes \DdR(V)$ l'application $\sum_l a_l t^l \otimes d_l \mapsto a_0 \otimes d_0$. L'opérateur stabilise $\Delta^{]0, r_n]}$ et on a $ \varphi^{-n} \circ \partial = p^n \frac{d}{dt} \circ \varphi^{-n}. $ 

Si $z \in D^{\psi = 1} \subseteq \Delta^{\psi = 1}$, $\eta$ est un caractère de conducteur $p^n$, $n \gg 0$ et $j \geq 0$, on a l'égalité suivante dans $L_n \otimes \DdR(D)$, qui est l'analogue, en termes de théorie d'Iwasawa, de l'égalité du lemme \ref{lemmeintro} ci-dessus:
\small{\begin{equation} \label{eq1} G(\eta)^{-1} \sum_{a \in (\Z / p^n \Z)^\times} \eta(a) \sigma_a [\varphi^{-n} \partial^j (1- \varphi) z]_0 = p^{n(j + 1)} \, \Gamma^*(j + 1) \cdot \exp^* (\int_{\zpe} \eta \chi^{-j} \mu_z) \otimes \mathbf{e}^{\rm dR, \vee}_{\eta, -j},
\end{equation}}\normalsize
où $\exp^* \colon H^1(\qp, V(\eta \chi^{-j})) \to \DdR(V(\eta \chi^{-j})) = \DdR(V) \otimes \mathbf{e}^{\rm dR}_{\eta, -j}$ dénote l'exponentielle duale de Bloch-Kato. La proposition suivante permet de prolonger analytiquement le terme de gauche de cette égalité.

\begin{proposition} [prop. \ref{prop4}]
Pour tout $z \in \Delta^{\psi = 0}$, il existe une unique fonction rigide analytique $\kappa \mapsto \kappa(\partial)z \colon \mathfrak{X} \to \Delta^{\psi = 0}$ interpolant les valeurs $\partial^j z$, $j \in \Z$, aux caractères $x^j$.
\end{proposition}

La restriction à $\zpe$ (i.e l'application de l'opérateur $1 - \varphi$) permet d'interpoler $p$-adiquement l'opérateur $\partial$ mais, en revanche, l'application de spécialisation $\varphi^{-n}$ n'est pas définie sur tout le module $\Delta^{\psi = 0}$. Une étude du rayon de convergence de l'élément $\kappa(\partial)$, pour un caractère $\kappa \in \mathfrak{X}(L)$, montre que le terme de gauche de l'équation \label{eq2} définit bien une fonction analytique sur une boule ouverte autour $\eta$ dans l'espace des poids. Le théorème suivant, qui est une généralisation du Logarithme de Perrin-Riou et le résultat principal de cet article, appliqué à $D = \Drig(V(f)(k - 1))$, fournit la construction de la fonction $\Lambda_p(f)$. On renvoie au texte pour les notations qui n'ont pas encore été introduites.

\begin{theorem} [thm. \ref{propfina}] \label{propfinaintro}
Soient $D$ un $(\varphi, \Gamma)$-module qui est de Rham à poids de Hodge-Tate positifs et $z \in D^{\psi = 1}$. Notons $\mu_z = \mathrm{Exp}^*(z) \in H^1_{\rm Iw}(\qp, D)$. Il existe un ouvert $\mathfrak{U}_D \subseteq \mathfrak{X}$, ne dépendant que du module $\D_{\rm pst}(D)$, et une unique fonction analytique $\Lambda_{D, z} \in \mathcal{O}(\mathfrak{U}_D) \otimes \DdR(D)$ telle que, pour tout $\eta \chi^j \in \mathfrak{U}_D$, où $\eta$ est un caractère de conducteur $p^n$ et $j \in \Z$, on a
\[ \Lambda_{D, z} (\eta \chi^j) = p^{n(j + 1)} \Gamma^*(j + 1) \cdot \left\{ \begin{array}{l l}
  \exp^{-1} (\int_{\Gamma} \eta \chi^{-j} \cdot \mu_z)  \otimes \mathbf{e}^{\rm dR, \vee}_{\eta, -j}  & \quad \text{si $j \ll 0$} \\
  \exp^* (\int_{\Gamma} \eta \chi^{-j} \cdot \mu_z)  \otimes \mathbf{e}^{\rm dR, \vee}_{\eta, -j} & \quad \text{si $j \geq 0$.} \\ \end{array} \right. \] 
\end{theorem}

Ce théorème, accouplé à un deuxième résultat de Gealy reliant les éléments motiviques du théorème \ref{GealyBK} et le système d'Euler de Kato (cf. prop. \ref{Gealy2}), permet de démontrer l'existence de l'application $\Lambda_p(f)$ du théorème \ref{eaealala}. Le fait que dans le cas cristallin l'on récupère la fonction $L$ $p$-adique classique associée à la forme $f$ suit du fait que notre application est une généralion du Logaritme de Perrin-Riou et de la description de la fonction $L$ $p$-adique de $f$ en termes du système d'Euler de Kato.

\subsubsection{L'équation fonctionnelle en dimension $2$ et valeurs aux entiers positifs}

Il est naturel de se demander si ce que la fonction $\Lambda_p(f)$ interpole aux entiers positifs s'interprète aussi en termes de valeurs spéciales de la fonction $L$ complexe. Rappelons que, dans la bande critique $1 \leq j \leq k - 1$, les valeurs $L(f, \eta, j)$ sont naturellement interprétées $p$-adiquement en les divisant par les périodes complexes de la forme $f$. Si $j > k - 2$, une équation fonctionnelle purement locale demontrée dans \cite{epsilonKato}, accouplée avec une équation fonctionnelle satisfaite par le système d'Euler de Kato demontrée par Nakamura (prop. \ref{eqfonctNakamura}), fournit l'équation fonctionnelle du théorème \ref{eaealala}.

\subsection{Plan de l'article}

L'organisation de l'article ne reflet pas celle de l'introduction. Les premières sections présentent des résultats purement locaux, reportant l'application aux formes modulaire à la fin.

Le premier chapitre contient le premier résultat principal de la thèse, à savoir l'extension analytique (partielle) de l'application Logarithme de Perrin-Riou pour les représentations de de Rham, fournissant des fonctions rigides analytiques sur des ouverts de l'espace des poids interpolant les applications exponentielle et exponentielle duale de Bloch et Kato. On y trouvera aussi des rappels et généralités sur les $(\varphi, \Gamma)$-modules, ainsi qu'une estimation précise du rayon de convergence de $\Delta$ en termes du module $\D_{\rm pst}(D)$, permettant de décrire l'ouvert $\mathfrak{U}_D$ du théorème \ref{propfinaintro}. Dans le cas étale de dimension $2$, en utilisant le résultat principal de \cite{epsilonKato}, on obtient une équation fonctionnelle pour notre fonction $L$ locale.

Finalement, on montre, à l'aide d'un théorème de M. Gealy et des conjectures de Bloch-Kato pour les formes modulaires, comment la construction de la fonction $L$ locale peut être utilisée pour donner une construction de la fonction $L$ $p$-adique d'une forme modulaire, sans aucune hypothèse sur sa pente, et montrer qu'elle interpole des valeurs spéciales de la forme (dûment interprétées $p$-adiquement) en tout caractère algébrique de composante finie suffisament ramifiée.

\subsection{Remerciements}

Le travail présent fait partie de ma thèse de doctorat réalisée à l'Institut de Mathématiques de Jussieu-Paris Rive Gauche (IMJ-PRG) sous la direction de Pierre Colmez. Je tiens à lui exprimer mes plus grands et sincères remerciements pour tout ce qu'il m'a appris et pour tout le temps qu'il a investi dans ma formation mathématique, avec disponibilité et bonne humeur. Je remercie aussi Arthur-César Le Bras, Gabriel Dospinescu et Shanwen Wang, avec qui j'ai partagé des nombreuses discussions qui ont beaucoup aidé à ma compréhension du sujet. Je remercie finalement les rapporteurs pour ses remarques et corrections.

\section*{Notations} On fixe certaines notations:

\begin{itemize}
\item Soient $p$ un nombre premier, $\qp$ le corps des nombres $p$-adiques, $\zp \subseteq \qp$ l'anneau des entiers de $\qp$, $\zpe$ le groupe des unités et $\cp$ la complétion $p$-adique de la clôture algébrique $\overline{\Q}_p$ de $\qp$.

\item On fixe un système $(\zeta_{p^n})_{n \in \N}$, où $\zeta_{p^n} \in \overline{\Q}_p$, de racines $p^n$-ièmes primitives de l'unité satisfaisant la relation $\zeta_{p^{n+1}}^p = \zeta_{p^n}$ dès que $n \geq 0$. Si $n \in \N$, on note $F_n = \qp(\zeta_{p^n})$ le $n$-ième niveau de la tour cyclotomique et $F_\infty = \cup_n F_n$.

\item Soit $\mathscr{G}_{\qp} = \mathrm{Gal}(\overline{\Q}_p / \qp)$ le groupe de Galois absolu de $\qp$. On définit le caractère cyclotomique $\chi \colon \mathscr{G}_\qp \to \zpe$ par la formule $g(\zeta_{p^n}) = \zeta_{p^n}^{\chi(g)}$, pour tout $n \in \N$, et on note $\mathscr{H} = \mathscr{H}_\qp$ le noyau de $\chi$, qui s'identifie au groupe de Galois absolu de $F_\infty$, et $\Gamma = \Gamma_\qp = \mathscr{G}_\qp / \mathscr{H}_\qp$ qui s'identifie à $\mathrm{Gal}(F_\infty / \qp)$. Le caractère cyclotomique induit un isomorphisme $\chi \colon \Gamma \xrightarrow{\sim} \zpe$. Si $a \in \zpe$, on note $\sigma_a \in \Gamma$ son inverse par $\chi$.

\item Soit $L$ une extension finie de $\qp$, qui sera notre corps de coefficients, et notons $L_n = L \otimes F_n = L(\mu_{p^n})$ et $L_\infty = \cup_n L_n$. On note $\mathrm{Rep}_L(\mathscr{G}_\qp)$ la catégorie des $L$-espaces vectoriels de dimension finie munis d'une action $L$-linéaire continue du groupe $\mathscr{G}_\qp$. On omettra souvent $L$ des notations, cependant il faut être conscient que $L$ est le corps sous-jacent et qu'il change au fur et à mesure des besoins. Par exemple, si $\eta$ est un caractère de $\mathscr{G}_\qp$ et si on note $\qp(\eta)$ la représentation de dimension $1$ engendrée par un élément $e_\eta$ sur lequel l'action de $\mathscr{G}_\qp$ est définie par la formule $g(e_\eta) = \eta(g) e_\eta$, on sous-entendra que le corps $L$ de coefficients contient les valeurs prises par $\eta$.
\end{itemize}

\subsubsection{Anneaux des séries de Laurent}

Commençons par quelques définitions classiques.

\begin{itemize}
\item On définit le corps $\E$ par $$ \E = \big\lbrace \sum_{k \in \Z} a_k T^k: \; a_k \in L; \;  \liminf_{k \to +\infty} v_p(a_k) > -\infty; \; \lim_{k \to -\infty} v_p(a_k) = + \infty \big\rbrace,  $$ muni de la valuation donnée par $v_\E(\sum_{k \in \Z} a_k T^k) = \inf_k v_p(a_k)$, ce qui fait de $\E$ un corps valué de dimension $2$ dont on note $\mathscr{O}_\E$ l'anneau des entiers.	
\item Si $0 < r < s$, on définit $\E^{[r, s]}$ comme l'anneau des fonctions analytiques à valeurs dans $L$ sur la couronne $C_{[r, s]} = \{ z \in \cp: \; r \leq v_p(z) \leq s \}$. On a $$ \E^{[r, s]} = \big\lbrace \sum_{k \in \Z} a_k T^k: a_k \in L, \lim_{k \to -\infty} v_p(a_k) + k s = +\infty, \; \mathrm{et} \; \lim_{k \to +\infty} v_p(a_k) + k r = +\infty \big\rbrace. $$ L'anneau $\E^{[r, s]}$ est principal, de Banach pour la valuation $v^{[r,s]}$ définie par $$ v^{[r,s]}(\sum_{k \in \Z} a_k T^k) = \inf_{r \leq v_p(z) \leq s} v_p(f(z)) = \min\{ \inf_{k \in \Z}(v_p(a_k) + kr), \inf_{k \in \Z}(v_p(a_k) + ks) \}.$$
\item Pour $0 < r < s$ on note $$ \E^{]r, s]} = \varprojlim_{t > r} \E^{[t, s]} $$ l'anneau de fonctions analytiques sur la couronne $C_{]r, s]} = \{ z \in \cp: r < v_p(z) \leq s \},$ qui est un anneau de Fréchet pour la famille de valuations $v^{[t, s]}$ pour $t \in ]r, s]$.
\item  Soit $r_n = \frac{1}{p^{n-1}(p - 1)} = v_p(\zeta_{p^n} - 1)$. On note $$ \E^{]0, r_n]} = \E^{]0, r_n]}, \;\;\; \Robba = \varinjlim_{s > 0} \E^{]0, s]} $$ l'anneau de Robba et $\E^\dagger \subseteq \Robba$ son sous-anneau d'éléments bornés. C'est l'anneau des séries de Laurent (resp. des séries de Laurent à coefficients bornés) qui convergent sur une couronne $C_{]0, s]}$ pour $s$ assez petit (qui dépend de chaque fonction). On remarquera que l'on obtient $\E$ et $\Robba$ en complétant $\E^\dagger$, respectivement, par la topologie $p$-adique et par la topologie de Fréchet définie par la famille de normes $v^{[r,s]}$, $0 < r < s$.
\end{itemize}

Si $\mathscr{A}$ est un des anneaux définis ci-haut, on note $\mathscr{A}^+$ son intersection avec $L[[T]]$. On a, par exemple, $\E^+ = \mathscr{O}_L[[T]][\frac{1}{p}]$ et $\Robba^+$ s'identifie à l'anneau des fonctions analytiques sur la boule ouverte unité.  On a une action de $\Gamma$ sur tous les anneaux définis et une action de l'opérateur $\varphi$ sur les anneaux $\E$, $\E^\dagger$ et $\Robba$, définies par les formules $$ \sigma_a(T) = (1 + T)^a - 1 \; \; \; \; \varphi(T) = (1 + T)^p - 1. $$

Tous ces anneaux portent des topologies naturelles pour lesquelles les actions de $\Gamma$ et $\varphi$ sont continues. Posons $\mathscr{A} \in \{ \E, \Robba \}.$ L'anneau $\mathscr{A}$ est muni d'une action de l'opérateur $\psi$: $\mathscr{A}$ est une extension de degré $p$ de $\varphi(\mathscr{A})$ avec une base formée par les éléments $(1 + T)^i$, $i = 0, \hdots, p - 1$, et on pose $\psi(\sum_{i = 0}^{p - 1} (1 + T)^i \varphi (f_i)) = f_0$. Ceci peut être écrit comme $$ \psi = p^{-1} \varphi^{-1} \circ \mathrm{tr}_{\mathscr{A} / \varphi(\mathscr{A})}. $$ L'opérateur $\psi$ ainsi construit est un inverse à gauche de $\varphi$.

Pour $\mathscr{A}$ comme ci-dessus, on note $\Phi \Gamma(\mathscr{A})$ la catégorie des $(\varphi, \Gamma)$-modules sur $\mathscr{A}$.

\subsubsection{Caractères de $\zpe$} Soit $\eta \colon \zpe \to \cpe$ un caractère de conducteur $p^n$. On définit, pour $b \in \zp$, $G(\eta, b) = \sum_{a = 1}^{p^n -1} \eta(a) \zeta_{p^n}^{ab}$ la somme de Gauss tordue et on note $G(\eta) = G(\eta, 1)$. On note $\eta^{-1}$ le caractère de Dirichlet modulo $p^n$, défini par $\eta^{-1}(n) = \eta(n)^{-1}$ pour $n \in (\Z / p^n \Z)^\times$. Rappelons deux résultats classiques de la théorie des caractères:

\begin{prop} Soit $\eta \colon \zpe \to \cpe$ un caractère de conducteur $p^n$. Alors
\begin{itemize}
\item $G(\eta, b) = \eta^{-1}(b) G(\eta, 1)$ pour tout $b \in \zp$.
\item $G(\eta) G(\eta^{-1}) = \eta(-1) p^n.$
\end{itemize}
\end{prop}

Si $f \in \mathrm{LC}_c(\qp, \cp)$ est une fonction localement constante modulo $p^n$ et à support compact, on peut définir sa transformée de Fourier discrète par la formule $$ \hat{f}(x) = p^{-m} \sum_{y \; \mathrm{mod} \; p^m} f(y) e^{- 2 \pi i x y}, $$ où $m$ est un entier arbitraire tel que $m \geq \sup{(n, - v_p(x))}$, et $e^{- 2 \pi i x y}$ est la racine de l'unité d'ordre une puissance de $p$ définie par l'application $\qp \to \qp / \zp \cong \Z[\frac{1}{p}] / \Z$ et par le choix du système $(\zeta_{p^n})_{n \in \N}$ de racines de l'unité: si $k \in \N$ est tel que $a = p^k y x\in \zp$, alors $e^{- 2 \pi i x y} = \zeta_{p^k}^{-a}$. Si $\eta$ est un caractère de Dirichlet de conducteur $p^n$, on a 
\[ \hat{\eta}(x) = \left\{ 
  \begin{array}{l l}
    \frac{1}{G(\eta^{-1})} \eta^{-1}(p^n x) & \quad \text{si $n > 0$}\\
    \mathbf{1}_{\zp}(x) - p^{-1} \mathbf{1}_{p^{-1}\zp}(x) & \quad \text{si $n = 0$.}
  \end{array} \right.\]
En particulier, on observe que $\hat{\eta}$ est à support dans $p^{-n} \zpe$ (resp. $p^{-1}\zp$) si $n > 0$ (resp. si $n = 0$).

\subsubsection{L'espace des poids $p$-adiques} \label{poids}

On note $\mathfrak{X} = \mathrm{Hom}(\zpe, \mathbf{G}_m)$ l'espace des poids $p$-adiques. Il est un espace rigide analytique, dont les $L$ points, pour une extension finie $L$ de $\qp$, paramètrent les caractères continus de $\zpe$ à valeurs dans $L^\times$. Posons $q = p$ si $p > 2$ et $q = 4$ si $p = 2$. On a $\zpe = (\Z / q \Z)^\times \times (1 + q \zp)$. L'application logarithme induit un isomorphisme de groupes $(1 + q \zp, \times) \xrightarrow{\sim} (q \zp, +)$, d'inverse l'application exponentielle. On a un isomorphisme \[ \mathfrak{X} \xrightarrow{\sim} {( (\Z / q \Z)^\times)}^\wedge \times \mathrm{B(1, 1^-)}: \;\;\; \eta \mapsto (\eta|_{(\Z / q \Z)^\times}, \eta(\exp(q))), \] où $B(1, 1^-)$ est la boule unité ouverte centrée en $1$. L'inverse de cet isomorphisme envoie $(\chi, z) \in ( (\Z / q \Z)^\times)^\wedge \times \mathrm{B(1, 1^-)}$ sur le caractère $x \mapsto \chi(\bar{x}) z^{\frac{\log(x)}{q}} \in \mathfrak{X}$, ou $\bar{x}$ dénote la réduction modulo $p$ (resp. $2p$ si $p = 2$) de $x$.

Si $\eta \colon \zpe \to L \in \mathfrak{X}$, on note $z_\eta = \eta(\exp(q)) \in \mathrm{B}(1, 1^-)$ la deuxième coordonné de l'image de $\eta$ par l'isomorphisme ci-dessus et on note $\omega_\eta = \eta'(1) = \frac{\log(z_\eta)}{q}$ son \emph{poids}.


L'espace $\mathfrak{X}$ admet un recouvrement croissant admissible $\mathfrak{X} = \cup_n \mathfrak{X}_n$ par des ouverts affinoïdes $\mathfrak{X}_n = {( (\Z / q \Z)^\times)}^\wedge \times \mathrm{B}(1, p^{-\frac{1}{n}})$, ce qui fait de $\mathfrak{X}$ un espace quasi Stein. On note $\mathcal{O}(\mathfrak{X})$ et $\mathcal{O}(\mathfrak{X}_n)$ les anneaux des fonctions analytiques de ces espaces. On a \[ \mathcal{O}(\mathfrak{X}) = \varprojlim_n \mathcal{O}(\mathfrak{X}_n). \] 

On dispose (\cite{AmiceVelu}) d'un isomorphisme, dû à Amice, \[ \mathscr{D}(\zpe, L) \xrightarrow{\sim} \mathcal{O}_L(\mathfrak{X}), \] envoyant $\mu \in \mathscr{D}(\zpe, L)$ sur la fonction analytique $F_\mu  \in \mathcal{O}_L(\mathfrak{X})$ définie par $F_\mu(\eta) =  \int_\zpe \eta(x) \cdot \mu$. Comme les polynômes sont denses dans l'espace des fonctions localement analytiques, l'isomorphisme précédent montre qu'une fonction $F \in \mathcal{O}(\mathfrak{X})$ s'annulant sur les caractères $x \mapsto x^k$, $k \in \Z$, est identiquement nulle. On exprime ceci en disant que \textit{les caractères de la forme $x \mapsto x^k$, $k \in \Z$, sont Zariski denses dans $\mathfrak{X}$}.

On définit $$ \mathcal{O}(\mathfrak{X}) \widehat{\otimes} \Robba = \varprojlim_{n > 0} \varinjlim_{s>0} \varprojlim_{0 < r < s} \mathcal{O}(\mathfrak{X}_n) \widehat{\otimes} \E^{[r, s]}, $$ où le produit tensoriel à droite c'est le produit tensoriel complété usuel entre deux espaces de Banach. On peut, plus généralement, considérer des distributions à valeurs dans une limite inductive de Fréchets quelconque par les mêmes formules. On dira que $f$ est une fonction analytique sur $\mathfrak{X}$ à valeurs dans $\Robba$ si elle appartient à $\mathcal{O}(\mathfrak{X}) \widehat{\otimes} \Robba$.

\makeatletter
\def\thesection{\Roman{section}}
\def\thesubsection{\thesection.\arabic{subsection}}
\def\thesubsubsection{\Roman{section}.\arabic{subsection}.\arabic{subsubsection}}

\makeatother

\section{La fonction $L$ locale d'un $(\varphi, \Gamma)$-module de de Rham}

Le présent chapitre contient le premier résultat de cet article, encadré dans l'étude des fonctions $L$ $p$-adiques des $(\varphi, \Gamma)$-modules de de Rham associées à un système d'Euler, comme démarré par Perrin-Riou (\cite{PerrinRiou1}). On construit (thm. \ref{propfin}), pour un $(\varphi, \Gamma)$-module $D$ de Rham sur $\Robba$, une extension analytique de l'application Logarithme de Perrin-Riou fournissant, à partir d'un élément dans la cohomologie d'Iwasawa de $D$, une fonction $L$ $p$-adique, définie sur un ouvert de l'espace de poids.

\subsection{Résultat principal} \label{sectresprinc}

On a besoin d'introduire quelques notations pour énoncer le résultat principal de cet article. Si $\xi, \delta \in \mathfrak{X}$ sont deux caractères, on définit leur distance par $v_p(\xi - \delta) = v_p(z_\xi - z_\delta)$ si $\xi|_{(\Z / q\Z)^\times} =  \delta|_{(\Z / q\Z)^\times}$  et $v_p(\xi - \delta) = - \infty$ si non, où $z_\xi$, $z_\delta \in B(1, 1^-)$ sont définis dans \ref{poids}. Si $\eta \in \mathfrak{X}$ est un caractère d'ordre fini, on note $\mathrm{c}(\eta)$ la $p$-partie de son conducteur (de sorte que $\mathrm{cond}(\eta) = p^{\mathrm{c}(\eta)}$) et, pour $N \geq -\infty$, on définit \[ \mathfrak{B}(\eta, N) = \{ \xi \in \mathfrak{X} \; | \;  v_p( \xi - \eta )  > p^{N - \mathrm{c}(\eta)} \} \subseteq \mathfrak{X}. \] Rappelons que $\Gamma^*(j)$ dénote le coefficient principal de la série de Laurent de la fonction $\Gamma(s)$ en $s = j$. Notons, pour $D \in \Phi\Gamma(\Robba)$ de Rham et $\mu \in H^1_{\rm Iw}(\qp, D)$,
$$ \log(\int_{\Gamma} \eta \chi^{-j} \cdot \mu) = \left\{ 
\begin{array}{l l}
  \exp^*(\int_{\Gamma} \eta \chi^{-j} \cdot \mu) & \quad \text{si ${\j} \geq 0$}\\
  \exp^{-1} (\int_{\Gamma} \eta \chi^{-j} \cdot \mu) & \quad \text{si ${\j} \ll 0$.}\\ \end{array} \right. $$ Enfin, on note $r_{m(D)}$, $m(D) \in \N$, le rayon de surconvergence de $D$ (cf. \S \ref{generalites} ci-dessous). Le résultat principal de ce chapitre est le suivant:

\begin{theorem} \label{propfina}
Soient $D \in \Phi\Gamma(\Robba)$ de Rham à poids de Hodge-Tate positifs, $z \in D^{\psi = 1}$, et notons $\mu_z = \mathrm{Exp}^*(z) \in H^1_{\rm Iw}(\qp, D)$. Il existe un entier $N(D)$ \footnote{L'entier $N(D)$ ne dépend que du module $\D_{\rm pst}(D)$ et est borné en termes du conducteur de la plus petite extension galoisienne $K$ de $\qp$ tel que l'action de $\mathscr{G}_\qp$ sur $\D_{\rm pst}(D)$ se factorise à travers $\mathscr{G}_K$. Si $D = \Drig(V)$, où $V \in \mathrm{Rep}_L \mathscr{G}_\qp$, la constante $N(D)$ est donc bornée en termes du conducteur de la plus petite extension $K / \qp$ sur laquelle $V$ dévient semi-stable.}, et une unique fonction rigide analytique $\Lambda_{D, z} \in \mathcal{O}(\mathfrak{U}_D) \otimes \DdR(D)$, où l'on a posé $\mathfrak{U}_D = \cup_{\mathrm{c}(\eta) > m(D)} \mathfrak{B}(\eta, N(D))$, telle que, pour tout $\eta \chi^j \in \mathfrak{U}_D$, où $\eta$ est un caractère de conducteur $p^n$, $n > 0$, et $j \in \Z$ est tel que $j \geq 0$ ou $j \ll 0$ \footnote{Plus précisément, $j$ doit être suffisamment petit de sorte que $\exp_{D(-j)} \colon \DdR(D(-j)) \to H^1_{\varphi, \gamma}(D(-j))$ soit un isomorphisme. Il suffirait donc de demander $j < h_0 - 2$, où $h_0$ dénote le plus petit poids de Hodge-Tate de $D$ (cf. \cite[Thm. 0.9]{Berger02}).}, on a
\[ \Lambda_{D, z} (\eta \chi^{\j}) = \Gamma^*(j + 1) p^{n(j + 1)} \cdot \log (\int_{\Gamma} \eta \chi^{-j} \cdot \mu_z)  \otimes \mathbf{e}^{\rm dR, \vee}_{\eta, -j} . \]
\end{theorem}

Voici quelques remarques:

\begin{itemize}
\item On obtiendra ce théorème (thm. \ref{propfin}) comme conséquence d'un résultat un peu plus général énoncé dans le théorème \ref{theo1}.

\item Comme les caractères $\eta x^j$ sont Zariski denses dans $\mathfrak{X}$, la fonction $\Lambda_{D, z}$ est unique.


\item Si $D$ est cristallin, alors l'application $\Delta_{D,z}$ provient par restriction d'une fonction rigide analytique définie sur tout l'espace des poids $\mathfrak{X}$.

\item Comme corollaire de ce théorème, on obtiendra une construction partielle de la fonction $L$ $p$-adique associée à un système d'Euler d'un $(\varphi, \Gamma)$-module de de Rham. Si $D = \D_{\rm rig}(V)$ est le $(\varphi, \Gamma)$-module associé à la représentation $p$-adique d'une forme modulaire, on verra comment ces valeurs s'interprètent en termes de valeurs spéciales de la fonction $L$ complexe de la forme modulaire.
\end{itemize}

Voici un bref résumé de ce chapitre. On commence par rappeler les outils et notations nécessaires dont on aura besoin pour la preuve du théorème. On pourra consulter \cite{Nakamura} et \cite{Berger03}, qui sont les références principales. La structure de la preuve du théorème \ref{propfina} est la suivante: Dans la section \ref{generalites}, on rappelle des généralités sur les $(\varphi, \Gamma)$-modules de de Rham sur l'anneau de Robba. Ensuite (\S \ref{prolan}) on définit la multiplication analytique par un caractère sur un $(\varphi, \Gamma)$-module. Dans \ref{interpolation}, on définit la fonction $\Lambda_{D, z}$. Le Lemme \ref{lemme6} calcule le rayon de convergence de $\Lambda_{D, z}$, ce qui explique la définition de l'ouvert $\mathfrak{B}(N)$ du théorème. Les propositions \ref{lemme4} et \ref{lemme5} montrent, en utilisant la loi de réciprocité explicite de Perrin-Riou comme démontrée par Nakamura, les propriétés d'interpolation de $\Lambda_{D, z}$, introduisant l'opérateur différentiel auxiliaire $\nabla_h$ dans les formules. Enfin, quand $D$ est à poids de Hodge-Tate positifs et $z \in D^{\psi = 1} \subseteq \Nrig(D)^{\psi = 1}$ on se débarrasse (\S \ref{poidspositifs}) de l'opérateur différentiel $\nabla_h$ pour obtenir ainsi l'interpolation voulue.

\subsection{Généralités sur les $(\varphi, \Gamma)$-modules} \label{phigammagen}

Notons $\Phi\Gamma(\Robba)$ la catégorie de $(\varphi, \Gamma)$-modules sur $\Robba$.

\subsubsection{Sous-modules naturels de $D$} \label{generalites}

Soit $D \in \Phi\Gamma(\Robba)$ de rang $d$. L'algèbre de Lie de $\Gamma$ agit sur $D$ (cf. \cite[\S 5.1]{Berger02}) via l'opérateur $L$-linéaire $$ \nabla = \lim_{a \to 1} \frac{\sigma_a - 1}{a - 1} = \frac{\log (\gamma)}{\log (\chi(\gamma))} = \frac{1}{\log (\chi(\gamma))} \sum_{i = 1}^{+\infty} (-1)^{i + 1} \frac{(\gamma - 1)^i}{i}, $$ où $\gamma \in \Gamma$ dénote n'importe quel élément sans torsion de $\Gamma$, ce qui définit un opérateur différentiel au-dessus de l'opérateur $\nabla = t (1 + T) \frac{d}{dT}$ agissant sur $\Robba$. 

\begin{lemme} [{\cite[Thm. I.3.3]{Berger08}}] \label{subm1}
Il existe $\epsilon > 0$ et, pour tout $0 < s \leq \epsilon$, des uniques sous-$\E^{]0, s]}$-modules $D^{]0, s]}$ de $D$, satisfaisant
\begin{itemize}
\item $D = \Robba \otimes_{\E^{]0, s]}} D^{]0, s]}.$
\item $\varphi(D^{]0, s]}) \subseteq D^{]0, s/p]}$, et on a un isomorphisme $\E^{]0, s/p]} \otimes_{\E^{]0, s]}} D^{]0, s]} \to D^{]0, s/p]}: f \otimes x \mapsto f \varphi(x)$, pour tout $n$ tel que $r_n \leq \epsilon$.
\end{itemize}
De plus, les modules $D^{]0, s]}$ sont stables par $\Gamma$ et $\nabla$.
\end{lemme}


On définit $m(D)$ comme le plus petit entier tel que le lemme \ref{subm1} est vrai avec $r_{m(D)} \leq \epsilon$ et on dit que $r_{m(D)}$ est le rayon de surconvergence de $D$. Pour $0 < r < s \leq r_{m(D)}$, on pose \[ D^{[r, s]} = \E^{[r, s]} \otimes_{\E^{]0, s]}} D^{]0, s]}. \] On a alors \[ D = \varinjlim_{s > 0} \varprojlim_{0 < r < s} D^{[r, s]}, \] ce qui montre que $D$ est un espace de type $\mathrm{LF}$ (i.e limite inductive d'espaces de Fréchet). 

Rappelons que l'on a des morphismes de localisation \[ \varphi^{-n} \colon \E^{]0, r_n]} \hookrightarrow L_n[[t]] \] envoyant $T$ sur $\zeta_{p^n} e^{t / p^n} - 1$. Pour $n \geq m(D)$ on définit \[ \D^+_{\mathrm{dif}, n}(D) = L_n[[t]] \otimes_{\varphi^{-n}, \E^{]0, r_n]}} D^{]0, r_n]}, \] \[ \D_{\mathrm{dif}, n}(D) = L_n((t)) \otimes_{L_n[[t]]} \D^+_{\mathrm{dif}, n}(D), \] qui sont des $L_n[[t]]$ et $L((t))$-modules, respectivement, libres de rang $d$ et munis d'une action semi-linéaire de $\Gamma$. Finalement, on définit $$ \D_{\mathrm{dif}}(D) = \varinjlim_n \D_{\mathrm{dif}, n}(D) \; \;\; \D^+_{\mathrm{dif}}(D) = \varinjlim_n \D^+_{\mathrm{dif}, n}(D), $$ qui sont, respectivement, des $L_\infty((t)) = \cup_n L_n((t))$ et $L_\infty[[t]] = \cup_n L_n[[t]]$-modules libres de rang  $d$.

\subsubsection{Théorie de Hodge $p$-adique}

Comme l'on a déjà remarqué, la plupart des objets de la théorie de Hodge $p$-adique peuvent être exprimés purement en terme des $(\varphi, \Gamma)$-modules. Suivant ce programme, commencé par Fontaine \cite{Fontaine90}, on définit les invariants suivants:

\begin{definition}
Soit $D \in \Phi\Gamma(\Robba)$ de rang $d$. On définit $$ \Dcris(D) = (D[1 / t])^\Gamma = (D \otimes_\Robba \Robba[1/t])^\Gamma \;\;\;\; \DdR(D) = (\D_{\mathrm{dif}}(D))^\Gamma, $$ qui sont des $L$-espaces vectoriels de dimension finie.
\end{definition}

On munit $\DdR(D)$ de sa filtration de Hodge, donnée par $ \mathrm{Fil}^i \, \DdR(D) = \DdR(D) \cap t^i \D_{\mathrm{dif}}^+ = (t^i \D_{\mathrm{dif}}^+)^\Gamma$. On observe que $\Dcris(D)$ est muni d'une action bijective du Frobenius $\varphi$ ainsi que d'une filtration induite par l'inclusion $\Dcris(D) \subseteq \DdR(D)$ définie par $x \in \Dcris(D) \mapsto \iota_n (\varphi^n(x)) \in \DdR(D)$, où on a noté $\iota_n = \varphi^{-n} \colon D^{]0, r_n]}[1/t] \to \Ddif(D)$ l'application de localisation \footnote{Il existe ici un petit abus évident en notant par $\varphi^{-n}$ deux applications différentes, l'une étant l'application de localisation notée usuellement $\iota_n$ et l'autre l'inverse de l'opérateur $\varphi$ agissant sur $\Dcris(D)$, mais cela ne devrait pas causer de problèmes de lecture.}. On a $$ \dim_L \Dcris(D) \leq \dim_L \DdR(D) \leq \mathrm{rang}_\Robba \; D,  $$ où la première inégalité est évidente par ce qui précède et la dernière suit en remarquant que $\Ddif(D)$ est un $L_\infty((t))$-espace vectoriel de rang $d = \mathrm{rang}_\Robba \; D$ et $\DdR(D) = (\Ddif(D))^\Gamma$ est donc un $(L_\infty((t)))^\Gamma = L$-espace vectoriel de rang $\leq d$.  

\begin{definition}
Soit $D$ un $(\varphi, \Gamma)$-module sur $\Robba$. On dit que $D$ est \textit{cristallin} (resp. \textit{de Rham}) si l'inégalité $\dim_\qp \Dcris(D) \leq \mathrm{rang}_\Robba \; D$ (resp. $\dim_\qp \DdR(D) \leq \mathrm{rang}_\Robba \; D$) est une égalité.
\end{definition}

Si $D$ est de Rham, on définit ses \textit{poids de Hodge-Tate} comme les opposés des entiers où la filtration change, c'est-à-dire l'ensemble $\{ -h \in \Z: \mathrm{Fil}^h \, \DdR(D) / \mathrm{Fil}^{h + 1} \, \DdR(D) \neq 0 \}$.

\subsubsection{L'équation différentielle $p$-adique $\Nrig(D)$} \label{eqdiff}

Rappelons la construction de Berger (cf. \cite{Berger02}, \cite{Berger08}) de l'équation différentielle $p$-adique $\Nrig(D)$ associée à un  $(\varphi, \Gamma)$-module de de Rham $D$ sur $\Robba$.

\begin{proposition} [{\cite[Thm. III.2.3]{Berger08}}] \label{Ndifprop}
Soit $D \in \Phi\Gamma(\Robba)$ de rang $d$, de Rham, et, pour chaque $n \geq m(D)$, posons $$ \mathbf{N}_{\mathrm{rig}}^{]0, r_n]}(D) = \{ x \in D^{]0, r_n]}[1/t] : \text{ $\varphi^{-m}(x) \in L_m[[t]] \otimes_L \DdR(D)$ pour tout $m \geq n$} \}, $$ et $\mathbf{N}_{\mathrm{rig}}(D) = \varinjlim_n \mathbf{N}_{\rm rig}^{]0, r_n]}(D)$. Alors, $\Nrig(D)$ est un $(\varphi, \Gamma)$-module sur $\Robba$, de rang $d$, qui satisfait
\begin{itemize}
\item $\Nrig(D)[1/t] = D[1/t]. $
\item $\mathbf{D}_{\mathrm{dif}, n}^+(\Nrig(D)) = L_n[[t]] \otimes_L \mathbf{D}_{\mathrm{dR}}(D)$ pour tout $n \geq m(D)$.
\item $\nabla(\Nrig(D)) \subseteq t \Nrig(D).$
\end{itemize}
\end{proposition}

Le $(\varphi, \Gamma)$-module $\Nrig(D)$ ainsi obtenu est de Rham à poids de Hodge-Tate tous nuls. Remarquons que l'on peut reconstruire $D$, à partir de la donnée de $\Nrig(D)$ et de la filtration de Hodge sur $\DdR(D)$, en utilisant la formule $$ D = \{ x \in \Nrig(D)[1/t] \; : \; \varphi^{-n}(x) \in \mathrm{Fil}^0(L_n((t)) \otimes_\qp \DdR(D)) \;\;\; \forall n \gg 0 \}. $$ Si les poids de Hodge-Tate de $D$ sont contenus dans $[a, b]$, on a des inclusions $t^{-a} D \subseteq \Nrig(D) \subseteq t^{-b} D$.


La troisième propriété de la proposition \ref{Ndifprop} caractérisant $\Nrig(D)$ permet de définir un opérateur différentiel $$ \partial := \frac{1}{t} \nabla \colon \Nrig(D) \to \Nrig(D), $$ satisfaisant les identités $\partial \circ \varphi = p \; \varphi \circ \partial$ et $\partial \circ \sigma_a = a \; \sigma_a \circ \partial $. Si $D$ est un $(\varphi, \Gamma)$-module de de Rham sur $\Robba$, on notera $\Delta = \Nrig(D)$.

\subsubsection{Les anneaux de Fontaine}

Rappelons la construction des anneaux de Fontaine associés à une extension galoisienne finie $K$ de $\qp$. Notons $K_n = K F_n = K(\mu_{p^n})$, $n \geq 1$, $K_\infty = K F_\infty = \cup_{n \geq 1} K_n$, $K_0 = K \cap \mathbf{Q}_p^{\mathrm{nr}}$ la plus grande sous-extension de $K$ non ramifiée et $K'_0$ la plus grande extension non ramifiée de $K_0$ dans $K_\infty$. Notons $\mathscr{H}_K = \mathrm{Gal}(\overline{\Q}_p / K_\infty)$, $\Gamma_K = \mathrm{Gal}(K_\infty / K)$.

La théorie du corps des normes (cf., par exemple, \cite{ColmezEV} ou \cite[\S I.2]{Berger08}) permet de construire des extensions étales $\E_K^\dagger / \E_\qp^\dagger$ de degré $[K_\infty: F_\infty]$, munies d'une action du Frobenius $\varphi$ et du groupe $\Gamma_K$. Plus précisément, soit $\widetilde{\mathbf{E}} = \varprojlim_{x \mapsto x^p} \C_p = \C_p^\flat$ le corps basculé de $\cp$. Il est un corps de caractéristique $p$, algébriquement clos et muni d'une valuation $v_\mathrm{E}(x) = v_p(x^{(0)})$ pour laquelle il est complet. Notons $1 \neq \epsilon = (1, \epsilon^{(1)}, \hdots) \in \widetilde{\mathbf{E}}$, et $\widetilde{\mathbf{E}}^+$ l'anneau des entiers de $\widetilde{\Ebf}$, qui s'identifie à la limite projective $\varprojlim_{x \mapsto x^p} \mathscr{O}_{\C_p} / \mathfrak{a}$ pour n'importe quel idéal $\mathfrak{a} \subseteq \mathscr{O}_{\C_p}$ contenant $p$ et différent de l'idéal maximal de $\mathscr{O}_{\C_p}$. Soient $\mathbf{E}_\qp = \F_p((\epsilon - 1))$ et $\mathbf{E}$ la clôture séparable de $\mathbf{E}_\qp$ dans $\widetilde{\mathbf{E}}$. La théorie du corps des normes permet aussi de montrer que $\mathrm{Gal}(\mathbf{E} / \mathbf{E}_\qp) \cong \mathscr{H}_\qp$. Si $K / \qp$ est une extension finie galoisienne de $\qp$, on pose $\mathbf{E}_K = \mathbf{E}^{\mathscr{H}_K}$. On dispose d'une application bien définie et injective $\varprojlim_{x \mapsto N_{K_n / K_{n-1}}(x)} \mathscr{O}_{K_n} \to \widetilde{\mathbf{E}}^+$ d'image l'anneau des entiers $\mathbf{E}_K^+$ de $\mathbf{E}_K$, qui fournit une uniformisante $\overline{\pi}_K$ de l'extension $\mathbf{E}_K / \mathbf{E}_\qp$ (l'image d'un système compatible $(\omega_n)_{n \in \N}$ où $\omega_n$ est une uniformisante de $K_n$ pour $n$ assez grand). Soit $\overline{P}(X) = X^{d_{K_\infty}} + \overline{a}_{{d_{K_\infty}} - 1} X^{d_{K_\infty} - 1} + \hdots + \overline{a}_0 \in \mathbf{E}_\qp[X]$, où $d_{K_\infty} = [K_\infty : F_\infty]$ et $\overline{a}_i \in \mathbf{E}_\qp$, le polynôme minimal de $\overline{\pi}_K$ sur $\mathbf{E}_\qp$. Enfin, si $P \in \zp[[T]][X]$ est tel que sa réduction modulo $p$ évaluée en $T = \epsilon - 1$ est $\overline{P}(X)$, alors $\E^\dagger_K = \E^\dagger_\qp [X] / P(X)$ et on note $\pi_K$ l'image de $X$ dans ce quotient, dont la réduction modulo $p$ est $\overline{\pi}_K$.

L'anneau $\E_K^\dagger$ s'identifie à l'anneau des séries formelles $f(T_K) = \sum_{k \in \Z} a_k T_K^k$, $a_k \in K'_0$, à coefficients bornés, qui convergent sur une couronne $0 < v_p(T_K) \leq r$ pour un $r$ assez petit (qui dépend de $f$). On peut de la sorte définir les anneau $\E_K^{[r, s]}$, qui s'identifient à l'espace des séries de Laurent à coefficients dans $K'_0$, convergentes sur la couronne $C_{[r / e, s / e]} = \{ z \in \cp: \; r / e \leq v_p(z) \leq s / e \}$, où $e$ dénote l'indice de ramification de l'extension $K_\infty / F_\infty$, muni de la norme spectrale $v^{[r, s]}$, ainsi que $\E^{]0, s]}_K$, $\E^{\dagger, r}_K$, etc. En complétant $\E^\dagger_K$ pour la famille de normes $v^{[r,s]}$, on obtient une extension étale $\Robba_K / \Robba_\qp$ et on a $$ \mathrm{Gal}(\Robba_K / \Robba_\qp) = \mathrm{Gal}(\E^\dagger_K / \E^\dagger_\qp) = \mathrm{Gal}(K_\infty / F_\infty). $$

On a un opérateur de dérivation $\partial$ agissant sur $\Robba_K$: si $K / \qp$ est une extension non ramifiée, on a $T_K = T$ et $\partial f(T) = (1 + T) f'(T)$ pour $f(T) \in \Robba_K$; si $K / \qp$ est ramifiée et $P \in \zp[[T]][X]$ dénote le polynôme minimal de $T_K \in \E^\dagger_K$ sur $\E^\dagger_\qp$ comme ci-dessus, l'identité $P(T_K) = 0$ et le fait que $\partial$ est une dérivation donnent la formule \[ \partial  T_K = - P'(T_K)^{-1} (\partial  P)(T_K), \] où, si $P = \sum f_i X^i \in \zp[[T]][X]$, $\partial P = \sum \partial f_i \cdot X^i$.

Enfin, soit $\ell_T = \log (T)$ une variable formelle et étendons les actions $\varphi$ et $\Gamma_K$ sur $\Robba_K$ à des actions sur $\Robba_K[\ell_T]$ par les formules $$ \varphi(\ell_T) = p \ell_T + \log (\varphi(T) / T^p), \;\;\; \gamma(\ell_T) = \ell_T + \log (\gamma(T) / T). $$ La dérivation $\partial$ agit sur $\ell_T$ par $\partial \ell_T = T^{-1} \partial T = 1 + T^{-1}$. On définit un opérateur de monodromie sur $\Robba_K[\ell_T]$ comme la dérivation $\Robba_K$-linéaire $N$ telle que $N(\ell_T) = -p / (p - 1)$.


\subsubsection{Théorème de monodromie $p$-adique et surconvergence}

Soit \[ \mathfrak{d}_{\Ebf_K / \Ebf_\qp} \subseteq \Ebf_K \] la différente de l'extension $\Ebf_K / \Ebf_\qp$ et posons $\delta_K = d_{K_\infty} \cdot v_\Ebf(\mathfrak{d}_{\Ebf_K / \Ebf_\qp}) \in \N$, où $d_{K_\infty} = [\Ebf_K : \Ebf_\qp] = [K_\infty: F_\infty]$ comme plus haut. Rappelons (cf. \cite[Prop. 4.12]{ColmezEV}) que, si $c(K)$ dénote le conducteur de $K$ \footnote{Le conducteur de $K$ est la borne inférieure de l'ensemble des $t$ tels que le groupe de ramification supérieure $\mathscr{G}_\qp^{(t)}$ agit trivialement sur $K$.} et $n \geq c(K) + 1$ est un entier, alors $[K_n:F_n] = d_{K_\infty}$ et $\delta_K = d_{K_\infty} p^n v_p(\mathfrak{d}_{K_n / F_n})$. En particulier, si $K$ a suffisamment de racines de l'unité, dans la terminologie de \cite{CN}, alors $v_K = [K:\qp] v_p(\mathfrak{d}_{K/\qp})$. On aura besoin des faits suivants:

\begin{lemme} [{\cite[Lem. 2.17]{CN}}] \label{CN} Si $s < (\delta_K + 1)^{-1}$, alors $\partial T_K \in T^{-\delta_K} \mathscr{O}_{\E^{\dagger, s}_K}^\times$ et on a $v^{[r, s]}(\partial T_K) \geq -1$ pour tout $0 < r < s$.
\end{lemme}

Soit $D \in \Phi \Gamma(\Robba)$ de Rham et notons $\Delta = \Nrig(D)$. D'après le théorème de monodromie $p$-adique (cf. \cite[Thm. III.2.1]{Berger08}), il existe une extension galoisienne finie $K$ de $\qp$ tel que l'on a un isomorphisme $$ \Robba_K[\ell_T] \otimes_{K_0} (\Robba_K[\ell_T] \otimes_{\Robba} \Delta)^{\Gamma_K} \cong \Robba_K[\ell_T] \otimes_{\Robba} \Delta. $$ On note \[ D_{\mathrm{pst}} = (\Robba_K[\ell_T] \otimes_{\Robba} \Delta)^{\Gamma_K} \] l'espace des $\Gamma_K$-solutions de $\Delta$. C'est un $(\varphi, N, \mathrm{Gal}(K / \qp))$-module filtré (la filtration dépend de la donnée de $D$). Le résultat suivant permet de reconstruire $\Delta$ à partir de $D_{\rm pst}$. 

\begin{proposition}[{\cite[Thm. III.2.1, Thm. C]{Berger08}}] \label{monodrom} Soient $D$ et $D_{\mathrm{pst}}$ comme ci-dessus. Si $K$ est une extension galoisienne finie de $\qp$ telle que $\mathscr{G}_\qp$ agit sur $D_{\rm pst}$ à travers $\mathrm{Gal}(K / \qp)$, alors $$ \Delta = (\Robba_K[\ell_T] \otimes_{K_0} D_{\mathrm{pst}})^{\mathrm{Gal}(K_\infty / F_\infty), N = 0}, $$ le groupe $\mathrm{Gal}(K_\infty / F_\infty)$ agissant sur $D_{\mathrm{pst}}$ à travers $\mathrm{Gal}(K / (K \cap F_\infty)) \subseteq \mathrm{Gal}(K / \qp)$, et sur $\Robba_K$ à travers $\mathrm{Gal}(K_\infty / F_\infty) = \mathrm{Gal}(\Robba_K / \Robba_\qp)$ (il agit trivialement sur l'élément $\ell_T$).
\end{proposition}

On récupère l'action de $\Gamma$ via l'action résiduelle du groupe $\Gamma_K$ sur le module $\Delta$, l'action de $\varphi$ via son action (diagonale) sur le produit tensoriel et l'action de l'opérateur $\partial$ à travers celle sur $\Robba_K[l_K]$. Le lemme suivant sera très utile pour nos constructions futures.

\begin{lemme} \label{lemme11}
Il existe, pour tout  $0 < s < (\delta_K + 1)^{-1} $, des sous-$\E_\qp^{]0, s]}$-modules $\Delta^{]0, s]}$ de $\Delta$, munis d'une famille de valuations $(v^{[r, s]})_{0 < r < s}$, satisfaisant les conditions du lemme \ref{subm1}, et de sorte que, pour tout $0 < r < s$, la norme $v^{[r, s]}$ de l'opérateur $\partial$ sur $\Delta^{]0, s]}$ soit $\geq - s - 1$.
\end{lemme}

\begin{proof}
Débarrassons-nous d'abord de l'opérateur $N$. Le module $M = K_0[\ell_T] \otimes_{K_0} D_{\rm pst}$ est un $K_0[\ell_T]$-module libre de rang $d$ muni d'une action de $G = \mathrm{Gal}(K_\infty / F_\infty)$ (agissant sur le facteur de droite) et de $N$ (agissant diagonalement) et on affirme qu'il contient un système libre de générateurs $w_1, \hdots, w_d$ tués par $N$. En effet, notons $\alpha = - \frac{p}{p-1}$ la valeur de $N$ sur $\ell_T$ et posons, pour $v_1, \hdots, v_d$ une base quelconque de $D_{\rm pst}$, \[ w_i = \sum_{k = 0}^{+\infty} (-1)^k \frac{\alpha^{-k}}{k!} \ell_T^k \otimes N^k v_i , \;\;\; 1 \leq i \leq d, \] chaque somme étant une somme finie car $N$ est nilpotent. Alors les $w_i$ sont par construction tués par $N$ et ils forment un système libre de générateurs de $M$, car l'opérateur $ \sum_{k = 0}^{+\infty} (-1)^k \frac{\alpha^{-k}}{k!} \ell_T^k \otimes N^k$ est inversible, $N$ étant nilpotent. Le module $W = (K_0[\ell_T] \otimes D_{\rm pst})^{N = 0}$ est donc un $(K_0[\ell_T])^{N = 0} = K_0$-espace vectoriel de dimension $d$, muni d'une action de $G$, et on a \[ K_0[\ell_T] \otimes_{K_0} W = K_0[\ell_T]  \otimes_{K_0} D_{\rm pst}. \]

Soit $v_1, \hdots, v_d$ une base de $D_{\rm pst}$ de sorte que la matrice de l'opérateur $N$ soit à coefficients entiers \footnote{Ceci est possible grâce au fait que $N$ est nilpotent: si $v_1, \hdots, v_d$ est une base de $D_{\rm pst}$ de sorte que la matrice de $N$ soit triangulaire supérieure avec des zéros sous la diagonale, alors on peut choisir des entiers $0 < n_2 < n_3 < \hdots < n_d$ de sorte que la matrice de $N$ dans la base $v_1, p^{n_2} v_2, \hdots, p^{n_d} v_d$ soit à coefficients entiers comme on voulait.}, et soient $w_1, \hdots, w_d$ les éléments associés à cette base formant une base de $W$ fournis par le paragraphe précédent. Posons, pour $0 < s < (\delta_K + 1)^{-1}$,  \[ \Delta^{]0, s]} = (\E_K^{]0, s]} \otimes W)^G \subseteq \E_K^{]0, s]} \otimes W \subseteq \Robba_K \otimes W. \]

On a $G \cong \mathrm{Gal}(\Robba_K / \Robba_\qp)$, et le théorème de Hilbert $90$ implique que $H^1(G, \mathrm{GL}_d(\Robba_K)) = 1$. Le module $W$ est donc $\Robba_K$-admissible au sens de Fontaine et on a un isomorphisme $$ \Robba_K \otimes_{\Robba_{\qp}} (\Robba_K \otimes_{K_0} W)^G \cong \Robba_K \otimes_{K_0} W, $$ induit par $f \otimes (g \otimes x) \mapsto f g \otimes x$, de $\Robba_K$-modules topologiques de rang $d$ munis d'une action de $G$. Cet isomorphisme est en outre compatible avec les structures présentes sur $\Robba_K$. Cela veut dire en particulier que le $\Robba_\qp$-module $(\Robba_K \otimes_{K_0} W)^G$ est de dimension $d$.

Si $s < (\delta_K + 1)^{-1}$, alors $\Robba_K = \Robba_\qp \otimes_{\E^{]0, s]}_\qp} \E_K^{]0, s]}$ \footnote{Comme on le voit, par exemple, en notant que, si $s < (\delta_K + 1)^{-1}$, $\E_K^{]0, s]}$ est un $\E_\qp^{]0, s]}$-module libre de rang $d_{K_\infty}$, dont une base est formée par $1, T_K, \hdots, T^{d_{K_\infty} - 1}$, et donc $\Robba_\qp \otimes_{\E^{]0, s]}_\qp} \E_K^{]0, s]} = \oplus_i \Robba_\qp \otimes_{\E^{]0, s]}_\qp} \E_\qp^{]0, s]} \cdot T_K^i = \oplus_i \Robba_\qp \cdot T_K^i = \Robba_K$. }, d'où \[ (\Robba_K \otimes_{K_0} W)^G = (\Robba_\qp \otimes_{\E^{]0, s]}_\qp} \E_K^{]0, s]} \otimes_{K_0} W)^G = \Robba_\qp \otimes_{\E^{]0, s]}_\qp} (\E_K^{]0, s]} \otimes_{K_0} W)^G, \] ce qui montre que \[ \Delta = \Robba \otimes_{\E^{]0, s]}} \Delta^{]0, s]}. \] De plus, le fait que $\E_K^{]0, s/p]} \otimes_{\E_K^{]0, s]}} \varphi(\E_K^{]0, s]}) = \E_K^{]0, s/p]}$ et les formules pour l'action de $\varphi$ sur $\ell_T$ montrent que la deuxième condition du lemme \ref{subm1} est satisfaite.

Les modules $\E_K^{]0, s]} \otimes W$ sont des $\E^{]0, s]}_\qp$-modules libres, dont une base est donnée par les éléments $(T_K^i \otimes w_j)_{0 \leq i \leq d_\infty, 1 \leq j \leq d}$, et ils sont munis d'une famille $(v^{[r, s]})_{0 < r < s}$ de valuations définies par \[ v^{[r, s]}(\sum_{i, j} f_{ij} \cdot T_k^i \otimes w_j) = \inf_{i, j} v^{[r, s]}(f_{ij}) \;\;\; (f_{ij} \in \E_\qp^{]0, s]}). \]

Calculons la norme de l'opérateur $\partial$ sur la base décrite ci-dessus pour une de ces valuations. On a $\partial(\ell_T) = 1 + T^{-1}$ et donc
\begin{eqnarray*}
\partial(w_i) &=& \partial( \sum_{k = 0}^{+\infty} (-1)^k \frac{\alpha^{-k}}{k!} \ell_T^k \otimes N^k v_i) \\
&=& (1 + T^{-1}) \sum_{k = 1}^{+\infty} (-1)^k \frac{\alpha^{-k}}{(k-1)!} \ell_T^{k-1} \otimes N^k v_i \\
&=& - \alpha^{-1} (1 + T^{-1}) \sum_{k = 0}^{+\infty} (-1)^k \frac{\alpha^{-k}}{k!} \ell_T^k \otimes N^k (N v_i) \\
&=& \sum_{l = 1}^d - a_l \alpha^{-1} (1 + T^{-1}) \sum_{k = 0}^{+\infty} (-1)^k \frac{\alpha^{-k}}{k!} \ell_T^k \otimes N^k v_l \\
&=& \sum_{l = 1}^d - a_l \alpha^{-1} (1 + T^{-1}) w_l, \\
\end{eqnarray*}
où $N v_i = \sum_{l = 1}^d a_l \cdot v_l$, avec $v_p(a_l) \geq 0$ pour tout $l$, par le choix de la base $v_1, \hdots, v_d$ de $D_{\rm pst}$. On en déduit que, si $A$ dénote l'anneau des entiers de $K_0$ et $W_0$ dénote le $A$-module engendré par les éléments $w_1, \hdots, w_d$, alors (rappelons que $\alpha = - \frac{p}{ p - 1}$) \[ \partial(W_0) \subseteq p^{-1} (1 + T^{-1})\cdot W_0. \] On en déduit
\begin{eqnarray*}
v^{[r, s]}(\partial(T_K^i \otimes w_j)) &=& v^{[r, s]}( i \partial(T_K) T_K^{i-1} \otimes w_j + T_K^{i} \otimes \partial w_j) \\
&\geq & \inf\{ v^{[r, s]}( i \partial(T_K) T_K^{i-1} \otimes w_j), v^{[r, s]}(p^{-1} (1 + T^{-1})) \} \\
&\geq & \inf\{ -1 , -1 - s \} \\
& = & - s - 1,
\end{eqnarray*}
où dans la première égalité on a utilisé la définition de $\partial$, dans la deuxième l'inégalité triangulaire et l'inclusion $\partial(W_0) \subseteq p^{-1} (1 + T^{-1})\cdot W_0$ et dans la troisième l'estimation du lemme \ref{CN} et le fait que $ v^{[r, s]}(p^{-1} (1 + T^{-1})) = -1 - s$.

Vu que l'on a borné la norme de l'opérateur $\partial$ sur une base de $\E_K^{]0, s]} \otimes W$, un petit calcul montre que \[ v^{[r, s]}(\partial z) \geq v^{[r, s]}(z) - s - 1 \] pour tout $z \in \E_K^{]0, s]} \otimes W$. En effet, écrivons $z = \sum_{i, j} f_{ij} \cdot T_k^i \otimes w_j$, $f_{ij} \in \E^{]0, s]}_\qp$,  de sorte que $\partial(z) = \sum_{i, j} \partial f_{ij} \cdot T_k^i \otimes w_j + f_{ij} \cdot \partial(T_k^i \otimes w_j)$. On a alors
 \begin{align*}
  v^{[r, s]}(\partial z) &\geq \inf \big\{ v^{[r, s]}(\sum_{ij} (\partial f_{ij}) (T_K^i \otimes w_j)), v^{[r, s]}( \sum_{ij} f_{ij} \partial(T_K^i \otimes w_j)) \big\} \\
  &= \inf \big\{ \inf_{ij} v^{[r, s]}(\partial f_{ij}), v^{[r, s]}( \sum_{ij} f_{ij} \partial(T_K^i \otimes w_j)) \big\} \\
  &\geq \inf \big\{ \inf_{ij} v^{[r, s]}(f_{ij}) - s, \inf_{ij} v^{[r, s]}( f_{ij}) - s - 1 \big\} \\
  &= \inf_{ij} v^{[r, s]}(f_{ij}) - s - 1 = v^{[r, s]}(z) - s - 1,
 \end{align*}
 où la première ligne suit de l'inégalité triangulaire, la deuxième suit de la définition de $v^{[r, s]}$ sur $\E_K^{]0, s]} \otimes W$, la troisième de l'inégalité (sur $\E^{]0, s]}_\qp$) $v^{[r, s]}(\partial f_{ij}) \geq v^{[r, s]}(f_{ij}) - s$ et du calcul du paragraphe précédent et la dernière suit en appliquant la définition de $v^{[r, s]}$ une autre fois.
 
 Enfin, $\Delta^{]0, s]} \subseteq \E_K^{]0, s]} \otimes W$ est stable par $\partial$, il est muni par restriction d'une famille de valuations $v^{[r, s]}$ et, par ce qui précède, la norme $v^{[r, s]}$ de $\partial$ agissant sur $\Delta^{]0, s]}$ est bornée par $-s - 1$, ce qui permet de conclure.

\end{proof}

\begin{remarque} \leavevmode
\begin{itemize}
\item Si $0 < r < s < (\delta_K + 1)^{-1}$, les modules \[ \Delta^{[r, s]} = \E_\qp^{[r, s]} \otimes_{\E_\qp^{]0, s]}} \Delta^{]0, s]} = (\E_K^{[r, s]} \otimes W)^G, \] complétés de $\Delta^{]0, s]}$ pour la valuation $v^{[r, s]}$, sont des $\E^{[r, s]}$-modules de rang $d$, munis de la valuation $v^{[r, s]}$ pour les quelles la norme de $\partial$ est $\geq -s - 1$. On a \[ \Delta = \varinjlim_{s > 0} \varprojlim_{0 < r < s} \Delta^{[r, s]}. \]
\item Le lemme \ref{lemme11} ci-dessus montre que le rayon de surconvergence du module $\Delta$ peut être majoré en termes du conducteur de la plus petite extension galoisienne $K / \qp$ telle que $\mathscr{G}_K$ agit trivialement sur $D_{\rm pst}$. Si $V \in \mathrm{Rep}_L \mathscr{G}_\qp$ est de Rham, ceci permet de borner le rayon de surconvergence de $\Delta = \Nrig(V)$ en termes de la plus petite extension (galoisienne) $K$ de $\qp$ sur laquelle $V$ dévient semi-stable.
\end{itemize}
\end{remarque}

%

\subsubsection{$(\varphi, \Gamma)$-modules relatifs} \label{phiGammarel} Soit $A$ une algèbre affino\"ide sur $\qp$. L'anneau de Robba $\Robba_A$ relatif à $A$ est défini en posant, pour $0 < r < s$, $$ \Robba_A^{[r, s]} = \Robba^{[r, s]} \widehat{\otimes} A; \;\;\; \Robba_A^{]0, s]} = \varprojlim_{0 < r < s} \Robba_A^{[r, s]}; \; \; \; \Robba_A = \varinjlim_{s > 0} \Robba_A^{]0, s]}, $$ où le premier produit tensoriel est le produit tensoriel complété entre deux espaces de Banach. On peut montrer que $\Robba_A = \Robba_\qp \widehat{\otimes} A$ (produit tensoriel complété inductif ou projectif entre deux espaces topologiques localement convexes). Ceci s'interprète en termes de fonctions analytiques sur des espaces rigides: si $I \subseteq [0, 1[$ est un intervalle d'extremités dans $p^\Q$ et $\mathbb{A}^1 = \mathrm{Sp}(\qp \langle T \rangle)$ dénote la droite affine rigide de paramètre $T$, en notant $B_I$ l'ouvert admissible de $\mathbb{A}^1$ défini par $v_A(T) \in I$ ($v_A = - \log_p |\cdot|_A$ dénote la valuation de l'algèbre aff\"inoide $A$), on a un isomorphisme naturel $$ \Robba^I_A \cong \mathcal{O}(\mathrm{Sp}(A) \times B_I). $$ On a aussi une interprétation en termes de séries de Laurent (resp. de puissances si $0 \in I$) à coefficients dans $A$ de la manière évidente. On a un endomorphisme $A$-linéaire d'anneaux $\varphi \colon \Robba_A^{]0, s]} \to \Robba_A^{]0, s / p]}$, qui envoie $T$ sur $(1 + T)^p - 1$, induisant une action de $\varphi$ sur $\Robba_A$ et on a une action continue du groupe $\Gamma$, agissant par $\sigma_a(T) = (1 + T)^a - 1$, $a \in \zpe$, sur tous les anneaux définis ci-dessus.

Pour un module $D^{]0, s]}$ sur $\Robba_A^{]0, s]}$ et $0 < s' < s$, on note $D^{]0, s']} = D^{]0, s]} \otimes_{\Robba_A^{]0, s]}} \Robba_A^{]0, s']}$ et $\varphi^*(D^{]0, s]}) = D^{]0, s]} \otimes_{\Robba_A^{]0, s]}, \varphi} \Robba_A^{]0, s/p]}$. Un \textit{$(\varphi, \Gamma)$-module sur $\Robba_A^{]0, s]}$} est un module projectif de type fini $D^{]0, s]}$ sur $\Robba_A^{]0, s]}$, muni d'un isomorphisme $\E^{]0, s/p]}$-linéaire $\tilde{\varphi} \colon \varphi^*(D^{]0, s]}) \to D^{]0, s/p]}$ et d'une action semi-linéaire de $\Gamma$, commutant avec $\tilde{\varphi}$ dans le sens évident. L'isomorphisme $\varphi^*(D^{]0, s]}) \to D^{]0, s/p]}$ induit, pour tout $0 < s' < s$, des opérateurs semi-linéaires $\varphi \colon D^{]0, s']} \to D^{]0, s'/p]}$, définis comme la composée \[ D^{]0, s']} = D^{]0, s]} \otimes_{\Robba_A^{]0, s]}} \Robba_A^{]0, s']} \to \varphi^*(D^{]0, s]}) \otimes_{\Robba_A^{]0, s]}} \Robba_A^{]0, s']} \to D^{]0, s/p]} \otimes_{\Robba_A^{]0, s]}} \Robba_A^{]0, s'/p]} = D^{]0, s'/p]}, \] la première flèche étant celle induite par $D^{]0, s]} \to \varphi^*(D^{]0, s]}) \colon x \mapsto x \otimes 1$ et la deuxième étant $\tilde{\varphi} \otimes \varphi$.

Un \textit{$(\varphi, \Gamma)$-module sur $\Robba_A$} est un $\Robba_A$-module (projectif de type fini) $D$ tel qu'il existe $s > 0$ et un $(\varphi, \Gamma)$-module $D^{]0, s]}$ sur $\Robba_A^{]0, s]}$ tel que $D \cong D^{]0, s]} \otimes_{\Robba_A^{]0, s]}} \Robba_A$.

Plus généralement, si $X$ est un espace rigide analytique sur $\qp$ et $r \geq 0$, on définit $\Robba_X^{]0, r]}$ comme le faisceau des fonctions rigides analytiques sur $X \times B_{]0, r]}$, $\Robba_X = \varinjlim_{r < 0} \Robba_X^{]0, r]}$ et un $(\varphi, \Gamma)$-module sur $\Robba_X$ est une collection compatible de $(\varphi, \Gamma)$-modules sur $\Robba_A$ pour chaque ouvert admissible $A$ de $X$.

\subsubsection{Cohomologie des $(\varphi, \Gamma)$-modules}

Si $V \in \mathrm{Rep}_L \mathscr{G}_\qp$ est une représentation $p$-adique de $\mathscr{G}_\qp$ et $\gamma$ est un générateur topologique de $\Gamma$ \footnote{Si $p = 2$, il faut considérer $\Gamma' \subseteq \Gamma$ la $p$-partie du sous-groupe de torsion de $\Gamma$, $\gamma \in \Gamma$ dont l'image dans le quotient $\Gamma / \Gamma'$ soit un générateur topologique, et prendre les invariants par $\Gamma'$ dans tous les modules de la suite ci-dessous. }, le complexe $$ 0 \to \D(V) \xrightarrow{x \mapsto ((1-\varphi) x, (1 - \gamma) x)} \D(V) \oplus \D(V) \xrightarrow{(x, y) \mapsto (1 - \gamma) x - (1 - \varphi) y} \D(V) \to 0 $$ calcule la cohomologie galoisienne de $V$ en termes de son $(\varphi, \Gamma)$-module $\D(V)$ sur $\E_\qp$ associé (cf. \cite{Herr}). D'après \cite{Liu} et \cite{KPX}, on peut définir la cohomologie des $(\varphi, \Gamma)$-modules sur l'anneau de Robba $\Robba_A$ relatif à une algèbre aff\"inoide $A$ sur $\qp$ (et, plus généralement, sur un espace analytique rigide sur une extension finie de $\qp$), retrouvant les constructions de Herr dans le cas d'un $(\varphi, \Gamma)$-module étale au-dessus d'un point.

Soient $A$ une algèbre aff\"inoide sur $\qp$ et $D$ un $(\varphi, \Gamma)$-module sur $\Robba_A$. On note $\Gamma' \subseteq \Gamma$ la partie de $p$-torsion de $\Gamma$ (qui est triviale si $p \neq 2$, et cyclique d'ordre $2$ quand $p = 2$). Soit $\gamma \in \Gamma$ tel que son image dans $\Gamma / \Gamma'$ en est un générateur topologique. On pose $\gamma_0 = \gamma$ et, pour $n \geq 1$, $\gamma_n$ un générateur topologique de $\Gamma_n$. Pour $\delta \in \{ \varphi, \psi \}$ et $\gamma' \in \{ \gamma_n : n \geq 0\}$, on note $D' = D^{\Gamma'}$ si $n = 0$ et $D' = D$ si $n \geq 1$, et on définit le complexe $$ \mathscr{C}^\bullet_{\delta, \gamma'}(D): 0 \to D' \to D' \oplus D' \to D' \to 0, $$ où les flèches sont données, respectivement, par $x \mapsto ((\delta - 1)x, (\gamma' - 1)x)$ et $(x, y) \mapsto (\gamma' - 1)x - (\delta - 1) y$. Les modules $H_{\delta, \gamma'}^\bullet(D)$ sont définis comme les groupes de cohomologie de ce complexe. Quand $n = 0$, on note $H_{\delta, \gamma}^i(D) = H_{\delta, \gamma_0}^i(D)$. 

\begin{proposition} [{\cite[Prop. 2.3.6, Thm. 4.4.2]{KPX}}] \label{KPX}
Soient $A$ une algèbre aff\"inoide sur $\qp$ et $D$ un $(\varphi, \Gamma)$-module sur $\Robba_A$. Les complexes $\mathscr{C}_{\varphi, \gamma'}(D)$ et $\mathscr{C}_{\psi, \gamma'}(D)$ sont quasi-isomorphes et les groupes de cohomologie $H^i_{\varphi, \gamma'}(D)$ sont des $A$-modules de type fini, compatibles au changement de base. On a une dualité locale et une formule de Euler-Poincaré.
\end{proposition}

\subsubsection{Cohomologie d'Iwasawa des $(\varphi, \Gamma)$-modules} \label{cohomIw}


Soit $\Lambda = \zp[[\Gamma]]$ l'algèbre d'Iwasawa de $\Gamma$. On décompose $\Gamma = \Gamma^{\rm tors} \times \widetilde{\Gamma}$, où $\Gamma^{\rm tors}$ désigne la partie de torsion de $\Gamma$ et $\widetilde{\Gamma} \cong 1 + 2p \zp$ via le caractère cyclotomique, et on obtient un isomorphisme $\Lambda \cong \zp[\Gamma^{\rm tors}] \otimes_\zp \zp[[\widetilde{\Gamma}]]$. Soit $\gamma$ un générateur topologique de $\widetilde{\Gamma}$ et notons $[\gamma]$ son image dans $\Lambda$. On obtient un isomorphisme $\Lambda \cong \zp[\Gamma^{\rm tors}] \otimes_\zp \zp[[T]]$ en envoyant $[\gamma]$ sur $1+T$.

Soit $\Lambda_\infty = \Robba^+(\Gamma)$ l'algèbre algèbre de distributions sur $\Gamma$. Précisément, on obtient $\Robba^+(\Gamma)$ en remplaçant la variable $T$ par $[\gamma] - 1$ dans la définition de $\Robba^+$. On peut de la sorte définir les anneaux $\Lambda_n = \Robba^{[r_n, +\infty]}(\Gamma)$ et on a $ \Lambda_\infty = \varprojlim_n \Lambda_n. $ Le choix de l'isomorphisme $\Lambda \cong \zp[\Gamma^{\rm tors}] \otimes_\zp \zp[[T]] $ fait que $\Lambda_n$ s'identifie aux fonction analytiques sur la boule $v_p(T) \geq r_n$ et $\Lambda_\infty$ à celui des fonctions analytiques sur la boule ouverte unité. Ce dernier espace s'identifie à l'anneau $H^0(\mathfrak{X}, \mathcal{O})$ des sections globales sur l'espace des poids $p$-adiques, et aussi, par le théorème d'Amice (cf. \cite[Th. 2.2]{S-T}) au $L$-dual continu des fonctions localement analytiques sur $\zpe$.


On note (cf. \cite[Def. 4.4.7]{KPX}) $\Lambda_n^\iota$ le module $\Lambda_n$ muni de l'action de $\Gamma$ via $\gamma(f) = [\gamma^{-1}] \cdot f$, $\gamma \in \Gamma$ et $f \in \Lambda_n$. On définit $$ \mathbf{Dfm} = \varprojlim_{n} \Lambda_n^\iota \widehat{\otimes}_\qp \Robba = \varprojlim_{n} \varinjlim_{s > 0} \varprojlim_{r < s} \Robba^{[s, r]} \widehat{\otimes} \Lambda_n^\iota. $$
Plus généralement, si $D$ est un $(\varphi, \Gamma)$-module sur $\Robba$, on définit sa déformation cyclotomique par 
\[ \mathbf{Dfm}(D) = D \widehat{\otimes}_\qp \Lambda_\infty^\iota = \varprojlim_n \varinjlim_{s > 0} \varprojlim_{r < s} D^{[r, s]} \widehat{\otimes} \Lambda_n^\iota. \]
Les actions $\varphi$, $\psi$ et $\Gamma$ sont données par les formules \[ \varphi(x \otimes \lambda) = \varphi(x) \otimes \lambda, \;\;\; \psi(x \otimes \lambda) = \psi(x) \otimes \lambda, \;\;\; \gamma(x \otimes \lambda) = \gamma(x) \otimes [\gamma^{-1}] \lambda, \] pour $x \in D$, $\lambda \in \Lambda_\infty$ et $\gamma \in \Gamma$. Le module $\mathbf{Dfm}(D)$ est un $(\varphi, \Gamma)$-module sur l'anneau $\Robba_\mathfrak{X} = \varprojlim_{n} \Robba_{\mathfrak{X}_n}$ de Robba relatif à l'espace des poids.

On définit la cohomologie d'Iwasawa de $D$ comme $$ H^i_{\rm Iw}(D) = H^i_{\psi, \gamma}(\mathbf{Dfm}(D)). $$ Ce sont, d'après le théorème \ref{KPX}, des $\Lambda_\infty$-modules de type fini. On peut donc voir les groupes de cohomologie d'Iwasawa comme des groupes de cohomologie à valeurs dans $\mathscr{D}(\zpe, D)$.


Si $\eta \colon \Gamma \to L^\times$ est un caractère, le changement de base par rapport à $f_\eta \colon \Lambda_\infty \to L : [\gamma] \mapsto \eta^{-1}(\gamma)$ fournit un isomorphisme $$\mathbf{Dfm}(D) \otimes_{\Lambda_\infty, f_\eta} L \xrightarrow{\sim} D(\eta). $$Si $\mu \in H^i_{\rm Iw}(D)$, cet isomorphisme induit des morphismes de spécialisation $$ H^i_{\rm Iw}(D) \to H^i_{\psi, \gamma}(D(\eta)): \; \; \; \mu \mapsto \int_\Gamma \eta \cdot \mu, $$ la notation étant justifiée par l'interprétation classique de la cohomologie d'Iwasawa en termes des distributions.

Si $D \in \Phi\Gamma(\Robba)$, on définit le complexe $\mathscr{C}^\bullet_\psi(D)$, concentré en $[1, 2]$, par $$ [D \xrightarrow{\psi - 1} D]. $$ Le complexe $\mathscr{C}^\bullet_{\psi}(D)$ appartient à $\D^{[0, 2]}_{\rm perf}(\Lambda_\infty)$ \footnote{Rappelons que, pour un anneau $R$, $\D^{[0, 2]}_{\rm perf}(R)$ dénote la sous-catégorie de la catégorie dérivée de la catégorie des modules sur $R$ dont les objets sont quasi-isomorphes à un complexe borné formé par des $R$-modules projectifs finis et tel que leur cohomologie est concentrée en dégrées $[0, 2]$.} et calcule la cohomologie d'Iwasawa de $D$ (cf. \cite[Thm. 4.4.8]{KPX}). On a, en particulier, un isomorphisme $$ \mathrm{Exp}^* \colon H^1_{\rm Iw}(D) \to D^{\psi = 1} $$ dont l'inverse est donnée par $ z \mapsto [\frac{p-1}{p} \log(\chi(\gamma)) (z \otimes 1), 0]$ \footnote{Si $p = 2$, il faut appliquer à $z$ le projecteur naturel sur le sous espace d'éléments $\Gamma'$-invariants. On évitera ce cas-ci, se traitant de la même manière mais avec une complication technique supplémentaire. }. Si $z \in D^{\psi = 1}$, on note, afin d'alléger les notation, $\mu_z$ l'élément $(\mathrm{Exp}^*)^{-1}(z) \in H^1_{\rm Iw}(D)$. Si $n \geq 0$, on a (cf. \cite[\S VIII.1.3]{ColmezPhiGamma} ou \cite[\S 2.2.3, Eq. (6)]{Nakamura2}) la formule pour la spécialisation $$\int_{\Gamma_n} \eta \cdot \mu_z = [\tau_n(\gamma_n) (z \otimes e_\eta), 0] \in H^1_{\psi, \gamma_n}(D(\eta)), $$ où $\tau_n(\gamma_n) = p^{-n} \log(\chi(\gamma_n))$, si $n \geq 1$ et $\tau_0(\gamma_0) = \frac{p - 1}{p} \log(\chi(\gamma_0))$.

Terminons en mentionnant que $\eta$ induit un automorphisme sur $\Lambda_\infty$ donné par $\eta([\gamma]) = \eta(\gamma)^{-1} \cdot [\gamma]$, ce qui induit un isomorphisme de $(\varphi, \Gamma)$-modules $\mathbf{Dfm}(D) \cong \mathbf{Dfm}(D(\eta))$, donné par $x \otimes [\gamma] \mapsto (x \otimes e_\eta) \otimes \eta(\gamma)^{-1} [\gamma]$, et donc un isomorphisme de $\Lambda_\infty$-modules \[ H^i_{\rm Iw}(D) \to H^i_{\rm Iw}(D(\eta)): \;\;\; \mu \mapsto \mu \otimes e_\eta. \] On a, par exemple, $$ \int_\Gamma \eta \cdot \mu = \int_\Gamma 1 \cdot (\mu \otimes e_\eta). $$

\subsubsection{Applications exponentielles}

Si $D \in \Phi\Gamma(\Robba)$ est de Rham et $n \geq 0$, on note $$ \exp_{D, F_n} \colon L_n \otimes \DdR(D) \to \H^1_{\varphi, \gamma_n}(D), $$ $$ \exp^*_{D, F_n} \colon \H^1_{\varphi, \gamma_n}(D) \to L_n \otimes \DdR(D), $$ les applications exponentielle et exponentielle duale de Bloch-Kato comme définies dans \cite[\S 2.3; \S 2.4]{Nakamura}. Quand cela ne pose pas de problèmes, on omettra les indices dans les notations des applications exponentielles.

Explicitement (\cite[Lem. 2.12(1)]{Nakamura}), si $x \in \DdR(D) = (\D_{\mathrm{dif}}(D))^\Gamma$, alors il existe $m$ tel que $x \in \D_{\mathrm{dif}, m}(D)$, et il existe aussi $\tilde{x} \in D^{]0, r_m]}[1/t]$ tel que $\varphi^{-k}(\tilde{x}) - x \in \D_{\mathrm{dif}, k}^+(D)$ pour tout $k \geq m$. Alors on a $$ \exp_{D, F_n}(x) = [(\gamma_n - 1) \tilde{x}, (\varphi - 1) \tilde{x} ] \in H^1_{\varphi, \gamma_n}(D). $$ On remarquera que l'application $\exp_D$ est nulle sur $\mathrm{Fil}^0 \, \DdR(D)$.

Si $M$ est un module muni d'une action de $\Gamma$, et $\gamma' \in \{ \gamma_n: n \geq 0 \}$, on pose \footnote{Comme précédemment, $M' = M^{\Gamma'}$ si $\gamma' = \gamma_0$ et $M' = M$ autrement.} $$ C^\bullet_{\gamma'}(M) = [M' \xrightarrow{\gamma' - 1} M'], $$ et on définit les groupes de cohomologie $H^i_{\gamma'}(M) = H^i(C^\bullet_{\gamma'}(M))$. Par exemple, si $D \in \Phi\Gamma(\Robba)$ et $n \geq 0$, alors $$ H^0_{\gamma_n}(\Ddif(D)) = \Ddif(D)^{\Gamma_n} = L_n \otimes \D_{\rm dR}(D). $$

Si $D \in \Phi\Gamma(\Robba)$ est de Rham et $n \geq 0$, on a un isomorphisme $L_\infty((t)) \otimes_{L_n} (L_n \otimes \DdR(D)) \cong \D_{\mathrm{dif}}(D)$ et l'application $L_n \otimes \DdR(D) \to H^1_{\gamma_n}(\D_{\mathrm{dif}}(D))$ qui envoie $x \in L_n \otimes \DdR(D)$ vers la classe de cohomologie $[\log \chi(\gamma_n) (1 \otimes x)] \in H^1_{\gamma_n}(\D_{\mathrm{dif}}(D))$ est un isomorphisme (cf. \cite{Nakamura}, juste avant le lemme 2.14). De plus, on a une application naturelle $H^1_{\varphi, \gamma_n}(D) \to H^1_{\gamma_n}(\Ddifp(D)) \to H^1_{\gamma_n}(\Ddif(D))$ définie par $[(x, y)] \mapsto [\varphi^{-m} x]$, où $m \gg 0$ et $[ \cdot ]$ dénote la classe de cohomologie correspondante.
On définit $$ \exp^*_{D, F_n} \colon H^1_{\varphi, \gamma_n}(D) \to L_n \otimes \DdR(D) $$ comme la composition $H^1_{\varphi, \gamma_n}(D) \to H^1_{\gamma_n}(\Ddifp(D)) \to H^1_{\gamma_n}(\Ddif(D))  \xrightarrow{\sim} L_n \otimes_\qp \DdR(D)$. Remarquons que, par construction, l'image de $\exp^*_{D, F_n}$ tombe dans $L_n \otimes \mathrm{Fil}^0 \, \DdR(D)$.


\subsubsection{Exponentielle de Perrin-Riou} \label{expPR}

Rappelons la formulation de l'application exponentielle de Perrin-Riou pour un $(\varphi, \Gamma)$-module de de Rham (cf. \cite{Nakamura}, \cite{Berger03}) et la loi de réciprocité. On définit, pour $h \geq 0$ \footnote{Dans la formule ci-dessous, $\nabla_0$ dénote simplement l'identité et $\nabla_1 = \nabla$.}, l'opérateur différentiel $$ \nabla_h = (\nabla - h + 1) \circ (\nabla - h + 2) \circ \hdots \circ \nabla \in \Lambda_\infty. $$

\begin{lemme} [{\cite[Lem. 3.6]{Nakamura}}]
Soit $D \in \Phi \Gamma(\Robba)$ de Rham et soit $h \geq 1$ tel que $\mathrm{Fil}^{-h} \, \DdR(D) = \DdR(D)$. Alors $\nabla_h (\Nrig(D)) \subseteq D$ et il induit un opérateur $\Lambda_\infty$-linéaire $$ \nabla_h \colon \Nrig(D)^{\psi = 1} \to D^{\psi = 1}. $$
\end{lemme}

Si $x \in \Nrig(D)^{\psi = 1}$ et $n\geq 0$, l'expression $[\varphi^{-n} x]_j$ \footnote{Rappelons que, si $x \in D^{]0, r_n]}$, l'élément $\varphi^{-n}(x)$ appartient à $L_n((t)) \otimes_\qp \DdR(D)$, s'écrit donc sous la forme $\varphi^{-n} x = \sum_j d_j t^j$, $d_j \in L_n \otimes \DdR(D)$ et on note $[\varphi^{-n}x]_j = d_j \in L_n \otimes \DdR(D) $.}, pour $j \in \N$, n'est définie que pour $n$ assez grand. Pourtant, l'identité $\mathrm{Tr}_{L_m / L_n} \circ \varphi^{-m} = \varphi^{-m} \circ \psi^{m-n}$, $m \geq n \geq 0$, nous permet de lui donner toujours un sens en utilisant les traces et en considérant les valeurs $p^{-m} \mathrm{Tr}_{L_m / L_n} [\varphi^{-m} x]_j$, où $m$ est n'importe quel entier assez grand.



\begin{theorem} [{\cite[Thm. 3.10]{Nakamura}}] \label{recPR}
Soient $D \in \Phi\Gamma(\Robba)$ de Rham, $h \geq 1$ un entier tel que $\mathrm{Fil}^{-h} \, \DdR(D) = \DdR(D)$, $z \in \Nrig(D)^{\psi = 1}$ et $n \geq 0$. On a

$(i)$ Si ${\j} > 0$ et si il existe $z_{\j} \in \Nrig(D)^{\psi = p^{-j}}$ vérifiant $\partial^{\j} z_{\j} = z$, ou si $-(h-1) \leq j \leq 0$ et $z_{\j} = \partial^{-{\j}} z$, alors, on a, pour $m \gg 0$, \[ \int_{\Gamma_n} \chi^{\j} \cdot \mu_{\nabla_h z} = \frac{p^{m (j - 1)}}{\Gamma^*(- j - h + 1)}  \exp ( \mathrm{Tr}_{L_m / L_n}  [\varphi^{-m} z_j]_0 \otimes t^{-{\j}} e_j), \]

$(ii)$ Si ${\j} \geq h$, on a, pour $m \gg 0$, \[ \exp^* (\int_{\Gamma_n} \chi^{-{\j}} \cdot \mu_{\nabla_h z} ) = \frac{p^{m (-j - 1)}}{\Gamma^*(- j - h + 1)} \mathrm{Tr}_{L_m / L_n} [ \varphi^{-m} \partial^{\j} z]_0 \otimes t^{\j} e_{-j}. \]
\end{theorem}

\begin{remarque} \leavevmode
\begin{itemize}
\item Rappelons que, pour $j \in \Z$, l'élément $e_j$ dénote une base du $L$-espace vectoriel $L(j)$ muni d'actions de $\Gamma$ et $\varphi$ par les formules $\sigma_a(e_j) = a^j \cdot e_j$, $a \in \zpe$, et $\varphi(e_j) = e_j$. Si $D$ est un $(\varphi, \Gamma)$-module, on note $D(j) = D \otimes_L L(j)$ la tordue de $D$ par la $j$-ième puissance du caractère cyclotomique. Si $D$ est de Rham, $D(j)$ l'est aussi et on a $\DdR(D(j)) = \DdR(D) \otimes \DdR(\Robba(j)) = \DdR(D) \otimes L \cdot (t^{-j} e_j) = \DdR(D) \otimes L \cdot \mathbf{e}^{\rm dR}_j$, dans la notation de \S \ref{basesdeRham2}. Le terme $[\varphi^{-m} z_{\j}]_0$ appartient à $L_m \otimes \DdR(D)$ et $x \mapsto x \otimes t^{-j} e_j$ induit un isomorphisme $\DdR(D) \xrightarrow{\sim} \DdR(D(j))$. La première égalité a lieu dans $L_n \otimes H^1_{\varphi, \gamma}(D(j))$ (ou $H_{\varphi, \gamma_n}(D(j))$) et la deuxième dans $L_n \otimes_\qp \DdR(D(-j))$.

\item Notons que $\mu \mapsto \mu \otimes e_j$ induit, pour $j \in \Z$, un isomorphisme $H^1_{\rm Iw}(D) \xrightarrow{\sim} H^1_{\rm Iw}(D(j))$ et on a donc, pour $z \in \Nrig(D)^{\psi = 1}$, $\int_{\Gamma_n} \chi(x)^{-j} \cdot \mu_z = \int_{\Gamma_n} 1 \cdot (\mu_z \otimes e_{-j})$. Si $j \geq 0$, l'élément $\varphi^{-n} (z \otimes e_{-j}) \in L_n[[t]] \otimes \DdR(D(-j))$ s'écrit sous la forme $t^{-j} \varphi^{-n} z \otimes t^{j} e_{-j}$ et, si on écrit $\varphi^{-n} z = \sum_l d_l t^l \in L_n[[t]] \otimes \DdR(D)$, l'expression $[\varphi^{-n}(z \otimes e_{-j})]_0$ dénote l'élément $d_j \otimes t^{-j} e_j$, qui n'est autre que $\frac{p^{-nj}}{j!}[\varphi^{-n}\partial^{j} z]_0 \otimes t^{j} e_{-j}$. Le deuxième point du théorème est donc une paraphrase de la loi de réciprocité de Cherbonnier-Colmez. Pour $j < 0$, le résultat est (au moins pour l'auteur) un peu plus mystérieux  car l'opérateur $\partial^{j}$, devenant un opérateur d'intégration, fait apparaître des termes qu'on ne voyait pas directement dans le développement de $\varphi^{-n} z$.

\item Si $j \gg 0$ l'application $\exp_{D(j), F_n}$ est bijective et on peut reformuler le résultat en disant que, si $j \geq 0$ ou $j \ll 0$, alors, pour $m \gg 0$, on a \[ \mathrm{Tr}_{L_m / L_n} [ \varphi^{-m} \partial^{\j} z]_0 =   \frac{\Gamma^*(- j - h + 1)}{p^{m (-j - 1)}} \cdot \log (\int_{\Gamma_n} \chi^{-j} \cdot \nabla_h \mu_z) \otimes \mathbf{e}^{\rm dR, \vee}_{-j}, \] où $\log$ dénote $\exp^{-1}$ ou $\exp^*$ selon que $j \ll 0$ ou que $j \geq 0$.

\item Si $n \geq m(\Delta)$ ($\Delta = \Nrig(D)$), alors l'introduction de $m$ dans les formules est superflue et, comme $\psi(z) = z$, alors $z \in D^{]0, r_n]}$ et on a \[ p^{-m(j + 1)} \mathrm{Tr}_{L_m / L_n} [\varphi^{-m} \partial^j z ]_0 =  \frac{[L_m: L_n]}{p^{m - n}} [\varphi^{-n(j+1)} \partial^j z]. \] On remarque que $\frac{[L_m: L_n]}{p^{m - n}} = 1$ si $n > 0$ et $\frac{[L_m: L]}{p^{m}} = \frac{p - 1}{p}$.

\end{itemize}
\end{remarque}

On se servira de la version suivante de la loi de réciprocité de Cherbonnier-Colmez, qui est un des ingrédients de la preuve du théorème précédent, mais pour lequel on n'a pas trouvé de référence précise.

\begin{proposition} \label{recPRcomp} Soient $D \in \Phi\Gamma(\Robba)$ de Rham, $z \in D^{\psi = 1}$, $n \geq 0$ et $j \geq 0$. Alors, pour $m \gg 0$, on a l'égalité suivante dans $L_n \otimes \DdR(D)$:
\[ \exp^*(\int_{\Gamma_n} \chi^{-j} \cdot \mu_z ) \otimes t^{-j} e_j = \frac{1}{j!} p^{-m(j + 1)} \mathrm{Tr}_{L_m/ L_n}([\varphi^{-m} \partial^j z]_0). \]
\end{proposition}

\begin{proof}
La formule $\exp^*(\int_{\Gamma_n} \chi^{-j} \cdot \mu_z ) = \exp^*(\int_{\Gamma_n} 1 \cdot \mu_{z \otimes e_{-j}} )$ permet de nous ramener au cas $j = 0$. En outre, la formule $\mathrm{Tr}_{L_{n + 1} / L_n} (\exp^*(x)) = \exp^*(\mathrm{cor}_{F_{n+1} / F_n}(x))$ \footnote{Si $[x, y] \in H^1_{\psi, \gamma_{n+1}}(D)$, sa corestriction est définie par la formule $\mathrm{cor}_{F_{n+1} / F_n}([x, y]) = [\frac{1 - \gamma_{n+1}}{1 - \gamma_n}x, y] \in H^1_{\psi, \gamma_n}(D)$, cf. \cite[Lem. II.2.1]{ColmezIw2}.} nous ramène à montrer le résultat pour $n$ assez grand. En particulier on peut considérer $n > 0$ de sorte que $z \in D^{]0, r_n]}$ et on doit donc montrer \[ \exp^*(\int_{\Gamma_n} 1 \cdot \mu_z ) = p^{-n} [\varphi^{-n} z]_0. \]

On a la formule pour la spécialisation $\int_{\Gamma_n} 1 \cdot \mu_{z} = [\tau_n(\gamma_n) (z \otimes 1), 0] \in H^1_{\psi, \gamma}(D)$, où $\tau_n(\gamma_n) = p^{-n} \log(\chi(\gamma_n))$ (car on suppose $n > 0$). Rappelons que l'application $\exp^*$ est définie comme la composition du morphisme $H^1_{\varphi, \gamma_n}(D) \to H^1_{\gamma_n}(\Ddif(D))$ avec l'inverse de l'isomorphisme $L_n \otimes \DdR(D) \xrightarrow{\sim} H^1_{\gamma_n}(\Ddif(D))$ donné par $x \mapsto [\log (\chi(\gamma_n)) (1 \otimes x)]$.

Si $z \in D^{\psi = 1}$, la classe de cohomologie $[z, 0] \in H^1_{\psi, \gamma_n}(D)$ correspond à la classe $[z, (1 - \gamma_n)^{-1} (1 - \varphi) z] \in H^1_{\varphi, \gamma}$ sous l'isomorphisme entre ces deux modules. L'image du cocycle $[ \tau_n(\gamma_n) (z \otimes 1), 0] $ par le morphisme $H^1_{\varphi, \gamma_n}(D) \to H^1_{\gamma_n}(\Ddif(D))$ est donc donnée par $ [ \tau_n(\gamma_n) \varphi^{-n} z]$.

Or, $\varphi^{-n} z = \sum_{l \geq 0} a_l t^l$, $a_l \in L_n[[t]] \otimes \DdR(D)$ et $[ \tau_n(\gamma_n) \sum_{l \geq 0} a_l t^l] = [\tau_n(\chi(\gamma_n)) a_0]$ (dans $H_{\gamma_n}^1(\Ddif(D))$), car tous les autres termes sont dans l'image de $(1 - \gamma_n)$ (en effet, si $l \neq 0$, $a_l t^l = (1-  \gamma_n) ( (1 - \chi(\gamma_n)^l)^{-1} a_l t^l$...), ce qui permet de conclure car $a_0 = [\varphi^{-n} z]_0$.
\end{proof}

\subsection{La fonction $L$ locale}

\subsubsection{Distributions à valeurs dans un espace de type $\mathrm{LF}$}

On reprend les constructions de \S \ref{poids}. Si $M = \varinjlim_s \varprojlim_r M^{[r, s]}$ est un espace de type $\mathrm{LF}$, on pose $$ \mathcal{O}(\mathfrak{X}) \widehat{\otimes} M = \varprojlim_{n} \varinjlim_{s > 0} \varprojlim_{0 < r < s} \mathcal{O}(\mathfrak{X}_n) \widehat{\otimes} M^{[r, s]}, $$ où le dernier produit tensoriel est le produit tensoriel usuel entre deux espaces de Banach. On dit que $f$ est une fonction analytique sur $\mathfrak{X}$ à valeurs dans $M$ si $f$ est un élément de $\mathcal{O}(\mathfrak{X}) \widehat{\otimes} M$.

\subsubsection{Prolongement analytique de $k \mapsto \partial^k$} \label{prolan}

Soit $D \in \Phi \Gamma(\Robba)$ de rang $d$, de Rham, et notons $\Delta = \Nrig(D)$. Rappelons que, d'après le lemme \ref{lemme11}, si $K / \qp$ dénote la plus petite extension galoisienne telle que $\mathscr{G}_\qp$ agit sur $\D_{\rm pst}(D)$ à travers $\mathrm{Gal}(K / \qp)$ et si l'on note $r(\Delta) = (\delta_K + 1)^{-1}$, il existe, pour tout $s < r(\Delta)$, des sous-$\E^{]0, s]}$-modules $\Delta^{]0, s]}$ de $\Delta$, munis d'une famille de valuations $v^{[r, s]}$, $0 < r < s < r(\Delta)$, pour lesquelles la norme de l'opérateur $\partial$ satisfait \[ v^{[r, s]}(\partial) \geq -s - 1 \geq - (\delta_K + 1)^{-1} - 1. \] \textit{On supposera dorénavant, et par commodité, que $m(\Delta)$ est tel que $r_{m(\Delta)} < r(\Delta)$}. De plus, pour $s < r(\Delta)$, les opérateurs \[ \varphi \colon \Delta^{]0, s]} \to \Delta^{]0, s/p]}, \;\;\; \psi \colon \Delta^{]0, s/p]} \to \Delta^{]0, s]}, \] agissent de façon continue pour les valuations $v^{[r, s]}$. Rappelons aussi que l'on a \[ \Delta = \varinjlim_{s > 0} \Delta^{]0, s]} = \varinjlim_{s > 0} \varprojlim_{0 < r < s} \Delta^{[r, s]}. \]

Si $z \in \Delta^{\psi = 0}$ et $N > 0$, on peut écrire $z = \sum_{i \in (\Z / p^N \Z)^\times} (1 + T)^i \varphi^N(z_i)$, $z_i = \psi^N( (1 + T)^{-i} z)$. La formule de Leibnitz et l'identité $\partial \circ \varphi = p \, \varphi \circ \partial$ donnent, pour $k \geq 0$, \[ \partial^k  z = \sum_{i \in (\Z / p^n \Z)^\times} \sum_{j = 0}^k {k \choose j} i^{k - j} (1 + T)^i p^{Nj} \varphi^N( \partial^j z_i). \] Enfin, rappelons que, si $\eta \in \mathfrak{X}(L)$ alors $\omega_\eta$ dénote son poids (cf. \S \ref{poids}) et que le poids du caractère $x \mapsto x^k$ est $k$. L'identité de la formule ci-dessus suggère la proposition suivante, qui est le point de départ des constructions de cet article.


\begin{proposition} \label{prop4}
Soient $\kappa \in \mathfrak{X}$ et $z \in \Delta^{\psi = 0}$. Alors, la série \begin{equation} \label{form2}  \sum_{i \in (\Z / p^N \Z)^\times} \sum_{j = 0}^{+\infty} {\omega_\kappa \choose j} \kappa(i) i^{-j} (1+T)^i p^{Nj} \varphi^N(\partial^j z_i)  \end{equation} converge, pour $N \gg 0$, dans $\Delta^{\psi = 0}$, et la somme ne dépend ni de $N$ ni du choix du système de représentants de $(\Z / p^N \Z)^\times$. L'application $\kappa \mapsto \kappa(\partial) z$ ainsi définie est une fonction rigide analytique sur $\mathfrak{X}$ à valeurs dans $\Delta^{\psi = 0}$.
\end{proposition}

\begin{proof}
Notons, pour $N > 0$ et $0 < r < s < r(\Delta) p^{-N}$, $C^{[r, s]}_{\varphi^N}$ et $C^{[r, s]}_{\psi^N}$ les normes, relatives aux valuations $v^{[r, s]}$ et $v^{[p^N r, p^N s]}$, des opérateurs \[ \varphi^N \colon \Delta^{]0, p^N s]} \to \Delta^{]0, s]}, \;\;\; \psi^N \colon \Delta^{]0, s]} \to \Delta^{]0, p^N s]}, \] et, pour $N > 0$ et $0 < r < s < r(\Delta)$, $C^{[r, s]}_\partial$ la norme de l'opérateur $\partial$ pour la valuation $v^{[r, s]}$, de sorte que $C_{\varphi^N}^{[r, s]}, C_{\psi^N}^{[r,s]} > -\infty$ et $C^{[r, s]}_\partial \geq - s - 2$.

Pour $i \in (\Z / p^N \Z)^\times$, on pose  $$g_j(\kappa) = {\omega_\kappa \choose j} \kappa(i) i^{-j} (1+T)^i p^{Nj} \varphi^N(\partial^j z_i).  $$ Supposons $s < r(\Delta) p^{-N}$, on a alors
\begin{eqnarray*}
v^{[r, s]}(\varphi^N (\partial^j z_i)) &\geq& v^{[p^Nr, p^Ns]}(\partial^j \psi^N ((1 + T)^{-i} z)) + C^{[r, s]}_{\varphi^N} \\
&\geq& v^{[p^Nr, p^Ns]}(\psi^N ((1 + T)^{-i} z)) + C^{[r, s]}_{\varphi^N} + j C^{[p^N r, p^N s]}_\partial \\
&\geq& v^{[r, s]}(z) + C^{[r, s]}_{\varphi^N} + j C^{[p^N r, p^N s]}_\partial + C^{[r, s]}_{\psi^N} \\
&\geq& v^{[r, s]}(z) + C^{[r, s]}_{\varphi^N} + j (- p^N s - 1) + C^{[r, s]}_{\psi^N}.
\end{eqnarray*}

Remarquons d'ailleurs les estimations suivantes évidentes:
\begin{itemize}
\item $v^{[r,s]}((T + 1)^i) = 0.$
\item $v^{[r,s]}({\omega_\kappa \choose j}) = v_p({\omega_\kappa \choose j}) \geq j(\inf(v_p(\omega_\kappa), 0) - \frac{1}{p - 1}) = j C_\kappa. $ Or $ \omega_\kappa = \frac{\log z_\kappa}{q}$ et donc \[ v_p(\omega_\kappa) \geq \inf_{k \geq 0} \{ p^k v_p(z_\kappa - 1) - k \} -  v_p(q). \] Or, si $\kappa \in \mathfrak{X}_n$, la constante $C_\kappa$ est bornée en fonction de $n$. En effet, si $v_p(z_\kappa - 1) \geq \frac{1}{n}$, la formule pour $\omega_\kappa$ donne $C_\kappa \geq C_n = \min ( \inf_{k \geq 0} \{ \frac{p^k}{n} - k\} - v_p(q), 0) - \frac{1}{p - 1}$. On en déduit donc \[ v_p({\omega_\kappa \choose j}) \geq j C_n .\]
\item $v^{[r,s]}(i^{-j}) = v_p(i^{-j}) = 0.$
\item $v^{[r,s]}(\kappa(i)) = v_p(\kappa(i)) = 0.$
\end{itemize}

On en déduit:
\[ \inf_{\kappa \in \mathfrak{X}_n} v^{[r,s]}(g_j(\kappa)) \geq v^{[r, s]}(z) + C^{[r, s]}_{\varphi^N} + C^{[r, s]}_{\psi^N} + j (C_n + N - p^N s - 1). \]

Par définition de fonction analytique sur $\mathfrak{X}$ à valeurs dans $\Delta$, il faut montrer que, pour tout $n > 0$, il existe $s > 0$ et $N > 0$ tel que, pour tout $r \in ]0, s]$, l'expression (\ref{form2}) est convergente pour la valuation naturelle de $\mathcal{O}(\mathfrak{X}_n) \widehat{\otimes} \Delta^{[r,s]}$, et que les éléments dans $\mathcal{O}(\mathfrak{X}_n) \widehat{\otimes} \Robba$ ainsi définis ne dépendent ni de $N$ ni du système de représentants de $(\Z / p^N \Z)^\times$ et sont compatibles par rapport aux applications naturelles de restriction $\mathcal{O}(\mathfrak{X}_{n+1}) \widehat{\otimes} \Robba \to \mathcal{O}(\mathfrak{X}_n) \widehat{\otimes} \Robba$.

Fixons $n$ et choisissons $N$ assez grand et $s < r(\Delta) p^{-N}$ assez petit de sorte que $(C_n + N - p^N s - 1) > 0$. On observe qu'aucun choix ne dépend de $r$. Ceci montre que la somme des $g_j(\kappa)$ converge sur $\mathcal{O}(\mathfrak{X}_n) \widehat{\otimes} \Delta^{[r, s]}$ pour tout $r < s$ et définit, pour chaque $n \geq 0$, un élément dans $\mathcal{O}(\mathfrak{X}_n) \widehat{\otimes} \Delta$.

On montre maintenant que l'expression ne dépend ni de $N$ ni du choix du système de représentants de $(\Z / p^N \Z)^\times$. Si l'on fixe $n > 0$, $N > 0$ et $s > 0$ comme ci-dessus et un système de représentants de $(\Z / p^N \Z)^\times$, l'expression \eqref{form2} définit une fonction rigide analytique sur $\mathfrak{X}_n$. De plus, les valeurs de ces fonctions aux caractères $x^k \in \mathfrak{X}_n$, $k \in \Z$, ne dépendent pas du choix de $N$. On conclut en remarquant que ces caractères sont Zariski denses dans $\mathfrak{X}_n$.

Enfin, les fonctions définies ne dépendent évidement pas de $n$, et on conclut donc que (\ref{form2}) définit un élément de $\mathcal{O}(\mathfrak{X}) \widehat{\otimes} \Delta^{\psi = 0}.$
\end{proof}

\begin{corollary} \label{coro2}
Soit $N(n)$ le plus petit $N$ qui fait que la formule \eqref{form2} soit bien définie sur $\mathfrak{X}_n$ et soit $\kappa \in \mathfrak{X}_n$. Alors, pour tous $0 < r < s < r(\Delta) p^{-N(n)}$, $\kappa(\partial)$ stabilise $\Delta^{[r, s]}$.
\end{corollary}

\subsubsection{Bases et modules de de Rham} \label{basesdeRham2}

Les éléments définis ci-dessous permettront de simplifier considérablement les notations futures ainsi que de clarifier certains facteurs apparaissant dans les formules d'interpolation  (on y reviendra dans \S \ref{basesdeRham3}).

Si $\xi \colon \qpe \to L^\times$ est un caractère, on note $e_\xi$ une base du $L$-espace vectoriel $L(\xi)$ de dimension $1$ muni d'actions de $\varphi$ et $\Gamma$ via les formules $\varphi(e_\xi) = \xi(p) \cdot e_\xi$ et $\sigma_a(e_\xi) = \xi(a) \cdot e_\xi$, $a \in \zpe$. Si $D$ est un $(\varphi, \Gamma)$-module, on note $D(\xi) = D \otimes \xi$ le module $D \otimes_L L(\xi)$ (c'est le tordu de $D$ par $\xi$). Le choix de $e_\xi$ fournit un isomorphisme de $L$-espaces vectoriels $D \xrightarrow{\sim} D(\xi): x \mapsto x \otimes e_\xi$. Si $\eta \colon \zpe \to L^\times$ est un caractère, les mêmes constructions s'appliquent sur $\eta$ en regardant $\eta$ comme un caractère sur $\qpe$ en posant $\eta(p) = 1$.

Soit $\eta \colon \zpe \to L^\times$ un caractère localement constant. L'élément $G(\eta) e_\eta \in L_\infty(\eta)$ est fixé par l'action de $\Gamma$ et on a donc un isomorphisme $\DdR(\Robba(\eta)) = (L_\infty((t)) \cdot e_\eta)^\Gamma = L \cdot G(\eta) e_\eta$, ce qui nous fournit un générateur $\mathbf{e}^{\rm dR}_\eta = G(\eta) e_\eta$ de ce module. Si $j \in \Z$, on rappelle que l'on a noté $e_j = e_{\chi^j}$ une base du module $L(\chi^j)$, de sorte que $\mathbf{e}^{\rm dR}_j =  t^{-j} e_j$ est une base de $\DdR(\Robba(\chi^j))$. Notons \[ \mathbf{e}_{\eta, j} = e_\eta \otimes e_j, \] qui est une base de $L(\eta \chi^j)$. L'élément \[ \mathbf{e}^{\rm dR}_{\eta, j} =  \mathbf{e}^{\rm dR}_\eta \otimes \mathbf{e}^{\rm dR}_j = G(\eta) t^{-j} \cdot \mathbf{e}_{\eta, j} = G(\eta) e_\eta \otimes t^{-j} e_j \] constitue une base du module $\DdR(\Robba(\eta \chi^j))$.

Si $D \in \Phi\Gamma(\Robba)$ est de Rham, alors $D(\eta \chi^j)$ l'est aussi et on a, par ce qui précède, \[ \DdR(D(\eta \chi^j)) = \Ddif( D \otimes L(\eta \chi^j))^\Gamma = \DdR(D) \otimes L \cdot \mathbf{e}^{\rm dR}_{\eta, j}, \] de sorte que l'application $x \mapsto x \otimes \mathbf{e}^{\rm dR}_{\eta, j}$ induit un isomorphisme $\DdR(D) \xrightarrow{\sim} \DdR(D(\eta \chi^j))$.

Enfin, on note $e_\eta^\vee = e_{\eta^{-1}}$, $e_j^\vee = e_{-j}$ les éléments duaux, respectivement, de $e_\eta$ et $e_j$, ainsi que \[ \mathbf{e}^{\rm dR, \vee}_{\eta, j} = G(\eta)^{-1} \cdot e_\eta^\vee \otimes t^j e_j^\vee, \] base du module $\DdR(\Robba(\eta^{-1} \chi^{-j}))$. L'application $x \mapsto x \otimes \mathbf{e}^{\rm dR, \vee}_{\eta, j}$ induit un isomorphisme $\DdR(D(\eta \chi^j)) \xrightarrow{\sim} \DdR(D)$ et on a $x \otimes \mathbf{e}^{\rm dR}_{\eta, j} \otimes \mathbf{e}^{\rm dR, \vee}_{\eta, j} = x$ pour tout $x \in \DdR(D)$.

\subsubsection{Interpolation} \label{interpolation}

Soient $D \in \Phi \Gamma(\Robba)$ de Rham, $\Delta = \Nrig(D)$, $\eta$ un caractère de Dirichlet de conducteur $p^m$, $m \geq m(\Delta)$ et $\kappa(x) = z_{\kappa}^{\frac{\log x}{q}}$. On pose
\begin{equation} \label{forminterp}
\Lambda_{D, z}(\eta \kappa) = \mathrm{G}(\eta)^{-1} \sum_{a \in (\Z / p^m \Z)^\times} \eta(a) \sigma_a [ \varphi^{-m} \kappa(\partial) (1 - \varphi) z ]_0.
\end{equation}

L'application $\varphi^{-m}$ est définie sur $\Delta^{]0, r_m]}$ et l'expression ci-dessus n'a donc un sens que pour les caractères $\kappa$ tels que $\kappa(\partial)(1 - \varphi) z \in \Delta^{]0, r_m]}$ et, dans ce cas, on a  $\Lambda_{D, z}(\eta \kappa) \in L_m \otimes \DdR(D)$. Nous allons montrer que $\Lambda_{D, z}$ est bien définie dès que $\eta$ est un caractère de Dirichlet assez ramifié et $\kappa$ est un caractère vivant dans un certain voisinage ouvert du caractère trivial et que les valeurs interpolées par $\Lambda_{D, z}$ aux caractères spéciaux $\eta \chi^{\j}$, ${\j} \in \Z$, sont reliées à des valeurs arithmétiquement intéressantes. Le théorème principal de ce chapitre est le suivant:

\begin{theorem} \label{theo1}
Soient $D \in \Phi\Gamma(\Robba)$ de Rham, $z \in \Delta^{\psi = 1}$ et $h \in \Z$ tel que $\mathrm{Fil}^{-h} \, \D_{\mathrm{dR}}(D) = \D_{\mathrm{dR}}(D)$. Il existe une constante $N(D)$ (ne dépendant que de l'extension $K$ fournie par le théorème de monodromie $p$-adique et de $p$) \footnote{cf. le lemme \ref{lemme6} ci-dessous pour le calcul de la constante.}, telle que la formule \eqref{forminterp} définit une fonction rigide analytique $\Lambda_{D, z} \in \mathcal{O}( \mathfrak{U}_D ) \otimes \DdR(D)$, où $\mathfrak{U}_D = \cup_{\mathrm{c}(\eta) > m(\Delta)} \mathfrak{B}(\eta, N(D))$ (cf. \S \ref{sectresprinc}). De plus, pour tout caractère $\eta \chi^j \in \mathfrak{U}_D$, où $\eta$ est un caractère de conducteur $p^n$, $n > 0$, et $j \in \Z$, on a

\[ \Lambda_{D, z} (\eta \chi^{\j}) = \Gamma^*(j - h + 1) p^{n(j + 1)} \cdot \left\{ 
  \begin{array}{l l}
   \exp^*(\int_{\Gamma} \eta \chi^{-j} \cdot \mu_{\nabla_h z}) \otimes \mathbf{e}^{\rm dR, \vee}_{\eta, -j} & \quad \text{si ${\j} \geq h$} \\
 \exp^{-1}(\int_{\Gamma} \eta \chi^{-j} \cdot \mu_{\nabla_h z}) \otimes \mathbf{e}^{\rm dR, \vee}_{\eta, -j}  & \quad \text{si ${\j} \ll 0$,}
  \end{array} \right.\] De plus, si $D$ est cristallin, l'application $\Lambda_{D, z}$ provient par restriction d'une fonction rigide analytique définie sur tout l'espace des poids. 
\end{theorem}

\begin{remarque} \leavevmode
\begin{itemize}
\item L'application $\mathrm{Exp}^*$ étant $\Lambda_\infty$-linéaire, on a $\mu_{\nabla_h z} = \nabla_h \mu_z$.
\item La preuve du théorème conste de trois parties. Dans la section \ref{calculrayon}, on calcule la constante $N(D)$, qui fournit l'ouvert de définition de $\Lambda_{D, z}$, et, dans les sections \ref{interexpd} et \ref{interexp}, on montre les propriétés d'interpolation. Notons que, si l'on définit $\log$ comme $\exp^{-1}$ ou $\exp^*$ selon le cas, on obtient un énoncé d'interpolation uniforme pour tous les entiers $j \in \Z$. 
\end{itemize}
\end{remarque}

\subsubsection{Calcul du rayon de convergence} \label{calculrayon}

Si $\eta$ est un caractère de conducteur $p^m$, $\kappa(x) = z_{\kappa}^{\frac{\log x}{q}}$ alors, pour que la formule \eqref{forminterp} définissant l'application $\Lambda_{D, z}(\eta \kappa)$ ait un sens, il suffit que $\kappa(\partial)(1 - \varphi) z \in \Delta^{]0, r_m]}$. Ceci impose des conditions sur la valeur de $N$ dans la définition de $\kappa(\partial)$, comme le montre le corollaire \ref{coro2}. Le lemme suivant calcule, pour un $\eta$ fixe, le rayon de convergence autour $\eta$ de cette formule, ce qui décrit l'ouvert $\mathfrak{U}_D$ de définition de l'application $\Lambda_{D, z}$.


\begin{lemme} \label{lemme6}
Soient $z \in \Delta^{\psi = 1}$ et $\eta \colon \zpe \to L^\times$ un caractère de conducteur $p^m$, $m > m(\Delta)$. Il existe une constante $N(D) = C_p + m(\Delta)$, où $C_p$ est une constante qui ne dépend que de $p$, telle que la formule définissant $\Lambda_{D, z}(\eta \kappa)$ est bien définie dès que $v_p(z_{\kappa} - 1) > p^{N(D) - m}$.
\end{lemme}

\begin{proof}
L'élément $\kappa(\partial) (1 - \varphi) z$ est défini comme somme des $$ g_j(\kappa) =  {\omega_\kappa \choose j} \kappa(a) a^{-j} (1+T)^a p^{Nj} \varphi^N(\partial^j z_a), $$ pour $N$ assez grand. Ces éléments appartiennent \footnote{Rappelons que, comme $z \in \Delta^{\psi = 1}$, alors $z \in \Delta^{]0, r_{m(\Delta)}]}$, et donc $\varphi^N(\partial^j z_a) \in \Delta^{]0, r_{m(\Delta) + N + 1}]}$.} à $\Delta^{]0, r_m]}$ si $N < m - m(\Delta)$. Prenons donc $N = m - m(\Delta) - 1$.

Soit $N(D) =  C_p +  m(\Delta)$, avec \[ C_p = - 1 / \log p - \log \log p + v_p(q) + \frac{1}{p - 1} + 3, \] et montrons que, si $v_p(z_{\kappa} - 1) > p^{N(D) - m}$, la somme définissant $\kappa(\partial) (1 - \varphi) z$ converge.

Par la démonstration de la proposition \ref{prop4}, on a \[ v^{[r, r_m]}(g_j(\kappa)) \geq C + j (\frac{1}{j} v_p \big( {\omega_\kappa \choose j} \big) + N + C_\partial^{[p^N r, p^N r_m]}), \] où $C$ est une constante qui ne dépend pas de $j$ et $C_\partial^{[p^N r, p^N r_m]} \geq - p^N r_m - 1 = - r_{m(\Delta) + 1} - 1$. On se ramène donc à montrer que $$ \frac{1}{j} v_p\big( {\omega_\kappa \choose j} \big) > m(\Delta) - m + r_{m(\Delta) + 1} + 2. $$

On a $\omega_\kappa = \frac{\log \kappa( \exp(q))}{ q } = \frac{\log z_\kappa}{ q }$, donc $v_p(\omega_\kappa) = v_p( \log z_\kappa ) - v_p(q) \geq \inf \{ v_p (z_\kappa - 1) p^k - k \} - v_p(q)$. On a deux cas:

\begin{itemize}
\item Si $v_p(\omega_\kappa) \geq 0$, alors la condition est automatiquement satisfaite dès que $m > m(\Delta) + 3$ .
\item Supposons que $v_p(\omega_\kappa) < 0$. On a $ v_p\big( {\omega_\kappa \choose j} \big) = j v_p(\omega_\kappa) - v_p(j!) \geq j(v_p(\omega_\kappa) - \frac{1}{p - 1})$. Pour montrer l'inégalité ci-dessus, il suffit donc de montrer que $$ \inf_{k \geq 0} \{ v_p(z_\kappa - 1) p^k - k \} > v_p(q) + \frac{1}{p - 1} +  m(\Delta) - m + r_{m(\Delta) + 1} + 2. $$

Ceci revient à montrer que, pour tout $k \geq 0$, on a $v_p(z_\kappa - 1) p^k - k > v_p(q) + \frac{1}{p - 1} +  m(\Delta) - m + r_{m(\Delta) + 1} + 2$, ou, de manière équivalente, $$ v_p(z_\kappa - 1) > p^{-k}(v_p(q) + \frac{1}{p - 1} +  m(\Delta) - m + r_{m(\Delta) + 1} + 2 + k) = p^{-k}(C + k). $$ La fonction d'une variable réelle $f(x) = p^{-x}(C + x)$ atteint son maximum absolu en $x = 1/ \log p - C$. On a donc $$ p^{-k}(v_p(q) + \frac{1}{p - 1} +  m(\Delta) - m + r_{m(\Delta) + 1} + 2 + k) \leq p^{m(\Delta) + C_p - m}, $$ ce qui permet de conclure.
\end{itemize}
\end{proof}

\begin{remarque} \leavevmode
\begin{itemize}
\item Le lemme ci-dessus nous permet de décrire l'ouvert $\mathfrak{U}_D \subseteq \mathfrak{X}$ de définition de $\Lambda_{D, z}$: il est une union de boules centrées sur les points $\zeta_{p^m}$, $m > m(\Delta)$, de rayon $p^{-p^{N(D) - m}}$ (qui tend vers $1$ quand $\mathrm{c}(\eta) \to +\infty$).
\item Si $\zeta_{p^m}^b$, $\zeta_{p^m}^a$, $1 \neq a, b \in (\Z / p^m \Z)^\times$, sont deux racines primitives de l'unité (correspondant au choix de deux caractères d'ordre fini du même conducteur), alors $v_p(\zeta_{p^m}^b - \zeta_{p^m}^a) = v_p(\zeta_{p^m}^{b - a} - 1) = \frac{1}{p^{m - 1 - v_p(b - a)}(p - 1)} = p^{v_p(b - a)} r_{m}$. On en déduit que les boules sont disjointes dès que $v_p(b - a) < N(D)$.
\item Un entier $j \in \Z$, $j \equiv 0 \text{ (mod } p - 1)$, correspond au caractère $\chi^j$ et $z_{j} = z_{\chi^j} = (1 + 2p)^j$ \footnote{On a posé $z_{\chi^j} = (\exp q)^j$, mais cela ne change rien si l'on choisit un autre générateur de $1 + 2p \zp$, comme par exemple $1 + 2p$.}. On en déduit que $v_p(z_j - 1) > p^{N - \mathrm{c}(\eta)}$ si $v_p(j) \geq p^{N - \mathrm{c}(\eta)}$. L'ouvert $\mathfrak{B}(\eta, N)$ contient alors tous les caractères de la forme $\eta \chi^j$, $v_p(j) \geq p^{N - \mathrm{c}(\eta)}$. En particulier, si $N - \mathrm{c}(\eta) < 0$, il contient tous les caractères $\eta \chi^j$, avec $(p - 1)p \mid j$.
\end{itemize}
\end{remarque}

\subsubsection{Interpolation des applications exponentielles duales} \label{interexpd}

On commence par quelques résultats préliminaires. On pourra comparer le lemme suivant avec \cite[Lem. II.1]{Berger03} (pour le cas cristallin et le caractère trivial).

\begin{lemme} \label{lemme7}
Soient $D \in \Phi\Gamma(\Robba)$ de Rham, $z \in \Delta^{\psi = 1}$ et $\eta \colon \zpe \to L^\times$ un caractère de conducteur $p^n$, $n \geq m(\Delta)$ et $m \geq 0$. Alors

$(i)$ Si $m(\Delta) \leq m < n$, on a $$ \sum_{a \in (\Z / p^n \Z)^\times} \eta(a) \sigma_a(\left[ \partial^{\j} \varphi^{-m} z \right]_0 ) = 0, $$

$(ii)$ L'expression $ \sum_{a \in (\Z / p^m \Z)^\times} \eta(a) \sigma_a( p^{- m} \mathrm{Tr}_{L_{m} / L_n} \left[ \partial^{\j} \varphi^{-m} z \right]_0 )$ ne dépend pas de $m \geq n$.
\end{lemme}

\begin{proof}
Par définition de $\Delta$, on a $ \varphi^{-m} z \in L_m[[t]] \otimes \DdR(D)$ et $ p^{-k} \mathrm{Tr}_{L_{m+k} / L_m} \varphi^{-(m+k)} z = \varphi^{-m} z$ car $z$ est fixé par $\psi$, d'où $\varphi^{-m} z$ peut être exprimé comme $\sum_{h \geq 0} \sum_{b \in (\Z / p^m \Z)} (\zeta_{p^m}^b \otimes d_{h, b}) t^h$, avec $d_{h, b} \in \DdR(D)$ (l'expression n'est évidement pas unique).

Montrons le premier point. Il suffit, par linéarité, de montrer le résultat sous l'hypothèse que $[\partial^{\j} \varphi^{-m} z]_0 = \zeta_{p^m}^b \otimes d$, avec $0 \leq b \leq p^m - 1, d \in \DdR(D)$. Or, $$ \sum_{a \in (\Z / p^n \Z)^\times} \eta(a) \sigma_a( \zeta_{p^m}^b \otimes d )  = \sum_{a \in (\Z / p^n \Z)^\times} \eta(a) \zeta_{p^m}^{ab} \otimes d = p^n \hat{\eta}(-b / p^{m}) \otimes d. $$ On en déduit le résultat en remarquant que $\hat{\eta}$ est une fonction à support dans $p^{-n} \zpe$ et que $v_p(-b / p^{m}) > -n$.

En ce qui concerne le dernier point, on peut supposer par linéarité que $\mathrm{Tr}_{L_{m} / L_n} [ \partial^{\j} \varphi^{-m} z ]_0 = \zeta_{p^n}^b \otimes d$, où $b \in (\Z / p^n \Z)$ et $d \in \DdR(D)$. On a
$$ p^{-m} \sum_{a \in (\Z / p^m \Z)^\times} \eta(a) \sigma_a( \zeta_{p^n}^b \otimes d )  = p^{-m} \sum_{a \in (\Z / p^m \Z)^\times} \eta(a) \zeta_{p^n}^{ab} \otimes d = \hat{\eta}(-b / p^n) \otimes d = p^{- n} G(\eta) \eta^{-1}(b) \otimes d. $$
L'expression ci-dessus ne dépend pas de $m$, d'où le résultat.
\end{proof}

\begin{remarque}
Si $D$ est cristallin, $z = \sum_k a_k T^k \in (\Robba^+ \otimes \Dcris(D))^{\psi = 1} \subseteq (\Delta[1/t])^{\psi = 1}$ et $n = 0$ dans le lemme ci-dessus, on a, d'après \cite[Lem. II.1]{Berger03}, \[ p^{-m} \mathrm{Tr}_{L_m / L} [\varphi^{-m} z]_0 = (1 - p^{-1} \varphi^{-1}) a_0.\]
\end{remarque}

\begin{lemme} \label{lemme1b} Soient $z \in D^{\psi = 1}$, $\eta$ un caractère constant modulo $p^n$ avec $n \geq m(\Delta)$ et $m \geq n$. On a alors l'égalité suivante dans $L_m \otimes H^1_{\varphi, \gamma}(D(-{\j}))$:
$$\int_{\Gamma} \eta \chi^{-j} \cdot \mu_z = \sum_{a \in (\Z / p^m \Z)^{\times}} \eta(a) a^{-{\j}} \int_{\Gamma_m} \chi^{-j} \cdot \mu_{\sigma_a(z)}. $$
De plus, si $D$ est de Rham, on a l'égalité (dans $L_m \otimes \DdR(D(-j))$) $$ \exp^* (\int_{\Gamma} \eta \chi^{-j} \cdot \mu_z ) \otimes \mathbf{e}^{\rm dR, \vee}_\eta = G(\eta)^{-1} \sum_{a \in (\Z / p^m \Z)^{\times}} \eta(a) a^{-{\j}} \exp^*_{D(-{\j})} (\int_{\Gamma_m} \chi^{-j} \cdot \mu_{\sigma_a(z)}). $$
\end{lemme}

\begin{proof}
Montrons seulement le deuxième point, vu que les mêmes techniques seront utilisées plus tard. 
Par la proposition \ref{recPRcomp} (appliquée à $D(\eta \chi^{-j})$, $j = 0$ et $n = 0$), on a
\begin{eqnarray*}
\exp^*(\int_\Gamma \eta \chi^{-j} \cdot \mu_z ) &=& p^{-m} \mathrm{Tr}_{L_m / L} ([\varphi^{-m}(z \otimes e_\eta \otimes e_{-j}) ]_0) \\
&=&p^{-m} \mathrm{Tr}_{L_m / L}(G(\eta)^{-1} [t^{-j} \varphi^{-m} z]_0) \otimes G(\eta) e_\eta \otimes t^j e_{-j},
\end{eqnarray*}
où on a utilisé que l'élément $\mathbf{e}^{\rm dR}_{\eta, -j} = G(\eta) e_\eta \otimes t^j e_{-j}$ est fixé par $\Gamma$ et commute donc à la trace. Notons que $[t^{-j} \varphi^{-m} z]_0 = [\varphi^{-m} z]_j$. La somme de Gauss s'écrit comme $$ G(\eta)^{-1} = p^{-n} \eta(-1) G(\eta^{-1}) = p^{-n} \eta(-1) \sum_{a \in (\Z / p^n \Z)^\times} \eta^{-1}(a) \zeta_{p^n}^a $$ et on a
\begin{eqnarray*} 
\mathrm{Tr}_{L_m / L}(G(\eta)^{-1}[\varphi^{-m} z]_j) &=& \eta(-1) p^{-n} \sum_{b \in (\Z / p^m \Z)^\times} \eta(b) \sum_{a \in (\Z / p^n \Z)^\times} \eta^{-1}(ab) \zeta_{p^n}^{ab} \sigma_b[\varphi^{-m} z]_j \\
&=& G(\eta)^{-1} \sum_{b \in (\Z / p^m \Z)^\times} \eta(b) \sigma_b[\varphi^{-m} z]_j \\
\end{eqnarray*}
où on a utilisé la formule pour la trace et encore une fois la formule reliant $G(\eta)$ et $G(\eta^{-1})$. On en déduit
\[ \exp^* (\int_\Gamma \eta \chi^{-j} \cdot \mu_z ) = p^{-m} G(\eta)^{-1} \sum_{b \in (\Z / p^m \Z)^\times} \eta(b) \sigma_b([\varphi^{-m} z]_j) \otimes G(\eta) e_\eta \otimes t^j e_{-j}. \] Or $\sigma_b[ \varphi^{-m} z ]_j = b^{-j} [\varphi^{-m} \sigma_b z]_j$ et, par la loi de réciprocité encore une fois (pour $D$, $j$ et $m$), on sait que $p^{-m} [\varphi^{-m} \sigma_b z]_j \otimes t^j e_{-j} = \exp^*(\int_{\Gamma_m} \chi^{-j} \cdot \mu_{\sigma_b(z)})$, d'où
\[ \exp^* (\int_\Gamma \eta \chi^{-j} \cdot \mu_z ) \otimes \mathbf{e}^{\rm dR, \vee}_\eta = G(\eta)^{-1} \sum_{b \in (\Z / p^m \Z)^\times} \eta(b) b^{-j} \exp^*(\int_{\Gamma_m} \chi^{-j} \mu_{\sigma_b(z)}) ,\]
ce qui permet de conclure.
\end{proof}

\begin{proposition} \label{lemme4} Soient $D \in \Phi\Gamma(\Robba)$ de Rham, $z \in \Delta^{\psi = 1}$, $h \in \Z$ tel que $\mathrm{Fil}^{-h} \, \D_{\mathrm{dR}}(D) = \D_{\mathrm{dR}}(D)$, $\eta \colon \zpe \to L^\times$ un caractère de conducteur $p^n$, $n \geq m(\Delta)$, et $j \geq h$ un entier. On a alors l'égalité suivante dans $L_n \otimes \DdR(D)$
$$ \Lambda_{D, z}(\eta \chi^{\j}) = (-h + {\j})! p^{n({\j} + 1)} \cdot \exp^* (\int_{\Gamma} \eta \chi^{-j} \cdot \mu_{\nabla_h z}) \otimes \mathbf{e}^{\rm dR, \vee}_{\eta, -j}. $$
\end{proposition}

\begin{proof}
La preuve du lemme \ref{lemme1b} appliqué à $\nabla_h z \in D^{\psi = 1}$ donne
\[ \exp^*(\int_{\Gamma} \eta \chi^{-j} \cdot \mu_{\nabla_h z}) \otimes \mathbf{e}^{\rm dR, \vee}_{\eta, -j} = G(\eta)^{-1} p^{-n} \sum_{a \in (\Z / p^n \Z)^\times} \eta(a) \sigma_a [\varphi^{-n} \nabla_h z]_j. \]
Or, on remarque que $(\nabla - h + 1) \circ \hdots \circ (\nabla - 1) \circ \nabla (\sum_{j \geq 0} a_j t^j) = \sum_{j \geq h} a_j j (j - 1) (j - 2) \hdots (j - h + 1) t^j$. L'opérateur $ (\nabla - h + 1) \circ \hdots (\nabla -1) $ a donc l'effet de tuer les coefficients plus petits que $h - 1$. Comme $j \geq h$, on a
\[ [\varphi^{-n} \nabla_h z]_j = \frac{j!}{(j - h)!} [\varphi^{-n} z]_j = \frac{p^{-nj}}{(j - h)!} [\varphi^{-n} \partial^j z]_0 \]
Par ailleurs, en utilisant, respectivement, le lemme \ref{lemme7} et la définition de $\Lambda_{D, z}$ (cf. \S \ref{interpolation}, \eqref{forminterp}), on obtient
\begin{eqnarray*}
\sum_{a \in (\Z / p^n \Z)^\times} \eta(a) \sigma_a [\varphi^{-n} \partial^{\j} z]_0 &=& \sum_{a \in (\Z / p^n \Z)^\times} \eta(a) \sigma_a [\varphi^{-n} \partial^{\j} z - p \varphi^{-(n-1)} \partial^{\j} z]_0 \\
&=& \sum_{a \in (\Z / p^n \Z)^\times} \eta(a) \sigma_a[\varphi^{-n} \partial^{\j} (1 - \varphi) z ]_0\\
&=& \mathrm{G}(\eta) \cdot \Lambda_{D, z}(\eta \chi^{\j}),
\end{eqnarray*}
ce qui permet de conclure.
\end{proof}

\subsubsection{Interpolation des applications exponentielles} \label{interexp}

Des calculs du même genre montrent que

\begin{proposition} \label{lemme5} Soient $D \in \Phi\Gamma(\Robba)$ de Rham, $z \in \Delta^{\psi = 1}$, $h \in \Z$ tel que $\mathrm{Fil}^{-h} \, \D_{\mathrm{dR}}(D) = \D_{\mathrm{dR}}(D)$, $\eta \colon \zpe \to L^\times$ un caractère de conducteur $p^n$ tel que $n \geq m(\Delta)$ et ${\j} \geq -h + 1$ un entier tel que l'équation $\partial^{\j} z_{\j} = z$ a une solution dans $\Delta$ et tel que $\exp_{D({\j})} \colon \DdR(D({\j})) \to H^1_{\varphi, \gamma}(D({\j}))$ soit un isomorphisme. On a alors l'égalité suivante dans $L_n \otimes \DdR(D)$
$$ \Lambda_{D, z}(\eta \chi^{-{\j}}) =  \Gamma^*(-h - j + 1) p^{n(-j + 1)} \cdot \exp^{-1}(\int_{\Gamma} \eta \chi^{\j} \cdot \mu_{\nabla_h z}) \otimes \mathbf{e}^{\rm dR, \vee}_{\eta, j}.$$
\end{proposition}

\begin{proof}
De la même façon que dans le lemme \ref{lemme1b}, on montre que
$$ \exp^{-1}(\int_{\Gamma} \eta \chi^{\j} \cdot \mu_{\nabla_h z}) \otimes \mathbf{e}^{\rm dR, \vee}_\eta = G(\eta)^{-1} \sum_{a \in (\Z / p^n \Z)^\times} \eta(a) a^{\j} \exp^{-1} (\int_{\Gamma_n} \chi^j \cdot \mu_{\nabla_h \sigma_a(z)} ). $$
Observons que, si $\partial^{\j} z_{\j} = z$, alors $\partial^{\j} (a^{-{\j}} \sigma_a z_{\j}) = \sigma_a z$. En appliquant la loi de réciprocité (théorème \ref{recPR}), on voit que l'expression ci-dessus est égale à
$$ G(\eta)^{-1} \frac{p^{n(j - 1)}}{\Gamma^*(-j - h + 1)} \sum_{a \in (\Z / p^n \Z)^\times} \eta(a) [\varphi^{-n}( \sigma_a z_{\j} )]_0 \otimes t^{-j} e_j. $$ D'où
\[ \Gamma^*(-h-j+1) p^{n(-j + 1)} \cdot \exp^{-1}(\int_{\Gamma} \eta \chi^{\j} \cdot \mu_{\nabla_h z}) \otimes \mathbf{e}^{\rm dR, \vee}_{\eta, j} = G(\eta)^{-1} \sum_{a \in (\Z / p^n \Z)^\times} \eta(a) [\varphi^{-n}( \sigma_a z_{\j} )]_0. \] 

Par le lemme \ref{lemme7} et par le fait que $\partial$ est inversible sur $\Delta^{\psi = 0}$ et qu'on a donc le droit d'écrire $(1 - p^{-{\j}} \varphi) z_{\j} = \partial^{-{\j}} (1 - \varphi) z$, on a

\begin{eqnarray*}
 \sum_{a \in (\Z / p^n \Z)^\times} \eta(a) \sigma_a [\varphi^{-n}(z_{\j})]_0 &=& \sum_{a \in (\Z / p^n \Z)^\times} \eta(a) \sigma_a [\varphi^{-n}((1 - p^{-{\j}} \varphi) z_{\j})]_0 \\
&=&  \sum_{a \in (\Z / p^n \Z)^\times} \eta(a) \sigma_a [\varphi^{-n} (\partial^{-{\j}} (1 - \varphi) z)]_0 \\
&=& \mathrm{G}(\eta) \cdot \Lambda_{D, z}(\eta \chi^{-{\j}}),
\end{eqnarray*}
ce qui permet de conclure.
\end{proof}

Ceci finit la preuve du théoreme \ref{theo1}.

\subsubsection{Le cas des poids de Hodge-Tate positifs} \label{poidspositifs}

Soit $D \in \Phi \Gamma(\Robba)$ de Rham à poids de Hodge Tate positifs et notons toujours $\Delta = \Nrig(D)$. Fixons un entier $h$ tel que $\mathrm{Fil}^{-h} \, \DdR(D) = \DdR(D)$. On a donc $D \subseteq \Delta$ et, si $z \in D^{\psi = 1}$, on peut appliquer la construction faite ci-dessus à l'élément $z$ (et $h$) pour obtenir une application $ \Lambda_{D, z} $. On commence avec quelques remarques

\begin{lemme} \label{lemmeconv} Soient $z \in D^{\psi = 1}$, $\eta \colon \Gamma \to L^\times$ un caractère localement analytique et ${\j} \in \Z$. On a $$ \int_\Gamma \eta \chi^{{\j}} \cdot \mu_{\nabla_h  z} = (- \omega_\eta - {\j} - h + 1) (- \omega_\eta - {\j} - h + 2) \hdots (- \omega_\eta - {\j})  \int_\Gamma \eta \chi^{{\j}} \cdot \mu_z. $$ En particulier, si $\eta$ est localement constant, on a $$ \int_\Gamma \eta \chi^{j} \cdot \mu_{\nabla_h z} = (- {\j} - h + 1) (- {\j} - h + 2) \hdots (-{\j})  \int_\Gamma \eta \chi^{j} \cdot \mu_z. $$
\end{lemme}

\begin{proof}
Si $a \in \zpe$, par définition de l'action de $\Gamma$ sur $D^{\psi = 1}$, on a $$ \int_\Gamma \eta \chi^{j} \cdot \sigma_a \mu_z = \eta(a)^{-1} a^{-{\j}} \int_\Gamma \eta \chi^{j} \cdot \mu_z, $$ et, si $i \in \Z$, la formule $\nabla = \lim_{a \to 1} \frac{\sigma_a - 1}{a - 1}$ donne $$ \int_\Gamma \eta \chi^{j} (\nabla - i) \cdot \mu_z = \bigg( \lim_{a \to 1} \frac{\eta(a)^{-1} a^{-{\j}} - 1}{a - 1} - i \bigg) \int_\Gamma \eta \chi^{j} \cdot \mu_z = (- \eta'(1) - {\j} - i) \int_\Gamma \eta \chi^{j} \cdot \mu_z, $$ ce qui permet de conclure car $\nabla_h = (\nabla - h + 1) \circ \hdots \circ (\nabla - 1) \circ \nabla$.
\end{proof}


\begin{lemme}
Soient $z \in D^{\psi = 1}$, $\eta$ un caractère de Dirichlet de conducteur $p^n$ et ${\j} \geq 0$ assez grand \footnote{Il suffit que $j$ soit plus grand que le plus grand poids de Hodge-Tate et assez grand de sorte que $\exp_{D(j)}$ soit bijective.}. Alors
$$ \Lambda_{D, z}(\eta \chi^{-{\j}}) = \Gamma^*(-j + 1) p^{n(-j + 1)} \cdot \exp^{-1}(\int_{\Gamma} \eta \chi^j \cdot \mu_z) \otimes \mathbf{e}^{\rm dR, \vee}_{\eta, j}. $$
\end{lemme}

\begin{proof}
C'est une conséquence directe de la proposition \ref{lemme5} et du lemme \ref{lemmeconv} ci-dessus.
\end{proof}

\begin{lemme}
Soient $z \in D^{\psi = 1}, $ $\eta$ et ${\j} \geq 0$ comme dans la proposition \ref{lemme4}. Alors
$$ \Lambda_{D, z}(\eta \chi^{\j}) = \Gamma^*(j + 1) p^{n(j + 1)} \cdot \exp^* (\int_{\Gamma} \eta \chi^{-j} \cdot \mu_z) \otimes \mathbf{e}^{\rm dR, \vee}_{\eta, -j}. $$
\end{lemme}

\begin{proof}
Si l'on part de $z \in D^{\psi = 1}$, les calculs faits dans la preuve de la proposition \ref{lemme4} marchent en posant $h = 0$ et donnent exactement le résultat cherché.
\end{proof}

En notant, comme précédemment, par $\log$ l'application $\exp^{-1}$ ou $\exp^*$ selon le cas, on peut résumer ces résultats dans la forme énoncée au début de du chapitre:

\begin{theorem} \label{propfin}
Soient $D \in \Phi \Gamma(\Robba)$ de Rham à poids de Hodge-Tate positifs et $z \in D^{\psi = 1}$. Il existe une fonction rigide analytique $\Lambda_{D, z} \in \mathcal{O}(\mathfrak{U}_D) \otimes \DdR(D)$ telle que, si $\eta \chi^j \in \mathfrak{U}_D$, où $\eta$ est un caractère de conducteur $p^n$ et $j \in \Z$ est tel que $j \geq 0$ ou $j \ll 0$, alors

\[ \Lambda_{D, z} (\eta \chi^{\j}) = \Gamma^*(j + 1) p^{n(j + 1)} \cdot \log (\int_{\Gamma} \eta \chi^{-j} \cdot \mu_z) \otimes \mathbf{e}^{\rm dR, \vee}_{\eta, -j}. \]
\end{theorem}

\subsubsection{Le cas cristallin}

Soit $D \in \Phi \Gamma(\Robba)$ cristallin de rang $d$. Supposons que les pentes de $\varphi$ sur $\Dcris(D)$ sont $< 0$. On a $\Delta \cong \Robba \otimes \Dcris(D)$ (cf. \cite[Lem. 3.16]{Nakamura}) et, si $z \in \Delta^{\psi = 1}$, alors $z \in \Robba^+ \otimes \Dcris(D)$ (cf. \cite[Prop. 2.5.2]{BergerBreuil}). Notons $\alpha_1, \hdots, \alpha_d \in L$ les valeurs propres de $\varphi$ et choisissons une base $e_i$, $1 \leq i \leq d$, de $\Dcris(D)$ telle que $\varphi(e_i) = \alpha_i e_i$ pour tout $i$. On peut alors écrire \[ z = \sum_{i = 1}^d \mathscr{A}_{\lambda_i} \otimes e_i \] pour certaines distributions $\lambda_i \in \mathscr{D}(\zp, L)$ vérifiant $\psi(\lambda_i) = \alpha_i \lambda_i$, $1 \leq i \leq d$. Quelques calculs classiques nous permettent de montrer le résultat suivant

\begin{prop}
Soient $\eta \colon \zpe \to L^\times$ un caractère de conducteur $p^n$, $n > 0$, et $j > 0$. On a alors
\[ \Lambda_{D, z} (\eta \chi^{\j}) = \sum_{i = 1}^d (\int_\zpe \eta^{-1} x^j \cdot \lambda_i) \otimes \alpha_i^{-n} e_i. \]
\end{prop}

En regardant le terme de droite de cette proposition, on en déduit

\begin{cor}
Soit $D \in \Phi \Gamma(\Robba)$ cristallin, alors la fonction $\Lambda_{D, z}$ provient par restriction d'une fonction définie sur tout l'espace des poids.
\end{cor}

\subsection{Équation fonctionnelle en dimension $2$} \label{equationfonctionnelle}

Dans cette section, nous utilisons le résultat principal de \cite{epsilonKato} (où on pourra trouver une exposition plus détaillée des résultats ainsi que ses preuves) pour en déduire une équation fonctionnelle satisfaite par la fonction $L$ locale quand le $(\varphi, \Gamma)$-module est de dimension $2$, de Rham et non-triangulin. On commence par fixer un certain nombre de notations.

\subsubsection{Notations} \label{basesdeRham3} Soit $D \in \Phi\Gamma^{\text{ét}}(\Robba)$ de dimension $2$, de Rham à poids de Hodge-Tate $0$ et $k$ que l'on suppose non triangulin \footnote{D'après Kedlaya, tout $(\varphi, \Gamma)$ module sur $\Robba$ est étale à torsion près par un caractère ou triangulin. Dans ce dernier cas les calculs qui suivent sont déjà connus et l'hypothèse de non triangularité n'est donc pas vraiment restrictive.} et notons $\Delta = \Nrig(D)$, qui est à poids de Hodge-Tate tous nuls. 

Étendons un peu les notations de \ref{basesdeRham2}. Si $\delta \colon \qpe \to L^\times$ est un caractère, on note $e_\delta$ une base du module $L(\delta)$ muni d'actions de $\varphi$ et $\Gamma$ via les formules $\varphi(e_\delta) = \delta(p) \cdot e_\delta$ et $\sigma_a(e_\delta) = \delta(a) \cdot e_\delta$, $a \in \zpe$. On note $D(\delta) = D \otimes \delta$ le module $D \otimes_L L(\delta)$. Le choix de $e_\delta$ fournit un isomorphisme de $L$-espaces vectoriels $D \xrightarrow{\sim} D(\delta): x \mapsto x \otimes e_\delta$.

Soit $\check{D} = \mathrm{Hom}_{\varphi, \Gamma}(D, \Robba \frac{dT}{1 + T})$ le dual de Tate de $D$, où $\Robba \frac{dT}{1 + T}$ est le $(\varphi, \Gamma)$-module étale libre de rang $1$ de base $\frac{dT}{1 + T}$ sur lequel $\varphi$ et $\Gamma$ agissent par les formules $\gamma (\frac{dT}{1 + T}) = \chi(\gamma) \frac{dT}{1 + T}$ si $\gamma \in \Gamma$, $\varphi(\frac{dT}{1 + T}) = \frac{dT}{1 + T} $. On note \[ \langle \; , \; \rangle \colon \check{D}_{\rm rig} \times D_{\rm rig} \to \Robba \frac{dT}{1+T} \] l'accouplement naturel. Soient $\omega_D = (\det_D) \chi^{-1}$ et $\omega_\Delta = (\det_\Delta) \chi^{-1}$. Le fait que $D$ soit de dimension $2$ nous permet d'identifier $\check{D} = D \otimes \omega_D^{-1}$. Comme les poids de Hodge-Tate de $D$ sont $0$ et $k$, et ceux de $\Delta$ sont nuls, le caractère $\det_\Delta$ est localement constant et $\omega_D = \omega_\Delta x^k$ (et $\det_D = x^k \det_\Delta$). Notons $e_D = e_{\det_D}$.

Soit $\eta \colon \zpe \to L^\times$ est un caractère localement constant, vu comme un caractère de $\qpe$ en posant $\eta(p) = 1$. Rappelons que l'on a un générateur $\mathbf{e}^{\rm dR}_\eta = G(\eta) e_\eta$ du module $\DdR(\Robba(\eta))$, que $e_\eta^\vee = e_{\eta^{-1}}$ dénote la base de $L(\eta)^*$ duale de $e_\eta$, et que $\DdR(\Robba(\eta)^*) =  L \cdot G(\eta)^{-1} e_\eta^\vee = L \cdot G(\eta^{-1}) e_{\eta^{-1}}$. Ceci fournit deux bases $\mathbf{e}^{\rm dR, \vee}_\eta = G(\eta)^{-1} e_\eta^\vee$ et $\mathbf{e}^{\rm dR}_{\eta^{-1}} = G(\eta^{-1}) e_{\eta^{-1}}$ du module $\DdR(\Robba(\eta)^*) = \DdR(\Robba(\eta^{-1}))$, reliées par la formule $p^n \eta(-1) \mathbf{e}^{\rm dR, \vee}_\eta = \mathbf{e}^{\rm dR}_{\eta^{-1}}$. 

On aura besoin de jongler un peu avec des éléments habitant dans le module de de Rham des différents tordus de $D$ et de son dual de Tate et les identifications suivantes permettent de voir tous ces éléments dans $\DdR(D)$. Fixons une base $f_1, f_2$ de $\DdR(D)$ et notons $$ \langle \;,\; \rangle_{\rm dR} \colon \DdR(D) \times \DdR(D) \to L, $$ le produit scalaire défini par la formule $\langle a_1 f_1 + a_2 f_2, b_1 f_1 + b_2 f_2 \rangle_{\rm dR} = a_1 b_1 + a_2 b_2$.

L'isomorphisme $\wedge^2 D = (\Robba \frac{dT}{1 + T}) \otimes \omega_D$ induit un isomorphisme $\wedge^2 \DdR(D) = \DdR((\Robba \frac{dT}{1 + T}) \otimes \omega_D) = (t^{-k} L_\infty e_D)^\Gamma$. On définit $\Omega \in L_\infty$ par la formule $f_1 \wedge f_2 = (t^k \Omega)^{-1} e_{D}$, ce qui nous permet de fixer les bases $(t^k \Omega)^{-1} e_D$ et $t^k \Omega e_D^\vee$ du module $\DdR(\wedge^2 D)$ et de son dual. On fixe aussi les bases $\mathbf{e}^{\rm dR}_{\omega_D} = (t^{k-1} \Omega)^{-1} e_{\omega_D}$ et $\mathbf{e}^{\rm dR, \vee}_{\omega_D} = (t^{k-1} \Omega) e_{\omega_D}^\vee$ du module $\DdR(\Robba(\omega_D))$ et de son dual.


Enfin, afin d'alléger les notations dans les calculs futurs, notons, pour $\eta$ comme ci-dessus et $j \in \Z$,
$$ \mathbf{e}_{\eta, j, \omega_D^\vee} \defeq e_\eta \otimes e_j \otimes e_{\omega_D}^\vee, \; \; \; \mathbf{e}_{\eta, j} \defeq e_\eta \otimes e_j, $$
bases de $L(\eta \chi^j \omega_D^{-1})$ et $L(\eta \chi^j)$ respectivement, et leurs duales
$$ \mathbf{e}^\vee_{\eta, j, \omega_D^\vee} \defeq e_\eta^\vee \otimes e_{-j} \otimes e_{\omega_D}, \; \; \; \mathbf{e}^\vee_{\eta, j} \defeq e_\eta^\vee \otimes e_{-j}, $$
ainsi que des bases des module $\DdR(\Robba(\eta \chi^j \omega_D^{-1}))$ et $\DdR(\Robba(\eta \chi^j))$
$$ \mathbf{e}^{\rm dR}_{\eta, j, \omega_D^\vee} \defeq G(\eta) e_\eta \otimes t^{-j} e_j \otimes t^{k-1} \Omega e_{\omega_D}^\vee = \mathbf{e}^{\rm dR}_\eta \otimes \mathbf{e}^{\rm dR}_j \otimes \mathbf{e}^{\rm dR, \vee}_{\omega_D}, $$
$$ \mathbf{e}^{\rm dR}_{\eta, j} \defeq G(\eta) e_\eta \otimes t^{-j} e_j = \mathbf{e}^{\rm dR}_\eta \otimes \mathbf{e}^{\rm dR}_j $$
et leurs duales
$$ \mathbf{e}^{\rm dR, \vee}_{\eta, j, \omega_D^\vee} \defeq G(\eta)^{-1} e_\eta^\vee \otimes t^{j} e_{-j} \otimes (t^{k-1} \Omega)^{-1} e_{\omega_D} = \mathbf{e}^{\rm dR, \vee}_\eta \otimes \mathbf{e}^{\rm dR}_{-j} \otimes \mathbf{e}^{\rm dR}_{\omega_D}, $$
$$ \mathbf{e}^{\rm dR, \vee}_{\eta, j} \defeq G(\eta)^{-1} e_\eta^\vee \otimes t^j e_{-j} = \mathbf{e}^{\rm dR, \vee}_\eta \otimes \mathbf{e}^{\rm dR}_{-j}, $$ et les variantes évidentes que l'on puisse imaginer.

Par exemple, si $\eta \colon \zpe \to L^\times$ est un caractère d'ordre fini, si $j \geq 0$ et si $x \in \DdR(\check{D}(\eta \chi^{-{\j}}))$, on écrira $x \otimes \mathbf{e}^{\rm dR, \vee}_{\eta, -j, \omega_D^\vee} \in \DdR(D)$ l'image de $x$ par l'isomorphisme $$ \DdR(\check{D}(\eta \chi^{-{\j}})) \xrightarrow{\sim} \DdR(D); \;\;\; x \mapsto x \otimes G(\eta)^{-1} e_\eta^\vee \otimes t^{-j} e_{j} \otimes (t^{k - 1} \Omega)^{-1} e_{\omega_D} $$ et de même, si $x \in \DdR(D(\eta^{-1} \chi^{\j}))$, on notera $x \otimes \mathbf{e}^{\rm dR, \vee}_{\eta^{-1}, j} \in \DdR(D)$ l'image de $x$ par l'isomorphisme $$ \DdR(D(\eta^{-1} \chi^j)) \xrightarrow{\sim} \DdR(D); \;\;\; x \mapsto x \otimes G(\eta^{-1})^{-1} e_{\eta^{-1}}^\vee \otimes t^j e_{-j}. $$

\begin{remarque} \leavevmode
\begin{itemize}
\item Les bases des modules de de Rham ainsi définies héritent une action de l'opérateur $\varphi$. On a, par exemple, \[ \varphi(\mathbf{e}^{\rm dR}_\eta) = \mathbf{e}^{\rm dR}_\eta, \;\;\; \varphi(\mathbf{e}^{\rm dR}_j) = p^{-j} \mathbf{e}^{\rm dR}_j, \] \[ \varphi(\mathbf{e}^{\rm dR}_{\eta, j, \omega_D^\vee}) = p^{-j + k - 1} \omega_D^{-1}(p) \cdot \mathbf{e}^{\rm dR}_{\eta, j,\omega_D^\vee} = p^{-j - 1} \omega_\Delta^{-1}(p) \cdot \mathbf{e}^{\rm dR}_{\eta, j,\omega_D^\vee}. \]
\item Il faut faire un peu d'attention et distinguer le caractère identité $x$ et le caractère cyclotomique $\chi = x |x|$. Les deux coïncident sur $\zpe$ mais $\chi(p) = 1$, tandis que le premier prend la valeur $p$. Par exemple, $\Gamma$ agit trivialement sur l'élément $e_j \otimes e_{x^j}^\vee$ mais $\varphi(e_j \otimes e_{x^j}^\vee) = p^{-j} \; e_j \otimes e_{x^j}^\vee$.
\end{itemize}
\end{remarque}

\subsubsection{Facteurs epsilon pour $\mathrm{GL}_1$} \label{replisses}

Commençons par rappeler la définition des facteurs locaux associés à un caractère. Soit $\eta \colon \qpe \to L^\times$ un caractère continu. On dit que $\eta$ est non ramifié si sa restriction à $\zpe$ est triviale et il est ramifié dans le cas contraire. On définit son conducteur par $0$ s'il est non ramifié, et par $p^n$, où $n$ est le plus petit entier tel que la restriction $\eta \mid_{1 + p^n \zp}$ soit triviale, dans le cas contraire. Notons (cf. \cite{BH}, \S 6.23, \cite[\S 1.1]{Schmidt})
$$ \epsilon(\eta, s) = \left\{
  \begin{array}{cc}
    1 & \quad \text{ si $\eta$ n'est pas ramifié} \\
    p^{-ns} \eta(p)^n G(\eta^{-1}) & \quad \text{si $\eta$ est ramifié}  \\
  \end{array}
\right. $$
le facteur epsilon associé au caractère $\eta$. Il satisfait l'équation fonctionnelle $$ \epsilon(\eta, s) \epsilon(\eta^{-1}, 1 - s) = \eta(-1). $$ On notera dans la suite $\epsilon(\eta) := \epsilon(\eta, 1/2) = p^{-n/2} \eta(p)^n G(\eta^{-1})$.

\subsubsection{Facteurs epsilon pour $\mathrm{GL}_2$} Soit $\pi$ une représentation lisse irréductible de $\mathrm{GL}_2(\qp)$ et notons $\check{\pi}$ sa contragrédiente. On note (cf. \cite[\S 6]{BH}) $\epsilon(\pi) \in L$ le facteur epsilon de la représentation $\pi$, ainsi que $\epsilon(\pi, s) = p^{- c(\pi) (s-1/2)} \epsilon(\pi), $ où $c(\pi)$ est le conducteur de $\pi$ \footnote{Le conducteur $c(\pi)$ est défini comme le plus petit entier $n$ tel que $\pi$ possède un élément fixe par les matrices de la forme $K_n = \{ {\matrice a b c d} \in \mathrm{M}_2(\zp)$, $c = d - 1 = 0$ mod $p^n \}$. On a $\dim_L \pi^{K_{c(\pi)}} = 1$. }. Observons que, si $j \in \Z$, alors $\epsilon(\pi, j + 1/2) = p^{-c(\pi) j} \epsilon(\pi) = \epsilon(\pi \otimes | \cdot |^j)$ \footnote{On note $| \cdot |^j := | \det ( \cdot ) |^j$ le caractère non-ramifié de $\mathrm{GL}_2(\qp)$ envoyant ${\matrice a b c d}$ vers $| ad - bc |^j$.}. Le facteur epsilon satisfait une équation fonctionnelle $$ \epsilon(\pi, s) \epsilon(\check{\pi}, 1 - s) = \omega_\pi(-1), $$ où $\omega_\pi$ est le caractère central de la représentation $\pi$.

\subsubsection{Une équation fonctionnelle locale} \label{eqfonctknot}

Soit $D \in \Phi \Gamma^{\text{\'et}}(\Robba)$ \'etale de dimension $2$, de Rham non triangulin à poids de Hodge-Tate $0$ et $k \geq 0$. Soit $\omega_D = \det_D \, \chi^{-1}$ de sorte que $\check{D} = D \otimes \omega_D^{-1}$. Rappelons que la correspondance de Langlands permet (cf. \cite[Prop. V.2.1]{ColmezPhiGamma}, \cite[Rem. V.14]{ColmezDosp}, \cite[\S 2.5]{epsilonKato}) de construire une involution \[ w_D \colon D^{\psi = 1} \to \check{D}^{\psi = 1}. \] Plus précisement, notre $(\varphi, \Gamma)$-module $D$ peut être vu comme un faisceau ${\matrice {\zp - \{0\}} {\zp} 0 1}$-équivariant \footnote{L'action de $\varphi$, $\sigma_a$ et la multiplication par $(1 + T)^b$, $a \in \zpe, b \in \zp$, correspondant à ${\matrice p 0 0 1}$, ${\matrice a 0 0 1}$ et ${\matrice 1 b 0 1}$ respectivement.} $U \mapsto D \boxtimes U$ sur $\zp$ dont les sections globales sont données par $D \boxtimes \zp = D$ et la construction (cf. \cite{ColmezPhiGamma}) de la représentaiton $\Pi(D)$ de $\mathrm{GL}_2(\qp)$ associée à $D$ par la correspondance de Langlands $p$-adique est fondée sur l'extension de ce faisceaux en un faisceau $G$-équivariant $U \mapsto D \boxtimes U$ sur $\PP^1 = \PP^1(\qp)$. On a un accouplement parfait et $G$-équivariant $[\; , \; ]_{\PP^1}$ sur $D \boxtimes \PP^1$ et la suite exacte fondamentale de $G$-modules suivante:
\[ 0 \to \Pi(D)^* \otimes \omega_D \to D \boxtimes \PP^1 \to \Pi(D) \to 0. \]
De plus, on a un isomorphisme
\[ \mathrm{Res}_\zp : (\Pi(D)^* \otimes \omega_D)^{{\matrice p 0 0 1} = 1} \xrightarrow{\sim} D^{\psi = 1}. \] Si $z \in D^{\psi = 1}$, on note $\tilde{z}$ l'image inverse de $z$ par cet isomorphisme. En notant $w = {\matrice 0 1 1 0}$, l'élément $w \cdot \tilde{z}$ appartient donc à $(\Pi(D)^* \otimes \omega_D)^{{\matrice p 0 0 1} = \omega_D(p)}$. On pose alors
\[ w_D(z) \defeq \mathrm{Res}_\zp(w_D(\tilde{z})) \otimes e_{\omega_D}^\vee \in \check{D}^{\psi = 1}. \]
L'équation fonctionnelle suivante est le résultat principal de \cite{epsilonKato}.

\begin{theorem} \label{eqfonct1b}
Soit $z \in D^{\psi = 1}$ et notons $\check{z} = w_D(z) \in \check{D}^{\psi = 1}$. On a \[ \exp^*(\int_\Gamma \eta \chi^{-j} \cdot \mu_{\check{z}}) \otimes \mathbf{e}^{\rm dR, \vee}_{\eta, -j, \omega_D^\vee} = C(D, \eta,  j) \cdot \exp^{-1}(\int_\Gamma \eta^{-1} \chi^{j} \cdot \mu_z) \otimes \mathbf{e}^{\rm dR, \vee}_{\eta^{-1}, j}, \] où \[ C(D, \eta, j) = - \Omega^{-1} \, \frac{\Gamma^*(-j + 1)}{\Gamma^*(j +k)}  \, \epsilon(\eta^{-1})^{-2} \epsilon(\pi \otimes \eta^{-1} \otimes | \cdot |^j) \] pour tout $j \geq 1$, où $\pi$ dénote la représentation lisse de $\mathrm{GL}_2(\qp)$ associée à (la représentation galoisienne associée à) $D$ par la correspondance de Langlands classique.
\end{theorem}

\subsubsection{Équation fonctionnelle de la fonction $L$ locale}

Le théorème \ref{eqfonct1b} nous permet de montrer une équation fonctionnelle pour la fonction $\Lambda_{D, z}$.

\begin{theorem} \label{eqfonct3}
Soit $D \in \Phi\Gamma^{\mathrm{\acute{e}t}}(\Robba)$ \'etale de dimension $2$, de Rham à poids de Hodge-Tate $0$ et $k \geq 0$. Soient $z \in D^{\psi = 1}$ et $\check{z} = w_D(z) \in \check{D}^{\psi = 1}$. Soit $\eta \colon \zpe \to L^\times$ un caractère de conducteur $p^n$ avec $n \geq m(\Delta)$. Soient $\mathfrak{U}_\Delta \subseteq \mathfrak{X}$ l'ouvert fourni par le théorème \ref{theo1} et $\kappa$ tel que $\eta \kappa \in \mathfrak{U}_\Delta$. Alors
\[ \Lambda_{\check{D}(k-1), \check{z}(k-1)}(\eta \chi^j) \otimes \mathbf{e}^{\rm dR, \vee}_{k-1, \omega_D^\vee} = -  \Omega^{-1} p^{n(k-1)} \epsilon(\eta^{-1} \otimes | \cdot |^{j - k + 1})^{-2} \epsilon(\pi \otimes \eta^{-1} \otimes | \cdot |^{j - k + 1}) \cdot \Lambda_{D, z}(\eta^{-1} \chi^{-j + k-1}). \]
\end{theorem}

\begin{proof}
Par le théorème \ref{propfin}, on sait que
\[ \Lambda_{\check{D}(k-1), \check{z}(k-1)}(\eta \chi^j) \otimes \mathbf{e}^{\rm dR, \vee}_{k-1, \omega_D^\vee} = \Gamma^*(j + 1) \, p^{n(j + 1)} \cdot \exp^*(\int_\Gamma \eta \chi^{-j + k - 1} \cdot \mu_{\check{z}}) \otimes \mathbf{e}^{\rm dR, \vee}_{\eta, -j+k-1, \omega_D^\vee}, \]
L'équation fonctionnelle du théorème \ref{eqfonct1b} nous dit que
\[ \exp^*(\int_\Gamma \eta \chi^{-j + k - 1} \cdot \mu_{\check{z}}) \otimes \mathbf{e}^{\rm dR, \vee}_{\eta, -j+k-1, \omega_D^\vee} = C(D, \eta, j - k + 1) \cdot \exp^{-1}(\int_\Gamma \eta^{-1} \chi^{j - k + 1} \cdot \mu_z) \otimes \mathbf{e}^{\rm dR, \vee}_{\eta^{-1}, j - k + 1}. \]
Finalement, en appliquant encore une fois le théorème \ref{propfin}, on obtient
\[ \exp^{-1}(\int_\Gamma \eta^{-1} \chi^{j - k + 1} \cdot \mu_z) \otimes \mathbf{e}^{\rm dR, \vee}_{\eta^{-1}, j - k + 1} = \big( \Gamma^*(-j + k) \, p^{n(-j + k)} \big)^{-1} \cdot \Lambda_{D, z}(\eta^{-1} \chi^{-j + k - 1}). \]
Ces trois équations et un petit calcul permettent de conclure.
\end{proof}

\section{Fonction $L$ $p$-adique d'une forme modulaire}

Pour terminer, on applique les résultats obtenus à la représentation associée à une forme modulaire et au système d'Euler de Kato pour obtenir une construction (partielle) de la fonction $L$ $p$-adique de la forme modulaire en question.

Soit \[ f = \sum_{n = 1}^{+\infty} a_n q^n \in \mathrm{S}_k(\Gamma_1(N), \omega_f) \otimes \C \] une forme primitive (cuspidale, propre pour les opérateurs de Hecke, nouvelle et normalisée) de poids $k \geq 2$, niveau $N$ et caractère $\omega_f \colon (\Z / N \Z)^\times \to \C^\times$. Les opérateurs de Hecke $T_n$ agissent sur $f$ par $T_n f = a_n f$. On note $F = \Q(a_1, a_2, \hdots)$ le corps de nombres engendré par les coefficients de $f$ et $\check{f} = \sum_{n = 1}^{+\infty} \overline{a}_n q^n \in \mathrm{S}_k(\Gamma_1(N), \omega_f^{-1}) \otimes \C$ la forme conjuguée à $f$. On note \[ \Lambda_\infty(f, \eta^{-1}, s) = \frac{\Gamma(s)}{(2 i \pi)^s} L(f, \eta^{-1}, s) \] la normalisation de la fonction $L$ complexe de la forme $f$. Soit $v$ une place de $F$ au-dessus de $p$ et soit $L = F_v$. Notons $V(f) \in \mathrm{Rep}_L \mathscr{G}_\Q$ la représentation galoisienne de dimension $2$ attachée à $f$ (\cite[\S 6.3]{Kato}) et $D = \D_{\rm rig}(V(f)|_{\mathscr{G}_\qp})(k - 1) \in \Phi\Gamma^{\text{ét}}(\Robba)$, qui est de Rham à poids de Hodge-Tate $0$ et $k - 1$.

En appliquant une version $p$-adique de la conjecture de Bloch-Kato, on construit (cf. \S \ref{transmutation}), pour tout $\eta \colon \zpe \to L$ (que l'on voit comme un caractère de Dirichlet de conducteur une puissance de $p$) et tout $j \geq 0$, un plongement $p$-adique naturel \[ \Lambda_\infty(f, \eta^{-1}, -j) \mapsto \iota_p(\Lambda_\infty(f, \eta^{-1}, -j)) \in \DdR(D) \] des valeurs spéciales aux entiers négatifs de la fonction $L$ complexe (normalisée) de la forme modulaire. Rappelons que, dans la bande critique $1 \leq j \leq k - 1$, les valeurs $\Lambda_\infty(f, \eta^{-1}, j)$ sont naturellement interprétés (cf. par exemple \cite[Thm. 16.2]{Kato}) $p$-adiquement en les multipliant par les périodes complexes de la forme $f$.
Le théorème final (annoncé dans l'introduction) de ce texte peut être énoncé sous la forme suivante:

\begin{theorem} \label{thmeaea}
Il existe un ouvert $\mathfrak{U}_f \subseteq \mathfrak{X}$, ne dépendant que de l'extension sur laquelle la représentation galoisienne associée à $f$ dévient semi-stable, et contenant tous les caractères d'ordre fini assez ramifiés, et une fonction rigide analytique $L_p(f) \in \mathcal{O}(\mathfrak{U}_f) \otimes \DdR(D)$ telle que, si $\eta \chi^j \in \mathfrak{U}_f$, où $\eta \colon \zpe \to L^\times$ est un caractère de conducteur $p^n$ et $j \in \Z$ est tel que $0 \leq j \leq k - 2$ ou $j \ll 0$, alors
\[L_p(f)(\eta \chi^j) =p^{n(j + 1)} G(\eta)^{-1} \cdot \iota_p(\Lambda_\infty(f, \eta^{-1}, j + 1)). \]

De plus, la fonction $L_p(f)$ satisfait une équation fonctionnelle de la forme
\[ L_p(f)(\eta \chi^{j}) =  C(f, \eta, j) \cdot L_p(\check{f})(\eta^{-1} \chi^{-j+k-2}) \otimes \mathbf{e}^{\rm dR, \vee}_{k-1, \omega_D^{-1}}, \] où \[ C(f, \eta, j) = \Omega \, p^n \, \epsilon(\eta \otimes | \cdot |^{-l + \frac{k - 1}{2}})^2 \epsilon(\pi_p(\check{f}) \otimes \eta \otimes | \cdot |^{-j + k - 1})^{-1} \cdot \prod_{\ell \mid N'} \epsilon(\pi_\ell(\check{f}) \otimes \eta^{-1} \otimes |\cdot|^{-j + \frac{k - 1}{2}})^{-1}. \]


Finalement, si $p \nmid N$ \footnote{Plus généralement, si la représentation associée à $f$ est cristabéline.}, $\alpha$ est une valeur propre du polynôme de Hecke de $p$ de $f$ et $e_\alpha \in \Dcris(V(f)) = \Dcris(V)^*$ est un vecteur propre du Frobenius cristallin de valeur propre $\alpha$, alors $\mathfrak{U}_f = \mathfrak{X}$ et on a \[ L_{p, \alpha}(f) = \langle L_p(f), e_\alpha \rangle. \]
\end{theorem}


\begin{remarque} \leavevmode
\begin{itemize}
\item Si $N = N' p^r$, $(N, N') = 1$, on devrait pouvoir trouver un lien entre l'exposant $r$ et le discriminant de l'extension $K$ sur laquelle la $L$-représentation $V_L(f)$ dévient semi-stable, c'est-à-dire, entre $r$ et le rayon de surconvergence de l'équation différentielle $p$-adique associée à $V(f)$. L'ouvert du théorème ne devrait donc dépendre que de $r$.
\item La dernière affirmation suit de la construction de Kato de la fonction $L$ $p$-adique de $f$ en utilisant le Logarithme de Perrin-Riou, et du fait que la construction menée dans ce travail est une généralisation directe de celui-ci.
\end{itemize}
\end{remarque}


\subsection{Notations et compléments} \label{repgal}

\subsubsection{Conjecture de Bloch-Kato pour les formes modulaires} Soit $Y_1(N)$ la courbe modulaire de niveau $\Gamma_1(N)$ et notons $\mathcal{KS}_{\Gamma_1(N)}^{k-2}$ la $k-2$-ième variété de Kuga-Sato de niveau $\Gamma_1(N)$ et $\epsilon$ l'idempotent usuel (cf. \cite[\S 1.1; \S 11.1]{Kato} ou \cite{Scholl}). Soit $M = M(f \otimes \eta^{-1})$ le motif associé à la forme $f \otimes \eta^{-1}$ (cf. \cite{Gealy2}, \cite{Scholl}) et considérons \[ M^*(1+j) = M(\check{f} \otimes \eta)(k+j), \] dont $V(\eta \chi^{j + 1})$ est la réalisation $p$-adique.

Notons $\mathbf{T}$ l'algèbre engendrée par les opérateurs de Hecke de niveau premier à $N$ et $\overline{\lambda} \colon \mathbf{T} \to L$ le caractère associé à $\check{f} \otimes \eta$. On a une description (cf. \cite{Scholl}, \cite{Sch-Den}, \cite{Gealy2})
\[ H^1(M^*(1+j)) = H_{\mathscr{M}}^{k}(\mathcal{KS}_{\Gamma_1(N)}^{k-2}, k + j)(\epsilon) \otimes_{\mathbf{T}} \overline{\lambda}, \]
\[ H^1_{\mathscr{D}}(M^*(1+j)) = H_{\mathscr{D}}^{k}(\mathcal{KS}_{\Gamma_1(N)}^{k - 2}, \R(k + j))(\epsilon) \otimes_{\mathbf{T}} \overline{\lambda}, \]
\[ H^1_{\text{ét}}(M^*(1+j)) = H^1(\Q, H^{k-1}_{\text{ét}}(\mathcal{KS}_{\Gamma_1(N), \overline{\Q}}^{k-2}, \qp)(k + j)(\epsilon) \otimes_{\mathbf{T}} \overline{\lambda}), \]
des groupes de cohomologie motivique, de Deligne et étale, respectivement, du motif $M^*(1 + j)$, ainsi que des régulateurs \[ r_\infty \colon H^1(M^*(1+j)) \otimes_\Q \R \to H^1_{\mathscr{D}}(M^*(1+j)), \] \[  r_{\text{ét}} \colon H^1(M^*(1+j)) \otimes_\Q \qp \to H^1_{\text{ét}}(M^*(1+j)). \]

En utilisant les symboles d'Eisenstein définis par Beilinson (cf. \cite{Beilinson2}), on construit,pour $1 \leq r \leq k - 1$, des éléments (cf. \cite[\S 3.2]{Gealy}, où les éléments sont notés $\overline{\xi}_r$) \[ \mathscr{Z}(k, j, r) \in H^{k}_{\mathscr{M}}(\mathcal{KS}_{\Gamma_1(N)}^{k-2}, k + j) (\epsilon). \] Si $\xi \in \mathrm{SL}_2(\Z)$, on pose \[ \mathscr{Z}(\check{f} \otimes \eta, j, r, \xi) = \xi^* ( \mathscr{Z}(k, j, r) ) \otimes_{\mathbf{T}} \overline{\lambda} \in H^1(M^*(1 + j)). \] Enfin, on sait (\cite[Thm. 13.6]{Kato}) que les symboles modulaires $\delta(\check{f} \otimes \eta, r, \xi) \in V_{F}(\check{f} \otimes \eta)$, $1 \leq r \leq k -1$, $\xi \in \mathrm{SL }_2(\Z)$ engendrent $V_{F}(\check{f} \otimes \eta))$ sur $F$, ce qui nous permet, en prenant des combinaisons linéaires des éléments $\mathscr{Z}(\check{f} \otimes \eta, j, r, \xi)$, de définir, pour tout $\gamma \in V_F(\check{f})$, \[ \mathscr{Z}(\check{f} \otimes \eta, j, \gamma) \in H^1(M^*(1 + j)). \] 

\begin{proposition} [{\cite[Thm. 4.1.1]{Gealy2}}] \label{GealyBK}  Soit $\gamma \in V_F(\check{f})$. Alors \[ r_\infty (\mathscr{Z}(\check{f} \otimes \eta, j, \gamma)) = L^{(N), *}(f, \eta^{-1}, -j) \cdot \gamma, \] où $L^{(N), *}(f, \eta^{-1}, -j)$ dénote le coefficient principal de la série de Taylor en $s = -j$ de la fonction $L$ de $f$ sans ses facteurs en les places divisant $N$. 
\end{proposition}

\subsubsection{Cohomologie syntomique} \label{plongement}

Les groupes de cohomologie syntomique $H^k_{\rm syn}(X_h, r)$ d'un schema séparé de type fini $X$ sur un corps $p$-adique, ainsi que des morphismes de périodes syntomiques \[ \rho_\mathrm{syn} \colon \mathbf{R} \Gamma_{\mathrm{syn}}(X_h, r) \to \mathbf{R}\Gamma_{\textrm{ét}} (X_{\text{ét}}, \qp(r)), \] et des morphismes de réalisation $p$-adiques (ou syntomiques) \[ r_p \colon H^i_{\mathscr{M}}(X, r) \to H^i_{\rm syn}(X_h, r), \] de la cohomologie motivique vers la cohomologie syntomique compatibles avec les morphismes de périodes syntomiques et les réalisations étales, ont été définis dans \cite{NN} (cf. \cite[Thm. A]{NN}).

Soient $X = \mathcal{KS}_{\Gamma_1(N)}^{k-2}$, $j \geq 0$ et $r = k + j$ (et donc $\mathrm{Fil}^r \, H^{k-1}_{\rm dR}(\mathcal{KS}_{\Gamma_1(N)}^{(k-2)}) = 0$). On a $H^k_{\rm syn}(X_h, r) \cong H^{k - 1}_{\rm dR}(X) = \DdR(H^{k - 1}_{\text{ét}}(X_{\overline{\Q}_p}, \qp))$ (cf. \cite[rem. 4.14]{NN} et le diagramme qui le précède pour la première égalité et \cite[Eq. 11.3.3]{Kato} pour la deuxième) et, en appliquant le projecteur $\epsilon$, en projetant sur la partie correspondante à la forme $\check{f} \otimes \eta$ et en tordant, on obtient des régulateurs $p$-adiques \[ r_p \colon H^1(M^*(1 + j)) \to \DdR(V(\eta \chi^{j + 1})). \] La proposition \cite[Prop. 4.13]{NN}, avec $q = k - 1$ (et $r = j + k$) se traduit alors en la relation \[ \exp \circ \; r_p = r_{\text{ét}}, \] qui nous sera très utile dans la suite.

\subsection{Plongements $p$-adiques des valeurs spéciales} \label{transmutation}

Soient $f$ comme ci-dessus et $\eta \colon \zpe \to L^\times$ d'ordre fini, vu comme un caractère de Dirichlet en fixant un isomorphisme entre $\overline{\Q}_p$ et $\C$. La proposition \ref{GealyBK} nous permet, en utilisant les régulateurs $p$-adiques, de donner un sens $p$-adique aux valeurs spéciales de la fonction $L$ complexe associée à $f$ en dehors de la bande critique.


Notons \[ D = D(\check{f})(k-1), \] qui est de Rham à poids de Hodge-Tate $0$ et $ k -1$. On a, inspirés de la proposition \ref{GealyBK}, envie de voir les éléments $r_p ( \mathscr{Z}_{\rm Kato}(\check{f} \otimes \eta, j, r, \xi) ) \in \DdR(D(\eta \chi^{j + 1}))$ comme les transmutations des valeur spéciales $L(f, \eta^{-1}, j)$ en $p$-adique. Or, comme on l'a déjà remarqué, afin de construire une fonction interpolant ces valeurs, il faut les voir tous dans un même module. Rappelons que, pour $\eta$ un caractère de Dirichlet, on a \[ \Lambda_\infty(f, \eta^{-1}, s) = \frac{\Gamma(s)}{(2 i \pi)^s} \cdot L(f, \eta^{-1}, s). \]

\begin{definition} \label{defplong} On pose \footnote{Notons que, dans la formule, on "mutiplie" et "divise" par la somme de Gauss de $\eta$, de sorte qu'elle n'a moralement aucun effet. Le terme $2 i \pi$ correspond, dans le monde $p$-adique, à l'élément $t$ de Fontaine, apparaissant dans le facteur $\mathbf{e}^{\rm dR, \vee}_{\eta, j + 1}$. Enfin, le facteur $\Gamma^*(j)$ est le coefficient principal de la série de Laurent de $\Gamma(s)$ en $s = -j$, où elle a un pôle simple.}, pour $j \geq 0$, \[ \iota_p(\Lambda_\infty(f, \eta^{-1}, -j)) = \Gamma^*(-j) \cdot G(\eta) \cdot r_p ( \mathscr{Z}(\check{f} \otimes \eta, j, r, \xi) ) \otimes \mathbf{e}^{\rm dR, \vee}_{\eta, j + 1} \in \DdR(D). \]
\end{definition}

\subsection{Interpolation} Dans cette section, on démontre que les constructions faites dans les chapitres précédents nous permettent d'interpoler les plongements $p$-adiques des valeurs spéciales de la fonction $L$ complexe de $f$ définis dans \ref{transmutation}. Ceci constitue la preuve du théorème \ref{thmeaea} annoncé au début du chapitre. On démontre d'abord, en utilisant un deuxième résultat de Gealy reliant les classes de cohomologie motiviques $\mathscr{Z}(\check{f} \otimes \eta, j, r, \xi)$ au système d'Euler de Kato et le théorème \ref{propfin}, les propriétés d'interpolation des valeurs spéciales aux entiers négatifs. On sait déjà, d'après les résultats de Kato, que les valeurs interpolées par notre fonction dans la bande critique s'interprètent bien en termes des valeurs spéciales complexes. Finalement, en utilisant une équation fonctionnelle du système d'Euler de Kato établie par Nakamura et l'équation fonctionnelle du théorème \ref{eqfonct3}, on obtiendra dans la section suivante une interprétation des valeurs spéciales aux entiers positifs $j \geq k - 1$, ce qui donne une image complète des valeurs interpolées par la fonction $L$ $p$-adique d'une forme modulaire.

\subsubsection{Relèvement motivique des éléments de Kato}
%

Rappelons que, pour chaque $\gamma \in V_L(\check{f})$, on a des éléments dans la cohomologie d'Iwasawa \[ \mathbf{z}_\gamma^{(p)}(\check{f})(k + j) =  (\mathbf{z}_{p^n}^{(p)}(\check{f}, -j, \gamma))_{n \geq 1}  \in H^1_{\rm Iw}(\Q, V(\check{f})(k + j)) \] construits par Kato (cf. \cite[Thm. 12.5]{Kato}).

\begin{proposition} [{\cite[Prop. 9.1.1]{Gealy2}}] \label{Gealy2} Soit $\gamma \in V_L(\check{f})$. Alors \[ r_{\text{ét}}(\mathscr{Z}(\check{f} \otimes \eta, j, \gamma)) = \int_\Gamma 1 \cdot \mathbf{z}_\gamma^{(p)}(\check{f} \otimes \eta)(k + j). \]
\end{proposition}

\begin{remarque} \leavevmode
\begin{itemize}
\item La proposition 9.1.1 de \cite{Gealy2} montre le résultat pour $\gamma = \delta(\check{f}, r, \mathrm{id})$. Le cas $\gamma = \delta(\check{f}, r, \xi)$ quelconque s'en déduit des compatibilités des réalisations par des correspondances algébriques et de la définition des éléments zêta, et le cas d'un élément $\gamma$ quelconque suit par linéarité.
\item Remarquons que, par construction, $\mathbf{z}_\gamma^{(p)}(\check{f} \otimes \eta)(k + j) = \mathbf{z}_\gamma^{(p)}(\check{f})(k) \otimes e_\eta \otimes e_{j} $ et la proposition ci-dessus s'exprime donc aussi comme \[ r_{\text{ét}}(\mathscr{Z}(\check{f} \otimes \eta, j, \gamma)) = \int_\Gamma \eta \chi^{j} \cdot \mathbf{z}_\gamma^{(p)}(\check{f})(k). \] 
\end{itemize}
\end{remarque}

\subsubsection{Interpolation aux entiers négatifs}

Notons \[ D(f) = \D_{\rm rig}(V_L(f) \mid_{\mathscr{G}_\qp}),  \;\;\; D(\check{f})  = \D_{\rm rig}(V_L(\check{f}) \mid_{\mathscr{G}_\qp}) \in \Phi\Gamma^{\text{ét}}(\Robba) \] les $(\varphi, \Gamma)$-modules associés aux formes $f$ et $\check{f}$. Rappelons que l'on a posé $D = D(\check{f})(k - 1)$, qui est de Rham à poids de Hodge-Tate $0$ et $k - 1$, et notons \[ \mathbf{z}_{\rm Kato} = \mathrm{Exp}^*(\mathbf{z}_\gamma^{(p)}(\check{f})(k - 1)) \in \Drig(V(\check{f})(k - 1)))^{\psi = 1} = D^{\psi = 1}. \]

\begin{lemme} \label{eaea2} Soit $j \geq 0$. On a \[ \Lambda_{D, \mathbf{z}_{\rm Kato}}(\eta \chi^{-j - 1}) = p^{-nj} \; G(\eta)^{-1} \cdot \iota_p(\Lambda_\infty(f, \eta^{-1}, -j)). \]
\end{lemme}

\begin{proof}

En utilisant le théorème \ref{Gealy2} et la remarque qui le suit, on obtient
\[ \exp^{-1} (r_{\text{ét}} ( \mathscr{Z}(\check{f} \otimes \eta, j, \gamma))) = \exp^{-1}(\int_\Gamma \eta \chi^{j + 1} \cdot \mu_{\mathbf{z}_{\rm Kato}}). \]

d'où, par la compatibilité entre le régulateur $p$-adique et le régulateur étale, on en déduit
\begin{eqnarray*}
\iota_p(\Lambda_\infty(f, \eta^{-j}, -j)) &=& \Gamma^*(-j) \; G(\eta) \cdot r_p ( \mathscr{Z}(\check{f} \otimes \eta, j, \gamma) ) \otimes \mathbf{e}^{\rm dR, \vee}_{\eta, j + 1} \\
&=& \Gamma^*(-j) \; G(\eta) \cdot \exp^{-1}(\int_\Gamma \eta \chi^{j + 1} \cdot \mathbf{z}_{\rm Kato}) \otimes \mathbf{e}^{\rm dR, \vee}_{\eta, j + 1}
\end{eqnarray*}

Par ailleurs, le théorème \ref{propfin} affirme que \[ \Lambda_{D, \mathbf{z}_{\rm Kato}}(\eta \chi^{-j - 1}) = p^{-nj} \; \Gamma^*(-j) \cdot \exp^{-1}(\int_\Gamma \eta \chi^{j + 1} \cdot \mu_{\mathbf{z}_{\rm Kato}}) \otimes \mathbf{e}^{\rm dR, \vee}_{\eta, j + 1}, \] d'où le résultat.
\end{proof}

\subsubsection{La bande critique}

Si $0 \leq j \leq k -1$, on pose \[ \iota_p(\Lambda_\infty(f, \eta^{-1}, j + 1)) = \Gamma^*(j + 1) \; G(\eta) \cdot \exp^*(\int_\Gamma \eta \chi^{-j} \cdot \mu_{\mathbf{z}_{\rm Kato}}) \otimes \mathbf{e}^{\rm dR, \vee}_{\eta, -j}. \] On sait, d'après \cite[Thm. 12.5]{Kato}, que les images de ces valeurs par l'application de périodes sont reliées aux valeurs spéciales de la fonction $L$ de $f$. Le lemme suivant est immédiat

\begin{lemme} Soit $0 \leq j \leq k-1$. Alors \[ \Lambda_{D, \mathbf{z}_{\rm Kato}}(\eta \chi^j) = p^{n(j + 1)} \; G(\eta)^{-1} \cdot \iota_p( \Lambda_\infty(f, \eta^{-1}, j + 1)). \]
\end{lemme}

\subsection{Équation fonctionnelle et valeurs aux entiers positifs}

Pour l'interprétation $p$-adique des valeurs $L(f, \eta, j)$, $j \geq k$, on fera appel à l'équation fonctionnelle de la fonction $L$ locale, qui s'avéra fortement ressemblante à l'équation fonctionnelle complexe.

\subsubsection{L'équation fonctionnelle complexe}

Notons $\mathbf{A}_\Q$ le groupe des adèles de $\Q$. Soit $\eta \colon \zpe \to L^\times$ un caractère d'ordre fini. On regarde $\eta$ comme un caractère de Dirichlet (via $L^\times \subseteq \overline{\Q}_p^\times \cong \C^\times$) ainsi que comme un caractère de Hecke de la façon usuelle \footnote{Si $\eta \colon \zpe \to L^\times$ est de conducteur $p^n$, il est vu comme un caractère des idèles en utilisant la décomposition $\mathbf{A}_\Q = \Q^\times \times \R^{> 0} \times \widehat{\Z}^\times$. Le caractère de $\qpe$ induit par $\eta$ est $\eta$ (avec $\eta(p) = 1$) et, si $\ell \neq p$, celui de $\Q_\ell^\times$ est l'unique caractère non-ramifié caractère prenant la valeur $\eta^{-1}(\overline{\ell})$ en $\ell$, où $\overline{\ell} \in \zpe$ est n'importe quel relèvement de la classe de $\ell$ modulo $p^n$.}. La forme $f \otimes \eta$ est supercuspidale et on note \[ \pi(f \otimes \eta) = {\bigotimes}_\ell' \pi_\ell(f \otimes \eta) \] la représentation automorphe de $\mathrm{GL}_2(\mathbf{A}_\Q)$ associée à $f \otimes \eta$. On a $\pi(f \otimes \eta) = \pi(f) \otimes \eta = {\bigotimes}'_\ell \pi_\ell(f) \otimes \eta$. Notons $\epsilon(\pi(f) \otimes \eta, s)$ le facteur epsilon global de la représentation $\pi(f \otimes \eta)$ défini par \footnote{Le produit étant fini car $\pi_\ell(f \otimes \eta)$ est non ramifiée en presque toute place} \[ \epsilon(\pi(f)\otimes \eta, s) = \epsilon(\pi_\infty(f) \otimes \eta, s) \cdot \prod_\ell \epsilon(\pi_\ell(f) \otimes \eta, s), \] où $\epsilon( \pi_\ell(f) \otimes \eta, s)$ est le facteur epsilon de la représentation $\pi_\ell(f) \otimes \eta$ de $\mathrm{GL}_2(\Q_l)$, comme décrit dans \S \ref{replisses}, et $\epsilon(\pi_\infty(f) \otimes \eta, s) = i^k$.

La fonction $L$ complexe satisfait l'équation fonctionnelle \footnote{Le décalage en $\frac{k-1}{2}$ provient du fait que les facteurs locaux des représentations de $\mathrm{GL}_2$ sont normalisés de sorte que le centre de symétrie de l'équation fonctionnelle des fonctions $L$ soit situé en $1/2$, tandis que celui des fonctions $L$ automorphes l'est en $\frac{k}{2}$.} \[ \frac{\Gamma(s)}{(2 \pi)^s} \cdot L(f, \eta^{-1}, s) = \epsilon( \pi(f) \otimes \eta^{-1}, s - \frac{k-1}{2}) \cdot \frac{\Gamma(k - s)}{(2 \pi)^{k - s}} \cdot L(\check{f}, \eta, k - s). \] Si $j \geq k$ est un entier, on peut écrire l'équation fonctionnelle sous la forme \[ \frac{\Gamma(j)}{(2 i \pi)^j} \cdot L(f, \eta^{-1}, j) = i^k (-1)^{-j} \epsilon( \pi(f) \otimes \eta^{-1} \otimes |\cdot|^{j - \frac{k-1}{2}}) \cdot \frac{\Gamma(k - j)}{(2 i \pi)^{k - j}} \cdot L(\check{f}, \eta, k - j), \] ou bien \[ \Lambda_\infty(f, \eta^{-1}, j) = i^k (-1)^{-j} \epsilon( \pi(f) \otimes \eta^{-1} \otimes |\cdot|^{j - \frac{k-1}{2}}) \cdot \Lambda_\infty(\check{f}, \eta, k - j).\] En décomposant le facteur epsilon, on peut réécrire l'équation fonctionnelle sous la forme suivante:
\[ \Lambda_\infty(f, \eta^{-1}, j) = (-1)^{k-j} \epsilon(\pi_p(f) \otimes \eta^{-1} \otimes |\cdot|^{j - \frac{k-1}{2}}) \cdot \prod_{\ell \mid N' }\epsilon( \pi_\ell(f) \otimes \eta^{-1} \otimes |\cdot|^{j - \frac{k-1}{2}}) \cdot \Lambda_\infty(\check{f}, \eta, k - j).\] Remarquons pour finir que, si l'on écrit $N = N' p^r$, alors, pour $p \neq \ell \mid N$, on a l'égalité
\[ \epsilon( \pi_\ell(f) \otimes \eta^{-1} \otimes |\cdot|^{j - \frac{k-1}{2}}) = \epsilon(\pi_\ell(f)) \, \eta(\ell)^{c(\pi_\ell(f))} \, \ell^{-c(\pi_\ell(f))(j - \frac{k-1}{2})}, \] 
et, en utilisant le fait que le conducteur de $\pi_\ell(f)$ est $\ell^{v_\ell(N)}$, on en déduit
\[ \prod_{\ell \mid N'} \epsilon( \pi_\ell(f) \otimes \eta^{-1} \otimes |\cdot|^{j - \frac{k-1}{2}}) = \eta(N') (N')^{\frac{k-1}{2} - j} \, \prod_{\ell \mid N'} \epsilon(\pi_\ell(f)). \]

%

\subsubsection{L'équation fonctionnelle du système d'Euler de Kato, d'après Nakamura}

Écrivons $N = N'p^r$, $(N',N)=1$. Si $\gamma \in V(\check{f})$, notons $\check{\gamma} \in V(f)$ l'élément dual à $\gamma \otimes e_k$ sous l'accouplement parfait $V(\check{f})(k) \times V(f) \to L$ donné par la dualité de Poincaré.

On note
\[ \mathbf{z}_{\rm Kato} = \mathrm{Exp}^*(\mathbf{z}_\gamma^{(p)}(\check{f})(k-1)) \in D(\check{f})(k-1)^{\psi = 1} = D^{\psi = 1}, \] 
\[ \check{\mathbf{z}}_{\rm Kato} = \mathrm{Exp}^*(\mathbf{z}_{\check{\gamma}}^{(p)}(f)(k-1)) \in D(f)(k-1)^{\psi = 1} = \check{D}(k-2)^{\psi = 1}, \] 
\[ \mathbf{z}_{\rm Kato}^* = w_D(\mathbf{z}_{\rm Kato}) \in \check{D}^{\psi = 1} \] les tordus des systèmes d'Euler associés aux formes $\check{f}$ et $f$ et l'image par l'involution de $\mathbf{z}_{\rm Kato}$ (cf. \S \ref{eqfonctknot}), respectivement.

\begin{proposition} [{\cite[Conj. 4.4; Thm. 4.6; Prop. 4.7]{Nakamura2}}] \label{eqfonctNakamura} Notons $[N'] = \prod_{\ell \mid N'} [\sigma_\ell]^{v_\ell(N')} \in \Lambda$. Alors 
\[ \mathbf{z}_{\rm Kato}^* = - \prod_{\ell \mid N'} \epsilon(\pi_\ell(\check{f}) \otimes |\cdot|^{\frac{k - 1}{2}})^{-1} \cdot ([N'] \cdot (\check{\mathbf{z}}_{\rm Kato} \otimes e_{2-k})). \]
\end{proposition}

\begin{remarque} \leavevmode
\begin{itemize}
\item \cite[Conj. 4.4]{Nakamura2} est énoncé en termes de facteurs epsilon des représentations de Weil-Deligne. Sa démonstration (dans le cas de Rham non triangulin) est basée sur la compatibilité locale-globale dans la correspondance de Langlands $p$-adique et \cite[Prop. 3.14]{Nakamura2}. Cette dernière proposition peut être énoncée naturellement (cf. thm. \ref{eqfonct1b}) en termes de facteurs locaux des représentations de $\mathrm{GL}_2(\qp)$. La preuve de la proposition ne ferait donc pas usage de la compatibilité locale-globale et elle resterait donc purement locale.

\item Il faut faire un peu d'attention car les normalisations des facteurs locaux dans ce travail ne coïncident pas avec celles de \cite{Nakamura2}. Comme on l'a remarqué, dans le texte présent, les facteurs locaux des représentations lisses sont normalisés de sorte que l'équation fonctionnelle de la fonction $L$ soit centrée en $s = 1/2$, tandis que, dans \cite{Nakamura2}, elle est centrée en $s = \frac{k}{2}$. La différence entre les facteurs locaux est donc un twist par $| \cdot |^{\frac{k - 1}{2}}$.

\end{itemize}
\end{remarque}

\subsubsection{L'équation fonctionnelle de la fonction $L$ $p$-adique} L'équation fonctionnelle du système d'Euler de Kato et l'équation fonctionnelle \ref{eqfonct3} nous permettent d'interpréter les valeurs aux entiers positifs de la fonction $\Lambda_{D, \mathbf{z}_{\rm Kato}}$.

\begin{theorem} \label{eqfonct4}
Soit $j > 0$ un entier. Alors
\[ \Lambda_{D, \mathbf{z}_{\rm Kato}}(\eta \chi^{j}) =  C(f, \eta, j) \cdot \Lambda_{\check{D}(k-2), \check{\mathbf{z}}_{\rm Kato}}(\eta^{-1} \chi^{-j+k-2}) \otimes \mathbf{e}^{\rm dR, \vee}_{k-1, \omega_D^{-1}}, \]
où
\[ C(f, \eta, j) = \Omega \, p^n \, \epsilon(\eta \otimes | \cdot |^{-j + \frac{k - 1}{2}})^2 \epsilon(\pi_p(\check{f}) \otimes \eta \otimes | \cdot |^{-j + k - 1})^{-1} \cdot \prod_{\ell \mid N'} \epsilon(\pi_\ell(\check{f}) \otimes \eta^{-1} \otimes |\cdot|^{-j + \frac{k - 1}{2}})^{-1}. \]
\end{theorem}

\begin{proof}
En appliquant le théorème \ref{eqfonct3} (avec $k - 1$ au lieu de $k$, $\eta^{-1}$ au lieu de $\eta$ et $-j + k - 2$ au lieu de $j$), on obtient
\begin{equation} \label{tratra1} \Lambda_{D, \mathbf{z}_{\rm Kato}}(\eta \chi^j) = C_1 \cdot \Lambda_{\check{D}(k-2), \mathbf{z}_{\rm Kato}^*(k - 2)}(\eta^{-1} \chi^{-j+k-2}) \otimes \mathbf{e}^{\rm dR, \vee}_{k-2, \omega_D^\vee},
\end{equation}
où \[ C_1 = - \Omega \, p^{-n(k - 2)} \epsilon(\eta \otimes | \cdot |^{-j})^{2} \epsilon(\pi \otimes \eta \otimes | \cdot |^{-j})^{-1}. \] Observons que
\[ p^{-n(k - 2)} \epsilon(\eta \otimes | \cdot |^{-j})^{2} = p^n (p^{-n (\frac{k - 1}{2})} \epsilon(\eta \otimes | \cdot |^{-j}))^2 = p^n \epsilon(\eta \otimes | \cdot |^{-l + \frac{k - 1}{2}})^2. \]
Comme $\pi = \pi(D) = \pi_p(\check{f}) \otimes | \cdot |^{k - 1}$, on en déduit
\[ C_1 = - \Omega \, p^n \, \epsilon(\eta \otimes | \cdot |^{-l + \frac{k - 1}{2}})^2 \epsilon(\pi_p(\check{f}) \otimes \eta \otimes | \cdot |^{-j + k - 1})^{-1}. \]


D'après le théorème \ref{propfin}, on a
\[ \Lambda_{\check{D}(k-2), \mathbf{z}_{\rm Kato}^*(k - 2)}(\eta^{-1} \chi^{-j+k-2}) = \Gamma^*(-j + k - 1) p^{n(-j + k - 1)} \cdot \log(\int_\Gamma \eta^{-1} \chi^{j - k + 2} \cdot \mu_{\mathbf{z}_{\rm Kato}^*(k - 2)}) \otimes \mathbf{e}^{\rm dR, \vee}_{\eta^{-1}, j - k + 2}.  \] On a
\begin{eqnarray*}
\int_\Gamma \eta^{-1} \chi^{j - k + 2} \cdot \mu_{\mathbf{z}_{\rm Kato}^*(k - 2)} &=& \int_\Gamma \eta^{-1} \chi^j \cdot \mu_{\mathbf{z}_{\rm Kato}^*} \\
&=&  - \prod_{\ell \mid N'} \epsilon(\pi_\ell(\check{f}) \otimes |\cdot|^{\frac{k - 1}{2}})^{-1} \cdot \int_\Gamma \eta^{-1} \chi^j \cdot \mu_{[N']( \check{\mathbf{z}}_{\rm Kato}(2 - k))} \\
&=& - \prod_{\ell \mid N'} \epsilon(\pi_\ell(\check{f}) \otimes |\cdot|^{\frac{k - 1}{2}})^{-1} \, \eta(N') (N')^{-j} \cdot \int_\Gamma \eta^{-1} \chi^{j-k+2} \cdot \mu_{\check{\mathbf{z}}_{\rm Kato}},
\end{eqnarray*}
où on a utilisé la proposition \ref{eqfonctNakamura} dans la deuxième égalité, et la définition de l'action de $\Lambda = \zp[[\zpe]]$ sur $D^{\psi = 1}$ dans la troisième. En utilisant l'égalité
\[ \prod_{\ell \mid N'} \epsilon(\pi_\ell(\check{f}) \otimes |\cdot|^{\frac{k - 1}{2}}) \, \eta^{-1}(N') (N')^j = \prod_{\ell \mid N'} \epsilon(\pi_\ell(\check{f}) \otimes \eta^{-1} \otimes |\cdot|^{-j + \frac{k - 1}{2}}), \] on en déduit
\begin{equation} \label{tratra2}
\Lambda_{\check{D}(k-2), \mathbf{z}_{\rm Kato}^*(k - 2)}(\eta^{-1} \chi^{-j+k-2}) = - \prod_{\ell \mid N'} \epsilon(\pi_\ell(\check{f}) \otimes \eta^{-1} \otimes |\cdot|^{-j + \frac{k - 1}{2}})^{-1} \cdot \Lambda_{\check{D}(k-2), \check{\mathbf{z}}_{\rm Kato}}(\eta^{-1} \chi^{-j+k-2}). 
\end{equation}

En rassemblant les formules \eqref{tratra1} et \eqref{tratra2}, on déduit le résultat.
\end{proof}

\begin{remarque}
Notons que, pour $j > k - 2$, les valeurs du côté droite de la formule du théorème s'interprètent en termes des valeurs spéciales de la fonction $L$ complexe de $\check{f}$. En utilisant l'équation fonctionnelle complexe on peut traduire ceci et donner une formule d'interpolation de la fonction $L$ $p$-adique en termes de valeurs spéciales complexes en tout entier $j \in \Z$.
\end{remarque}

\newpage
\bibliographystyle{acm}
\bibliography{bibliography}

\def\Dbar{\leavevmode\lower.6ex\hbox to 0pt{\hskip-.23ex \accent"16\hss}D}
  \def\cfac#1{\ifmmode\setbox7\hbox{$\accent"5E#1$}\else
  \setbox7\hbox{\accent"5E#1}\penalty 10000\relax\fi\raise 1\ht7
  \hbox{\lower1.15ex\hbox to 1\wd7{\hss\accent"13\hss}}\penalty 10000
  \hskip-1\wd7\penalty 10000\box7}
  \def\cftil#1{\ifmmode\setbox7\hbox{$\accent"5E#1$}\else
  \setbox7\hbox{\accent"5E#1}\penalty 10000\relax\fi\raise 1\ht7
  \hbox{\lower1.15ex\hbox to 1\wd7{\hss\accent"7E\hss}}\penalty 10000
  \hskip-1\wd7\penalty 10000\box7}
\begin{thebibliography}{10}

\bibitem{AmiceVelu}
{\sc Amice, Y. et~V{\'e}lu, J.}
\newblock Distributions {$p$}-adiques associ\'ees aux s\'eries de {H}ecke.
\newblock {\em Ast\'erisque}, 24--25 (1975), 119--131.

\bibitem{Beilinson2}
{\sc Be{\u\i}linson, A.~A.}
\newblock Higher regulators of modular curves.
\newblock In {\em Applications of algebraic {$K$}-theory to algebraic geometry
  and number theory,}, vol.~55 of {\em Contemp. Math.} pp.~1--34.

\bibitem{Bellaiche}
{\sc Bella\"\i~che, J.}
\newblock Critical {$p$}-adic {$L$}-functions.
\newblock {\em Invent. Math. 189}, 1 (2012), 1--60.

\bibitem{Berger02}
{\sc Berger, L.}
\newblock Repr\'esentations {$p$}-adiques et \'equations diff\'erentielles.
\newblock {\em Invent. Math. 148\/} (2002), 219--284.

\bibitem{Berger03}
{\sc Berger, L.}
\newblock Bloch and {K}ato's exponential map: three explicit formulas.
\newblock {\em Doc. Math. Extra Vol.\/} (2003), 99--129.

\bibitem{Berger08}
{\sc Berger, L.}
\newblock \'{E}quations diff\'erentielles {$p$}-adiques et {$(\phi,N)$}-modules
  filtr\'es.
\newblock {\em Ast\'erisque}, 319 (2008), 13--38.

\bibitem{BergerBreuil}
{\sc Berger, L., and Breuil, C.}
\newblock Sur quelques repr\'esentations potentiellement cristallines de {${\rm
  GL}_2(\bold Q_p)$}.
\newblock {\em Ast\'erisque}, 330 (2010), 155--211.

\bibitem{Besser}
{\sc Besser, A.}
\newblock Syntomic regulators and {$p$}-adic integration. {I}. {R}igid syntomic
  regulators.
\newblock In {\em Proceedings of the {C}onference on {$p$}-adic {A}spects of
  the {T}heory of {A}utomorphic {R}epresentations\/} (2000), vol.~120,
  pp.~291--334.

\bibitem{BH}
{\sc Bushnell, C. J. et~Henniart, G.}
\newblock {\em The local {L}anglands conjecture for {$\rm GL(2)$}}.
\newblock Springer-Verlag, 2006.

\bibitem{CC}
{\sc Cherbonnier, F. et~Colmez, P.}
\newblock Repr\'esentations {$p$}-adiques surconvergentes.
\newblock {\em Invent. Math. 133\/} (1998), 485--571.

\bibitem{ColmezIw2}
{\sc Cherbonnier, F. et~Colmez, P.}
\newblock Th\'eorie d'{I}wasawa des repr\'esentations {$p$}-adiques d'un corps
  local.
\newblock {\em J. Amer. Math. Soc. 12\/} (1999), 241--268.

\bibitem{ColmezDosp}
{\sc Colmez, P. et~Dospinescu, G.}
\newblock Compl\'et\'es universels de repr\'esentations de {${\rm GL}_2(\bold
  Q_p)$}.
\newblock {\em Algebra and Number Theory 8\/} (2014), 1447--1519.

\bibitem{ColmezIw1}
{\sc Colmez, P.}
\newblock Th\'eorie d'{I}wasawa des repr\'esentations de de {R}ham d'un corps
  local.
\newblock {\em Ann. of Math. 148\/} (1998), 485--571.

\bibitem{ColmezBourbaki}
{\sc Colmez, P.}
\newblock La conjecture de {B}irch et {S}winnerton-{D}yer {$p$}-adique.
\newblock {\em Ast\'erisque}, 294 (2004), 251--319.

\bibitem{ColmezEV}
{\sc Colmez, P.}
\newblock Espaces vectoriels de dimension finie et repr\'esentations de de
  {R}ham.
\newblock {\em Ast\'erisque}, 319 (2008), 117--186.

\bibitem{ColmezPhiGamma}
{\sc Colmez, P.}
\newblock Repr\'esentations de {${\rm GL}_2(\bold Q_p)$} et
  {$(\varphi,\Gamma)$}-modules.
\newblock {\em Ast\'erisque}, 330 (2010), 281--509.

\bibitem{CN}
{\sc Colmez, P. et~Niziol, W.}
\newblock Syntomic complexes and p-adic nearby cycles.
\newblock {\em preprint\/} (2015).

\bibitem{Delbourgo}
{\sc Delbourgo, D.}
\newblock On the {$p$}-adic {B}irch, {S}winnerton-{D}yer conjecture for
  non-semistable reduction.
\newblock {\em J. Number Theory 95\/} (2002), 38--71.

\bibitem{Sch-Den}
{\sc Deninger, C. et~Scholl, A.~J.}
\newblock The {B}e\u\i linson conjectures.
\newblock In {\em {$L$}-functions and arithmetic}, vol.~153 of {\em London
  Math. Soc. Lecture Note Ser.} Cambridge Univ. Press, Cambridge, 1991,
  pp.~173--209.

\bibitem{Fontaine90}
{\sc Fontaine, {\relax J.-M}.}
\newblock Repr\'esentations {$p$}-adiques des corps locaux. {I}.
\newblock In {\em The {G}rothendieck {F}estschrift, {V}ol.\ {II}}, vol.~87 of
  {\em Progr. Math.} 1990, pp.~249--309.

\bibitem{Gealy}
{\sc Gealy, M.~T.}
\newblock Special values of $p$-adic $l$ functions associated to modular forms.
\newblock {\em Non publi\'e\/} (2003).

\bibitem{Gealy2}
{\sc Gealy, M.~T.}
\newblock {\em On the {T}amagawa number conjecture for motives attached to
  modular forms}.
\newblock ProQuest LLC, 2006.
\newblock Th\`ese (Ph. D).

\bibitem{Gros}
{\sc Gros, M.}
\newblock R\'egulateurs syntomiques et valeurs de fonctions {$L$}
  {$p$}-adiques. {I}.
\newblock {\em Invent. Math. 99}, 2 (1990), 293--320.

\bibitem{Herr}
{\sc Herr, L.}
\newblock Une approche nouvelle de la dualit\'e locale de {T}ate.
\newblock {\em Math. Ann. 320}, 2 (2001), 307--337.

\bibitem{Kato}
{\sc Kato, K.}
\newblock {$p$}-adic {H}odge theory and values of zeta functions of modular
  forms.
\newblock {\em Ast\'erisque}, 295 (2004), 117--290.

\bibitem{Kedlaya}
{\sc Kedlaya, K.}
\newblock A {$p$}-adic monodromy theorem.
\newblock {\em Ann. of Math. 160\/} (2004), 93--184.

\bibitem{KPX}
{\sc Kedlaya, {\relax K. S., Pottharst, J.} et~Xiao, L.}
\newblock Cohomology of arithmetic families of {$(\varphi,\Gamma)$}-modules.
\newblock {\em J. Amer. Math. Soc. 27\/} (2014), 1043--1115.

\bibitem{Liu}
{\sc Liu, R.}
\newblock Cohomology and duality for {$(\phi,\Gamma)$}-modules over the {R}obba
  ring.
\newblock {\em Int. Math. Res.}, 3 (2008).

\bibitem{L-W}
{\sc Loeffler, D., and Weinstein, J.}
\newblock On the computation of local components of a newform.
\newblock {\em Math. Comp. 81\/} (2012), 1179--1200.

\bibitem{Manin}
{\sc Manin, J.~I.}
\newblock Periods of cusp forms, and {$p$}-adic {H}ecke series.
\newblock {\em Mat. Sb. 92\/} (1973), 378--401.

\bibitem{MTT}
{\sc Mazur, {\relax B., Tate, J.} et~Teitelbaum, J.}
\newblock On {$p$}-adic analogues of the conjectures of {B}irch and
  {S}winnerton-{D}yer.
\newblock {\em Invent. Math. 84\/} (1986), 1--48.

\bibitem{Nakamura}
{\sc Nakamura, K.}
\newblock Iwasawa theory of de {R}ham {$(\varphi,\Gamma)$}-modules over the
  {R}obba ring.
\newblock {\em J. Inst. Math. Jussieu 13\/} (2014), 65--118.

\bibitem{Nakamura2}
{\sc Nakamura, K.}
\newblock Local {$\epsilon$}-isomorphisms for rank two {$p$}-adic
  representations of {${\rm Gal}(\overline{ \bold Q}_p / \bold \Q_p)$} and a
  functional equation of {K}ato's {E}uler system.
\newblock {\em Cambridge Journal of Mathematics 5\/} (2017), 65--118.

\bibitem{NN}
{\sc Nekovar, J. et~Niziol, W.}
\newblock Syntomic cohomology and regulators for varieties over p-adic fields.
\newblock {\em Algebra and Number Theory 10}, 8 (2016), 1695--1790.

\bibitem{Niziol}
{\sc Niziol, W.}
\newblock On the image of {$p$}-adic regulators.
\newblock {\em Invent. Math. 127\/} (1997), 375--400.

\bibitem{PerrinRiou1}
{\sc Perrin-Riou, B.}
\newblock Th\'eorie d'{I}wasawa des repr\'esentations {$p$}-adiques sur un
  corps local.
\newblock {\em Invent. Math. 115}, 1 (1994), 81--161.

\bibitem{PerrinRiou2}
{\sc Perrin-Riou, B.}
\newblock Fonctions {$L$} {$p$}-adiques des repr\'esentations {$p$}-adiques.
\newblock {\em Ast\'erisque}, 229 (1995), 198.

\bibitem{P-S}
{\sc Pollack, R., and Stevens, G.}
\newblock Critical slope {$p$}-adic {$L$}-functions.
\newblock {\em J. Lond. Math. Soc.) 87}, 2 (2013), 428--452.

\bibitem{Pottharst}
{\sc Pottharst, J.}
\newblock Cyclotomic {I}wasawa theory of motives.
\newblock {\em Preprint, disponible \`a
  \url{https://vbrt.org/writings/cyc.pdf}\/} (2012).

\bibitem{epsilonKato}
{\sc Rodrigues~Jacinto, J.}
\newblock La conjecture {$\epsilon$} locale de {K}ato en dimension $2$.
\newblock {\em Mathematische Annalen (\`a para\^itre)\/} (2018).

\bibitem{Schmidt}
{\sc Schmidt, R.}
\newblock Some remarks on local newforms for {$\rm GL(2)$}.
\newblock {\em J. Ramanujan Math. Soc. 17}, 2 (2002), 115--147.

\bibitem{S-T}
{\sc Schneider, P. et~Teitelbaum, J.}
\newblock Algebras of {$p$}-adic distributions and admissible representations.
\newblock {\em Invent. Math. 153\/} (2003), 145--196.

\bibitem{Scholl}
{\sc Scholl, A.~J.}
\newblock Motives for modular forms.
\newblock {\em Invent. Math. 100}, 2 (1990), 419--430.

\bibitem{Visik}
{\sc Vi{\v{s}}ik, M.~M.}
\newblock Nonarchimedean measures associated with {D}irichlet series.
\newblock {\em Mat. Sb. 99\/} (1976), 248--260.

\end{thebibliography}

\end{document}